\documentclass[11pt]{article}

\usepackage{amsmath,graphicx,amscd,textcomp}
\usepackage{tikz,authblk}
\usepackage{subfig}
\usepackage{color}

\usepackage{amsfonts,amssymb,amsthm,enumerate,tikz,authblk,subfig}
\usepackage[colorlinks=false, urlcolor=blue,bookmarks=true,bookmarksopen=true, citecolor=blue]{hyperref}

\newcommand{\CC}{\mbox{\bbb C}}

\newcommand{\RR}{\mbox{\bbb R}}
\newcommand{\ZZ}{\mbox{\bbb Z}}

\font\bb=msbm7 scaled\magstep1

\numberwithin{equation}{section}

\newtheorem{theorem}{Theorem}[section]
\newtheorem{lemma}[theorem]{Lemma}

\newtheorem{cor}[theorem]{Corollary}

\theoremstyle{definition}
\newtheorem{defn}[theorem]{Definition}
\newtheorem{eg}[theorem]{Example}

\newtheorem{prop}[theorem]{Proposition}
\newtheorem{remark}[theorem]{Remark}

\newtheorem{qn}[theorem]{Question}

\def \begineq{\begin{equation}}
\def \endeq{\end{equation}}

\def \bb{\mathbb}

\def \CC{{\bb{C}}}
\def \CP{{\bb{CP}}}

\def \RR{{\bb{R}}}
\def \TT{{\bb{T}}}

\def \ZZ{{\bb{Z}}}

\def \({\left(}
\def \){\right)}
\def \<{\langle}
\def \>{\rangle}
\def \bar{\overline}

\def \qed{\hfill $\square$ \vspace{0.03in}}

\begin{document}

\author[1] {Basudeb Datta}
\author[2] {Soumen Sarkar}

\affil[1] {Department of Mathematics, Indian Institute of Science, Bangalore 560\,012,
India. dattab@math.iisc.ernet.in}

\affil[2] {Department of Mathematics and Statistics, University of Regina,
3737 Wascana Parkway, Regina, Canada. sarkar5s@uregina.ca,
Soumen.Sarkar@uregina.ca}

\title{Equilibrium triangulations of some quasitoric 4-manifolds}


\date{}

\maketitle

\vspace{-12mm}

\begin{center}

\date{July 23, 2015}

\end{center}

\noindent {\bf Abstract:} Quasitoric manifolds, introduced by M. Davis and T. Januskiewicz in 1991, are topological
generalizations of smooth complex projective spaces. In 1992, Banchoff and K\"uhnel constructed a 10-vertex
equilibrium triangulations of $\CP^2$. We generalize this construction for quasitoric manifolds and construct
some equilibrium triangulations of $4$-dimensional quasitoric manifolds. In some cases,
our constructions give vertex minimal equilibrium triangulations.

\section{\sc Introduction}

Quasitoric manifolds were introduced by Davis and Januskiewicz in their pioneering paper \cite{DJ}. These are topological
generalizations of nonsingular complex projective spaces. In toric topology, action of the compact torus $\TT^n$
on smooth oriented manifolds with nice properties produce bridges between topology and combinatorics. The moment
map for Hamiltonian action of compact torus on symplectic manifolds is such an example, see \cite{Au}. The
complex projective space $\CP^n$ with standard $\TT^n$ action is a quasitoric manifold over an $n$-simplex. Many
topological properties of quasitoric manifolds have been studied last $20$ years.

Triangulations of polytopes and manifolds have been studying for many years \cite{Le, Ri, Ca, Alt, Lu, BD, Da}.
There is a $7$-vertex unique vertex minimal triangulation of the torus $\bb{T}^2$. Starting with this triangulation
Banchoff and K\"uhnel \cite{BK} constructed a $10$-vertex equilibrium triangulation of $\bb{CP}^2$. The goal of
this article is to introduce the notion of equilibrium triangulation for quasitoric manifolds and to study
combinatorial properties, especially equilibrium triangulations, of $4$-dimensional
quasitoric manifolds which include nonsingular projective surfaces. At this point we do not know if every
quasitoric manifold has an equilibrium triangulation.

The article is organized as follows. In Section \ref{preli}, we recall the basic definitions of quasitoric
manifolds, triangulations of manifolds and lens spaces. With the minimal $7$-vertex triangulation of ${\TT}^2$,
we have constructed countably many triangulations of $S^3$, $\mathbb{RP}^3$, $L(3, 1)$, $L(4, 1)$, $L(5, 2)$ and
$L(7, 2)$ in Section \ref{tortri}. Number of vertices of this triangulations are given in the same section.
In Section \ref{qmtri}, we introduce the definition of equilibrium triangulation for quasitoric manifolds.
We present equilibrium triangulations of some nonsingular projective surfaces and $4$-dimensional quasitoric
manifolds in Sections \ref{m4}, \ref{m5} and \ref{m6}. In some cases we found vertex minimal equilibrium triangulations.

\section{\sc Preliminaries}\label{preli}

\subsection{\sc Quasitoric manifolds and nonsingular projective spaces}\label{defegm}
In this subsection we recall the definition of quasitoric manifolds, the equivariant classification of quasitoric
manifolds and a relation between quasitoric manifolds and nonsingular projective spaces.

An $n$-dimensional {\em simple} polytope in $\RR^n$ is a convex polytope where exactly $n$ bounding hyperplanes
meet at each vertex. So, an $n$-polytope $P$ in $\RR^n$ is simple if and only if each vertex of $P$ is in exactly
$n$ facets (codimension one faces).  Therefore, a codimension $k$ face of $P$ is the intersection of exactly $k$
facets.  Clearly every $2$-polytope is simple.

Let $\TT^n=\{(z_1, \ldots, z_n) \in \CC^n : |z_1| = \cdots = |z_n| =1\}$. A smooth action of $\TT^n$ on a
$2n$-dimensional smooth manifold $M$ is said to be {\em locally standard} if every point $y \in M $ has a
$\TT^n$-stable open neighborhood $U_y$ and a diffeomorphism $\psi : U_y \to V$, where $V$ is a $\TT^n$-stable
open subset of $\CC^n$, and an isomorphism $\delta_y : \TT^n \to \TT^n$ such that $\psi (t\cdot x) = \delta_y (t)
\cdot \psi(x)$ for all $(t,x) \in \TT^n \times U_y$.

A $2n$-dimensional {\em quasitoric manifold} is a triple $(M, P, \mathfrak{q})$, where $M$ is a $2n$-dimensional
closed smooth manifold with an effective $\TT^n$-action, $P$ is a simple $n$-polytope and $\mathfrak{q}: M \to P$
is a projection map satisfying the following conditions: $(i)$ the $\TT^n$-action is locally standard, $(ii)$ the
map $\mathfrak{q}: M \to P$ is constant on $\TT^n$ orbits and maps every $\ell$-dimensional orbit to a point in
the interior of an $\ell$-dimension face of $P$. We also say that $\mathfrak{q} : M \to P$ is a quasitoric
manifold over $P$ (or simply $M$ is a quasitoric manifold over $P$). All complex projective spaces $\CP^{n}$ and
their equivariant connected sums, products are quasitoric manifolds, (cf. \cite{DJ}).

\begin{prop}[{\cite[Lemma 1.4]{DJ}}] \label{clema}
Let $\mathfrak{q} : M \to P$  be a $2n$-dimensional quasitoric manifold. Then, there is a projection map $f: \TT^n
\times P \to M$ so that, $f$ maps $\TT^n \times q$ onto $ \mathfrak{q}^{-1}(q)$ for each $q \in P$.
\end{prop}

Now we briefly discuss the combinatorial definition of quasitoric manifolds following \cite{DJ}.
Let $\mathcal{F}(P)$ be the set of all facets of an $n$-dimensional simple polytope $P$. For a simple $n$-polytope
$P$, a function 
$$
\xi : \mathcal{F}(P) \to \ZZ^n
$$ 
is called a {\em characteristic function} if the submodule
generated by $\{\xi(F_{j_1}), \ldots, \xi(F_{j _{\ell}})\}$ is an $\ell$-dimensional direct summand of $\ZZ^n$
whenever the intersection $F_{j_1}\cap \cdots \cap F_{j _{\ell}}$ is nonempty. The vectors $\xi(F_{1}), \ldots,
\xi(F_m)$ are called {\em characteristic vectors} and the pair $(P, \xi)$ is called a {\em characteristic pair}.

In \cite{DJ}, the authors show that given a quasitoric manifold one can associate a characteristic pair to it up
to choice of signs of characteristic vectors. They also constructed a quasitoric manifold from the pair $(P,
\xi)$. It's description is the following.

Let $P$ be a simple $n$-polytope and $(P, \xi)$ be a characteristic pair. A codimension
$k$ face $F$ of $P$ is the intersection $F_{j_1} \cap \cdots \cap F_{j_k}$ 
of unique collection of $k$ facets $F_{j_1}, \ldots, F_{j_k}$ of $P$. Let $\ZZ(F)$ be the submodule of
$\ZZ^n$ generated by the characteristic vectors $\xi(F_{j_1}), \ldots, \xi(F_{j_k})$. The module $\ZZ(F)$ is a
direct summand of $\ZZ^n$. Therefore, the torus $\TT_F : = (\ZZ(F) \otimes_{\ZZ} \RR)/\ZZ(F)$ is a direct summand
of $\TT^n$. Define $\ZZ(P)=\{0\}$ and $\TT_P$ to be the proper trivial subgroup of $\TT^n$. If $p \in P$, $p$
belongs to relative interior of a unique face $ F$ of $P$.

Define an equivalence relation `$\sim$' on the product $\TT^n \times P$ by
\begin{equation}\label{eqsou01}
(t,p) \sim (s,q)~\mbox{if} ~p=q~\mbox{and} ~ s^{-1}t \in \TT_F,
\end{equation}
where $p \in F^{\circ}$. Let 
$$
M(P, \xi) := (\TT^n \times P) / \sim 
$$ 
be the quotient space. The group operation
in $\TT^n$ induces a natural action of $\TT^n$ on $M(P, \xi)$, namely $g. [h, x] = [gh, x]$. The projection onto
the second factor of $\TT^n \times P$ descends to the quotient map
\begin{equation}\label{equ002}
\pi : M(P, \xi) \to P ~\mbox{given by} ~ \pi([t,p]) = p.
\end{equation}
So, the orbit space of this $\TT^n$ action on $M(P, \xi)$ is the polytope $P$ itself. One can show that the space
$M(P, \xi)$ has the structure of a quasitoric manifold (cf. \cite{DJ}).

\begin{remark}[{\cite[Corollary 3.9]{DJ}}]
{\rm A quasitoric manifold $\mathfrak{q} : M \to P$ is simply connected. So, $M$ is orientable. A choice of
orientation on $\TT^n $ and $ P$ gives an orientation on $M$.}
\end{remark}

Let $\delta : \TT^n \to \TT^n$ be an automorphism. Two quasitoric manifolds $M_1$ and $M_2$ over the same
polytope $P$ are called {\em $\delta$-equivariantly} $homeomorphic$ if there is a homeomorphism $f : M_1 \to M_2$
such that $f(t \cdot x) = \delta(t)\cdot f(x)$ for all $(t, x ) \in \TT^n \times M_1$. Moreover, if $\delta$ is
the identity, $f$ is called an {\em equivariant homeomorphism} and the $\TT^n$ actions on $M_1$ and $M_2$ are
called {\em equivalent}.

\begin{prop}[{\cite[Proposition 1.8]{DJ}}]\label{clema2}
Let $\mathfrak{q}:M \to P$ be a $2n$-dimensional quasitoric manifold over $P$ and $\xi:\mathcal{F}(P) \to \ZZ^n$
be its associated characteristic function. Let $\mathfrak{q}_{M}: M(P, \xi) \to P$ be the quasitoric manifold
constructed (as in equation $\ref{equ002}$) from the pair $(P, \xi)$. Then, the map $f: \TT^n \times P \to M$ of
Proposition $\ref{clema}$ descends to an equivariant homeomorphism $ M(P, \xi) \to M$ covering the identity on
$P$.
\end{prop}

Two characteristic pairs $(P, \xi)$ and $(Q, \eta)$ are called {\em equivalent} if there is a diffeomorphism
$\psi: P \to Q$ (as manifold with corners) such that $\eta(\psi(F))=\pm \xi(F)$ for all $F \in \mathcal{F}(P)$.
Note that $\psi(\mathcal{F}(P)) = \mathcal{F}(Q)$. Using Proposition \ref{clema2} we can prove the following.
\begin{prop}\label{clema3}
 Let $M$ and $N$ be $2n$-dimensional quasitoric manifolds over $P$ and $Q$ with the characteristic
 functions $\xi$ and $\eta$ respectively. Then, $M$ and $N$ are equivariantly homeomorphic if and only if 
 $(P, \xi)$ and $(Q, \eta)$ are equivalent.
\end{prop}

\begin{eg}[Example 1.18, \cite{DJ}]\label{triangle} \rm{
Let $Q$ be a triangle in $\RR^2$. The possible characteristic functions are indicated by the following figures.

\begin{figure}[ht]
\centerline{ \scalebox{0.63}{\input{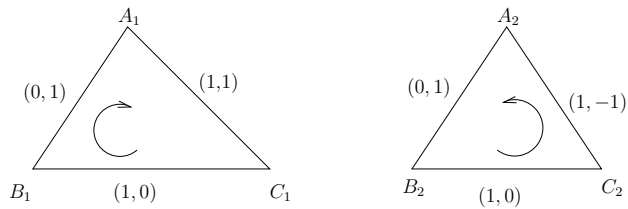} } }
\caption {The characteristic functions corresponding to a triangle.}
\label{egch001}
\end{figure}

The quasitoric manifold corresponding to the first characteristic pair is $\CP^2$ with the usual ${\TT}^2$ action
and standard orientation, we denote it by $\CP^2$. The
second characteristic pair correspond to the same ${\TT}^2$ action with the reverse orientation on $\CP^2$, we
denote this quasitoric manifold by $\bar{\CP^2}$.}
\end{eg}

Note that there are many non-equivariant ${\TT}^2$ actions on $\CP^2$, see Section 6 of \cite{Sar}.

\begin{eg}[Example 1.19, \cite{DJ}]\label{square}
\rm{
Suppose that $P$ is combinatorially a rectangle in $\RR^2$. In this case there are many possible characteristic
functions. Some examples are given by Figure \ref{egch002}.
\begin{figure}[ht]
\centerline{ \scalebox{0.64}{\input{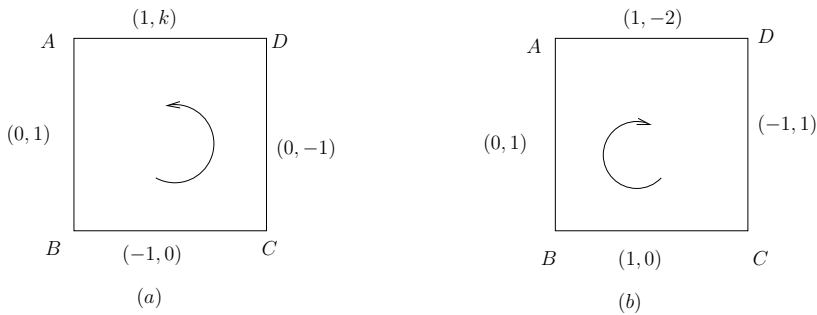} } }
\caption {Some characteristic functions corresponding to a rectangle.}
\label{egch002}
\end{figure}

The first characteristic pairs may construct an infinite family of $4$-dimensional quasitoric manifolds, denote
them by $M_k$ for each $k \in \ZZ$. By Proposition \ref{clema3}, the manifolds $\{M_k : k \in \ZZ\}$ are
equivariantly distinct. Let $L(k)$ be the complex line bundle over $\CP^1$ with the first Chern class $k$.
The complex manifold $\CP(L(k) \oplus \CC)$ is the Hirzebruch surface for each integer $k$, where $\CP(\cdot)$
denotes the projectivisation of a complex bundle. So, each Hirzebruch surface is the total space of the
bundle $\CP(L(k) \oplus \CC) \to \CP^1$ with fiber $\CP^1$. It is well-known that with the natural action
of ${\TT}^2$ on $\CP(L(k) \oplus \CC)$ it is equivariantly homeomorphic to $M_k$ for each $k$ (cf. \cite{Oda}).
That is, with respect to the ${\TT}^2$-action, Hirzebruch surfaces are quasitoric manifolds where the orbit
space is a combinatorial square and the corresponding characteristic maps are as in Fig. \ref{egch002} $(a)$.

On the other hand, the second combinatorial model (Figure \ref{egch002}$(b)$) gives the quasitoric manifold
$\CP^2 ~\# ~\CP^2$, the equivariant connected sum of $\CP^2$ (cf. \cite{OR}).}
\end{eg}

Equivariant connected sum of quasitoric manifolds is discussed in 1.11 of \cite{DJ} as well
as in Section 2 of \cite{Sar}. The following Proposition classifies all $4$-dimensional quasitoric manifolds.

\begin{prop} [{\cite[p. 553]{OR}}]\label{cla4}
Any $4$-dimensional quasitoric manifold $M^4$ over a $2$-polytope is an equivariant connected sum of some copies
of $\CP^2$, $\bar{\CP^2}$ and $M_k$ for some $k\in \ZZ$. In particular, $M^4$ is homeomorphic to a connected
sum of copies of $\mathbb{CP}^2$, $\bar{\mathbb{CP}^2}$ and $S^2 \times S^2$.
\end{prop}

\begin{remark}
{\rm The decomposition in the above proposition of $M^4$ as a connected sum is not unique, see Remark 5.8 in
\cite{OR} for details.}
\end{remark}

Now, we recall a relation between toric varieties and quasitoric manifolds. For more details
about toric variety see Oda \cite{Oda} and Fulton \cite{Fu}.

A {\em toric variety} is a normal algebraic variety $M$ containing the algebraic torus $(\CC^{\ast})^n$
as a Zariski open subset in such a way that the natural action of $(\CC^{\ast})^n$ on itself extends to an action on $M$.
Thus $(\CC^{\ast})^n$ acts on $M$ with a dense orbit. 

A {\em convex polyhedral cone} (or simply {\em cone}) $\sigma$ spanned by $p_1, \ldots, p_k \in \RR^n$ is
$$
\sigma=\{r_1p_1 + \cdots + r_kp_k \in \RR^n : r_1, \ldots, r_k \geq 0 \}.
$$
A cone in $\RR^n$ is called nonsingular if it is generated by a part of a basis of $\ZZ^n$.
So a nonsingular cone does not contain a line through the origin in $\RR^n$. A {\em nonsingular fan} $\Sigma$
is a set of nonsingular cones in $\RR^n$ such that a face of a cone in $\Sigma$ is a cone in $\Sigma$ and the intersection
of two cones in $\Sigma$ is a face of each. A nonsingular fan is called {\em complete} if the union of all cones
in $\Sigma$ is $\RR^n$.

There is a one-to-one bridge from fans to toric varieties of complex dimension $n$ which produces a beautiful
connection between topology and combinatorics. We denote the corresponding toric variety of the fan $\Sigma$
by $M_{\Sigma}$.

Suppose $(P, \xi)$ is a characteristic pair and $F$ be a nonempty face. If $F=P$ then, we define $\sigma_P=\{0\}$
the origin of $\RR^n$. Otherwise $F=F_{i_1} \cap \ldots \cap F_{i_k}$ for a unique collection of facets
$F_{i_1}, \ldots, F_{i_k}$ of $P$. Then, we define $\sigma_F$ is the cone generated by $\xi(F_{i_1}), \ldots,
\xi(F_{i_k})$. So, by definition of characteristic function, $\sigma_F$ is a nonsingular cone for a nonempty
proper face $F$ of $P$. In general the collection $\Sigma_P=\{\sigma_F : F ~\mbox{is a face of}~ P\}$ is not a fan.
We say the characteristic pair is {\em complete} if $\Sigma_P$ is a complete fan.

On the other hand, if $M$ is a nonsingular projective toric variety then it is a quasitoric manifold and the 
corresponding characteristic pair is complete (see
Section 5.1.3 and 5.2.3 of \cite{BP}). From the discussion of these Sections we state the following proposition.

\begin{prop}\label{corf}
 Let $M$ be a quasitoric manifold manifold with characteristic pair $(P, \xi)$. Then, $M$ is projective
 toric variety if and only if $(P, \xi)$ is complete.
\end{prop}

\subsection{Triangulation}

We recall some basic definitions and examples in triangulation. All simplicial complexes considered here are
finite and the empty set is a simplex (of dimension $-1$) of every simplicial complex. If $\sigma = \{v_1, \dots,
v_k\}$ is a simplex in a simplicial complex then, $\sigma$ is also denoted by $v_1v_2\cdots v_k$. The vertex set
of a simplicial complex $X$ is denoted by $V(X)$.

If $X$, $Y$ are two simplicial complexes, then a {\em simplicial isomorphism} from $X$ to $Y$ is a bijection $\pi
: V(X) \to V(Y)$ such that for $\sigma\subseteq V(X)$, $\sigma$ is a face of $X$ if and only if $\pi (\sigma)$ is
a face of $Y$. Two simplicial complexes $X$, $Y$ are called {\em isomorphic} (denoted by $X \cong Y$) when such
an isomorphism exists.

We identify two complexes if they are isomorphic. We also identify a simplicial complex with the set of maximal
faces in it. For any two simplicial complexes $X$ and $Y$, their {\em join} $X \ast Y$ is the simplicial complex
whose faces are the disjoint unions of the faces of $X$ with the faces of $Y$.

A subcomplex $Z$ of a simplicial complex $X$ is said to be an {\em induced} subcomplex if every face of $X$
contained in $V(Z)$ is a face of $Z$. If $\sigma$ is a face of a simplicial complex $X$ then, the {\em link} of
$\sigma$ in $X$, denoted by ${\rm lk}_X(\sigma)$, is the simplicial complex whose faces are the faces $\tau$ of
$X$ such that $\tau \cap \sigma = \emptyset$ and $\sigma\cup\tau \in X$.

A $d$-dimensional simplicial complex is called {\em pure} if all its maximal faces (called {\em facets}) are
$d$-dimensional. A $d$-dimensional pure simplicial complex is called a {\em weak pseudomanifold} if each of its
$(d - 1)$-faces is in at most two facets.

For a $d$-dimensional weak pseudomanifold $X$, the {\em boundary} $\partial X$ of $X$ is the pure $(d-
1)$-dimensional subcomplex of $X$ whose facets are those $(d-1)$-dimensional faces of $X$ which are contained in
unique facets of $X$. For a simplicial complex $X$, $|X|$ denotes the {\em geometric carrier} of $X$. It may be
described as the subspace of $[0, 1]^{V(X)}$ consisting of all functions $f \colon V(X) \to [0, 1]$ satisfying
(i) ${\rm Support}(f) \in X$ and (ii) $\sum_{x\in V(X)} f(x) = 1$.

If a topological space $M$ is homeomorphic to $|X|$ then, we say that $X$ {\em triangulates} $M$. If $|X|$ is a
topological $d$-ball (resp., $d$-sphere) then,  $X$ is called a {\em triangulated $d$-ball} (resp., {\em
$d$-sphere}). If $|X|$ is a pl manifold (with the pl structure induced by $X$) (resp., topological manifold) then
$X$ is called a {\em combinatorial manifold} (resp., triangulated manifold).

For $1\leq d\leq 4$, $X$ is a combinatorial $d$-manifold if and only if the vertex links are triangulated
$(d-1)$-spheres or $(d -1)$-balls if and only if $X$ is a triangulated $d$-manifold.

For a finite set $\alpha$, let $\overline{\alpha}$ (respectively $\partial \alpha$) denote the simplicial complex
whose faces are all the subsets (respectively, all proper subsets) of $\alpha$. Thus, if $\#(\alpha)= n \geq 2$,
$\overline{\alpha}$ is a copy of the standard triangulation $B^{\,n - 1}_n$ of the $(n- 1)$-dimensional ball, and
$\partial \alpha$ is a copy of the standard triangulation $S^{\,n-2}_n$ of the $(n-2)$-dimensional sphere (and is
also denoted by $S^{\hspace{.2mm}n-2}_n(\alpha)$). Thus, for any two disjoint finite sets $\alpha$ and $\beta$,
$\overline{\alpha}\ast\partial\beta$ and $\partial\alpha \ast \overline{\beta}$ are two triangulations of a ball;
they have identical boundaries, namely $(\partial {\alpha}) \ast (\partial {\beta})$.

By a {\em subdivision} of a simplicial complex  $X$ we mean a simplicial complex $X^{\hspace{.1mm}\prime}$
together with a homeomorphism from $|X^{\hspace{.1mm}\prime}|$ onto $|X|$ which is facewise linear. Two complexes
$X$, $Y$ have isomorphic subdivisions if and only if $|X|$ and $|Y|$ are piecewise linear homeomorphic.

If $\alpha$ is a face of a simplicial complex $X$ and $u\not\in V(X)$ then, consider the simplicial complex (on
the vertex set $V(X) \cup \{u\}$)
\begin{equation}
 \tau^u_{\alpha}(X) := \{\sigma \in X : \alpha \not\subseteq \sigma\} \cup
\overline{u}\ast \partial\alpha \ast {\rm lk}_X(\alpha).
\end{equation}
Then,  $\tau^u_{\alpha}(X)$ is a subdivision of $X$ and
is called the subdivision obtained from $X$ by {\em starring the new vertex $u$ in $\alpha$}. A simplicial
complex $Y$ is called a {\em stellar} subdivision of $X$ if $Y$ is obtained from $X$ by starring (successively)
finitely many vertices. Two simplicial complexes $Y$ and $Z$ are called {\em stellar equivalent} if they have
isomorphic stellar subdivisions. So, if two simplicial complexes are stellar equivalent then they triangulate the
same topological space.

If $X$ is a $d$-dimensional simplicial complex with an induced subcomplex $\overline{\alpha} \ast \partial \beta$
($\alpha \neq \emptyset$, $\beta \neq \emptyset$) of dimension $d$ (thus, $\dim(\alpha)+ \dim(\beta) = d$), then
\begin{equation}
 \kappa_{\alpha}^{\beta}(X) := (X \setminus (\overline{\alpha}\ast \partial \beta)) \cup (\partial\alpha \ast
\overline{\beta})
\end{equation}
 is clearly another triangulation of the same topological space $|X|$. In this case,
$\kappa_{\alpha}^{\beta}(X)$ is said to be obtained from $X$ by the {\em bistellar move} $\alpha \mapsto \beta$.
If $\dim(\beta) = i$ ($0\leq i \leq d$), we say that $\alpha \mapsto \beta$ is a {\em bistellar move of index
$i$} (or an {\em $i$-move}, in short). Clearly, if $\kappa_{\alpha}^{\beta}(X)$ is obtained from $X$ by an
$i$-move $\alpha \mapsto \beta$ then $X$ is obtained from $\kappa_{\alpha}^{\beta}(X)$ by the (reverse) $(d-
i)$-move $\beta \mapsto \alpha$, i.e., $\kappa_{\beta}^{\alpha}(\kappa_{\alpha}^{\beta}(X)) =X$. Notice that, in
case $i=0$, i.e., when $\beta$ is a single vertex, we have $\partial \beta = \{\emptyset\}$ and hence
$\overline{\alpha} \ast \partial \beta = \overline{\alpha}$. Therefore, our requirement that $\overline{\alpha}
\ast \partial \beta$ is the induced subcomplex of $X$ on $\alpha \sqcup \beta$ means that $\beta$ is a new
vertex, not in $X$. Thus, a $0$-move creates a new vertex, and correspondingly a $d$-move deletes an old vertex.

For a simplicial complex $X$, if $\sim$ is an equivalence relation on the vertex set $V(X)$ then consider the
simplicial complex (called the {\em quotient complex}) $X/\hspace{-1.5mm}\sim$ whose vertices are the
$\sim$-classes and $\{C_1, \dots, C_k\}$ is a face of $X/\hspace{-1.5mm}\sim$ if $C_1\cup \cdots \cup C_k$ is a
face of $X$.

If $\phi= u_1\cdots u_{i-1}u_iu_{i+1}\cdots u_k$ is a loop in a simplicial complex $X$ and $u_{i-1}u_iu_{i+1}$ is a
face of $X$ then, $\phi$ is equivalent to the loop $\varrho= u_1 \cdots u_{i-1}u_{i +1}\cdots u_k$ (denoted by
$\phi \simeq \varrho$) and hence $\phi$ and $\varrho$ represent the same element in the fundamental group $\pi_1(X, u_1)$.

Let $X_1$, $X_2$ be two triangulated closed $d$-manifold with disjoint vertex-sets, $\sigma_i$ a facet of $X_i$
($i=1, 2$) and $\psi \colon \sigma_1 \to \sigma_2$ a bijection. Let $(X_1\sqcup X_2)^{\psi}$ denote the
simplicial complex obtained from $X_1 \sqcup X_2 \setminus \{\sigma_1, \sigma_2\}$ by identifying $x$ with
$\psi(x)$ for each $x\in \sigma_1$. Then, $(X_1\sqcup X_2)^{\psi}$ is called an {\em elementary connected sum} of
$X_1$  and $X_2$, and is denoted by $X_1 \#_{\psi} X_2$ (or simply by $X_1\# X_2$). Note that the combinatorial
type of $X_1 \#_{\psi} X_2$ depends on the choice of the bijection $\psi$. However, $|X_1 \#_{\psi} X_2|$ is the
topological connected sum of $|X_1|$ and $|X_2|$ (taken with appropriate orientations). Thus, $X_1 \#_{\psi} X_2$
is a triangulated closed $d$-manifold (\cite[Definition 1.3]{BD}).

In \cite{KB}, K\"uhnel and Banchoff constructed a 9-vertex vertex minimal triangulation of $\CP^2$ and in
\cite{Lu}, Lutz has constructed an 11-vertex vertex minimal triangulation of $S^2 \times S^2$. Using these we get
the following proposition.

\begin{prop} \label{sigma123}
For $j+k+l\geq 1$, the manifold $j \mathbb{CP}^2 \# k \bar{\mathbb{CP}^2} \# \ell(S^2 \times S^2)$ has a
triangulation with $4j + 4k +6\ell +5$ vertices.
\end{prop}

\begin{proof}
Let $X$ be a 9-vertex triangulation of $\CP^2$. Since $\bar{\CP^2}$ is homeomorphic to $\CP^2$, $\bar{\CP^2}$ has
a 9-vertex triangulation $Y$. Let $Z$ be an 11-vertex triangulation of $S^2 \times S^2$.

If $M$ and $N$ are triangulated closed $d$-manifolds with $m$ and $n$ vertices respectively then,  the elementary
connected sum $M \# N$ of $M$ and $N$ is an $(m+n-d-1)$-vertex connected triangulated closed $d$-manifold and
triangulates $|M|\# |N|$. Let $\Sigma_{jk\ell}$ be the elementary connected sum of $j$ copies of $X$, $k$ copies
of $Y$ and $\ell$ copies of $Z$. Then, $\Sigma_{jk\ell}$ is a triangulated closed 4-manifold (and hence a
combinatorial 4-manifold) with $9j+9k+11\ell - (j+k+\ell-1)\times 5$ vertices and triangulates $j \mathbb{CP}^2
\# k \bar{\mathbb{CP}^2} \# \ell(S^2 \times S^2)$. (If one of $j, k, \ell$ is 0 then, we do not take the
corresponding triangulated manifold in the connected sum.) This proves the result.
\end{proof}

\subsection{Three dimensional Lens spaces}

The $3$-dimensional lens spaces were introduced by Tietze, \cite{Tie}. Let $p$ and $q$ be relatively prime
integer. Consider $S^3= \{(z_1, z_2) \in \CC^2: |z_1|^2 + |z_2|^2 =1\}$. Then, the action of $\ZZ_p= \ZZ/ p\ZZ$ on
$S^3$ generated by $e^{2\pi i/p}. (z_1, z_2) = (e^{2\pi i/p}z_1, e^{2\pi iq/p}z_2)$ is free since $p$ and $q$ are
relatively prime. The quotient space of this action is called the {\rm lens space} and is denoted by $L(p, q)$.
The lens space $L(p, q)$ is homeomorphic to $L(p, r)$ if and only if $r = \pm q$ (mod $p$) or $qr = \pm 1$ (mod
$p$) (\cite{Man}).

Let $m, n \in \ZZ$ so that ${\rm det}((m,n)^t, (p,q)^t) = 1$. Let $S^1 \times D^2$ be a solid torus considered as
a subset of $\CC^2$, that is $S^1 \times D^2 = \{(z, w) \in \CC^2 : |z| = 1, |w| \leq 1\}$. We will also think
${\TT}^2 = \partial(S^1 \times D^2) = S^1 \times \partial D^2 = S^1 \times S^1$ as a subset of $\CC^2$. Let
$\varphi_{(p,q)}: {\TT}^2 \to {\TT}^2$ be the map $\varphi_{(p,q)}(z, w)= (z^m w^p, z^n w^q)$. Clearly this is a
homeomorphism of ${\TT}^2$. It is well known that the lens space $L(p, q)$ is the identification space of the
disjoint union $S^1 \times D^2 \sqcup S^1 \times D^2$ with $(z,w) \sim \varphi_{(p,q)}(z,w)$ for all $(z,w) \in
{\TT}^2$.
It is known that 
\begin{align}
 L(p,q) & = L(-p, -q)= \bar{L(-p,q)}= \bar{L(p, -q)} = \bar{L(p, p-q)}, 
\end{align}
where $\bar{L(k, \ell)}$ denotes $L(k, \ell)$ with the opposite orientation.

\begin{qn}
{\rm Can we construct a triangulation of ${\TT}^2$ such that $\varphi_{(1, 0)}, \varphi_{(0,1)}$ and
$\varphi_{(p,1)}$ are simplicial?}
\end{qn}

\begin{qn}
{\rm Let $A_1, \ldots, A_k \in GL(n, \ZZ)$ $n\geq 2, k\geq 1$. What are the vertex minimal triangulations
of $\TT^n$ so that the homomorphisms determined by $A_1, \ldots, A_k$ are simplicial? }
\end{qn}

\section{Some sequences of triangulated solid tori and lens spaces}\label{tortri}

In this section we present some triangulated $S^1 \times D^2$, $S^3$, $S^1 \times S^2$, $\mathbb{RP}^3$,
$L(3, 1)$, $L(4, 1)$, $L(5,2)$ and $L(7, 2)$. We need these in Sections \ref{m4}, \ref{m5} and \ref{m6}.
\begin{eg}\label{tst1}
{\rm We know that any triangulation of the 2-dimensional torus $S^1 \times S^1$ requires at least 7 vertices and up to
isomorphism there is a unique triangulation $\mathcal{T}$ of $S^1 \times S^1$, where
\begin{eqnarray} \label{ET}
\mathcal{T}= \{013, 124, 235, 346, 045, 156, 026, 023, 134, 245, 356, 046, 015, 126\}.
\end{eqnarray}

\setlength{\unitlength}{3mm}

\begin{picture}(44,16.5)(0,2.5)

\thicklines

\put(2,10){\line(1,0){21}}

\put(2,10){\line(0,1){4}}
\put(2,14){\line(1,0){21}} \put(23,10){\line(0,1){4}}

\put(30,5){\line(1,0){12}} \put(30,5){\line(0,1){12}}
\put(30,17){\line(1,0){12}} \put(42,5){\line(0,1){12}}

\put(30,5){\line(1,1){12}}


\thinlines

\put(5,10){\line(0,1){4}} \put(8,10){\line(0,1){4}}
\put(11,10){\line(0,1){4}} \put(14,10){\line(0,1){4}}
\put(17,10){\line(0,1){4}} \put(20,10){\line(0,1){4}}

\put(2,10){\line(3,4){3}} \put(5,10){\line(3,4){3}}
\put(8,10){\line(3,4){3}} \put(11,10){\line(3,4){3}}
\put(14,10){\line(3,4){3}} \put(17,10){\line(3,4){3}}
\put(20,10){\line(3,4){3}}

\put(34,5){\line(0,1){4}} \put(38,13){\line(0,1){4}}
\put(30,9){\line(1,0){4}} \put(38,13){\line(1,0){4}}

\put(42,9){\line(-1,1){4}} \put(30,13){\line(1,1){4}}
\put(34,5){\line(1,2){4}} \put(34,5){\line(2,1){8}}
\put(34,9){\line(-1,1){4}} \put(34,9){\line(1,2){4}}
\put(30,13){\line(2,1){8}} \put(38,5){\line(1,1){4}}

\put(30,5){\vector(1,0){3}} \put(30,5){\vector(0,1){2.5}}
\put(30,5){\vector(0,1){3}} \put(30,5){\vector(1,1){2}}
\put(30,5){\vector(1,1){2.25}} \put(30,5){\vector(1,1){2.5}}


\put(30,3.7){\mbox{0}}  \put(34,3.7){\mbox{1}}
\put(38,3.7){\mbox{6}}  \put(42,3.7){\mbox{0}}
\put(29,8.7){\mbox{2}} \put(29,12.7){\mbox{5}}
\put(42.5,8.7){\mbox{2}} \put(42.5,12.7){\mbox{5}}
\put(30,17.3){\mbox{0}}  \put(34,17.3){\mbox{1}}
\put(38,17.3){\mbox{6}}  \put(42,17.3){\mbox{0}}
\put(34.5,8){\mbox{3}} \put(37,13){\mbox{4}}

\put(2,8.8){\mbox{0}} \put(5,8.8){\mbox{1}} \put(8,8.8){\mbox{2}}
\put(11,8.8){\mbox{3}} \put(14,8.8){\mbox{4}}
\put(17,8.8){\mbox{5}} \put(20,8.8){\mbox{6}}
\put(22.5,8.8){\mbox{0}}

\put(2,14.2){\mbox{2}} \put(5,14.2){\mbox{3}}
\put(8,14.2){\mbox{4}} \put(11,14.2){\mbox{5}}
\put(14,14.2){\mbox{6}} \put(17,14.2){\mbox{0}}
\put(20,14.2){\mbox{1}} \put(22.5,14.2){\mbox{2}}

\put(8,7){\mbox{$(a)$}} \put(27,6){\mbox{$(b)$}}

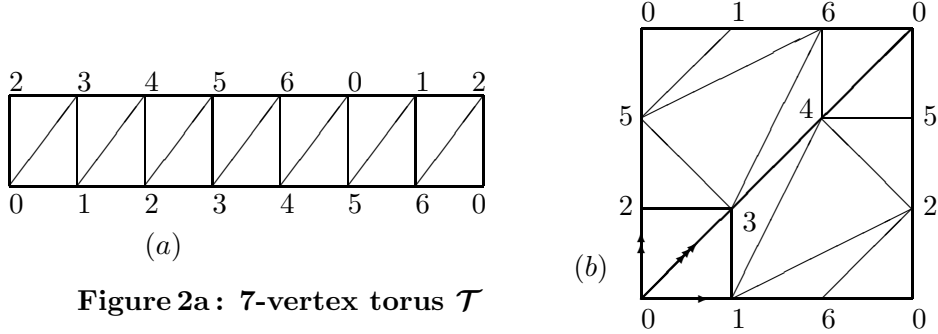
\put(5,4.5){\mbox{\bf Figure\,2a\,: 7-vertex torus \boldmath{$\mathcal{T}$}}}
\end{picture}

Here the vertex set of $\mathcal{T}$ is $\{0, 1, \dots, 6\}$. If we identify the vertex set with $\ZZ_7=
\ZZ/7\ZZ$ then, $\ZZ_7$-acts on $\mathcal{T}$ given by $i\mapsto i+1$. There are exactly three 7-vertex
triangulations of the solid torus $D^{\hspace{.2mm}2} \times S^1$ whose boundaries are $\mathcal{T}$. These are
the following (cf. \cite{BK})\,:
\begin{eqnarray} \label{ET123}
T_1 & := & \{0123, 1234, 2345, 3456, 0456, 0156, 0126\},  \nonumber\\
T_2 & := & \{0246, 1246, 1346, 1356, 0135, 0235, 0245\}, \nonumber\\
T_3 & := & \{0145, 1245, 1256, 2356, 0236, 0346, 0134 \}.
\end{eqnarray}
Observe that $T_1\cap T_2 = T_1\cap T_3 = T_2\cap T_3 = \mathcal{T} = \partial T_1 = \partial T_2 = \partial
T_3$. The triangulated $2$-manifold $\mathcal{T}$ has fourteen $2$-dimensional faces. Among the remaining
$\binom{7}{3} - 14 =21$ possible $2$-simplices, $7$ are in $T_1$, $7$ are in $T_2$ and $7$ are in $T_3$. For
$1\leq i \neq j\leq 3$, let 
$$
S_{ij} := T_i\cup T_j.
$$ 
Then, for any vertex $k\in \ZZ_7$, ${\rm lk}_{S_{ij}}(k) =
{\rm lk}_{T_i}(k) \cup {\rm lk}_{T_j}(k)$. Since ${\rm lk}_{T_{i}}(k)$ and ${\rm lk}_{T_{j}}(k)$ are $2$-discs
and ${\rm lk}_{T_{i}}(k)\cap {\rm lk}_{T_{j}}(k) = {\rm lk}_{T}(k)$ is a cycle, it follows that ${\rm
lk}_{S_{ij}}(k)$ is a $2$-sphere. Thus, $S_{ij}$ is a triangulated $3$-manifold without boundary. Since $S_{ij}$
has $7$ vertices, it triangulates the $3$-sphere $S^{\hspace{.3mm}3}$ (since a triangulated $3$-manifold without
boundary with at most $8$ vertices triangulates $S^{\hspace{.3mm}3}$ \cite{Alt}).}\qed
\end{eg}

\begin{remark}\label{R1}
{\rm Observe that $0160$, $0250$, $0340$ are loops in $\mathcal{T}$ and $\pi_1(\mathcal{T}, 0) = \langle\alpha_1,
\alpha_2\rangle = \langle \alpha_1, \alpha_3\rangle = \langle\alpha_2, \alpha_3\rangle \cong \ZZ \oplus \ZZ$,
where $\alpha_1 = [0160]$, $\alpha_2 = [0250]$, $\alpha_3 = [0340] = \alpha_1 + \alpha_2$ (cf. Fig. 2a).
Moreover, $\alpha_i = 0$ in $\pi_1(T_i, 0)$ for $1\leq i\leq 3$. Thus, for $1\leq i\neq j\leq 3$, the
triangulated $3$-manifold $S_{ij} =  T_i \cup T_j$ is simply connected and hence is a triangulation of
$S^{\hspace{.3mm}3}$. }
\end{remark}

\begin{eg} \label{ETjn}
{\rm Let $T_1, T_2, T_3$ be as in Eq. \eqref{ET123} with common vertex set $\ZZ_7$. Take nine pairwise disjoint
(and disjoint from the set $\ZZ_7$) countable sets 
\begin{align*}
& U_{1}= \{u_{1,n} : n\geq 0\}, \quad U_{2} = \{u_{2,n} : n\geq
0\}, \quad U_{3} = \{u_{3,n} : n\geq 0\}, \\
& V_{1} = \{v_{1,n} : n\geq 0\}, \quad V_2 = \{v_{2,n} : n\geq 0\},
\quad V_3= \{v_{3,n} : n\geq 0\}, \\
& W_{1} = \{w_{1,n} : n\geq 0\}, \quad W_{2} = \{w_{2,n} : n\geq 0\},
\quad W_{3} = \{w_{3,n} : n\geq 0\}.
\end{align*}
First consider the subdivisions $T_{1,0}$ and
$T^{\hspace{.3mm}\prime}_{1,0}$ of $T_1$ by adding new
vertices $u_{1,0}$, $v_{1,0}$ and $u_{1,0}$, $v_{1,0}$, $w_{1,0}$ respectively.
\begin{eqnarray} \label{T10}
T_{1,0} & := & \kappa^{u_{1,0}v_{1,0}}_{345}(\kappa^{2v_{1,
0}}_{016}(\kappa^{3v_{1,0}}_{456}(\tau^{v_{1,0}}_{056} (\kappa^{5u_{1,
0}}_{234}(\tau^{u_{1,0}}_{123}(T_1)), \nonumber \\
T^{\hspace{.3mm}\prime}_{1,0} & := &
\tau^{w_{1,0}}_{012}(T_{1,0}).
\end{eqnarray}
Observe that $T_1$ is the union of the triangulated 3-balls $B_1 =
\{0123, 1234, 2345\}$ and $B_{2} = \{3456, 0456, 0156, 0126\}$ and
$T_{1,0} = \kappa^{u_{1,0}v_{1,0}}_{345}((u_{1,0}\ast \partial
B_1)\cup (v_{1,0}\ast \partial B_2))$.

Consider the maps $f, g \, \colon \ZZ_7 \sqcup
(\sqcup_{i=1}^3(U_{i} \sqcup  V_{i} \sqcup W_{i})) \to \ZZ_7
\sqcup (\sqcup_{i=1}^3(U_{i} \sqcup  V_{i} \sqcup W_{i}))$  given
by\,:
\begin{eqnarray} \label{maps_fg}
&& f(i) = i+1, \, g(i) = 2i \, (\mbox{mod } 7) \, \mbox{ for } \, i \in \ZZ_7, \nonumber \\
&& f(u_{j,n}) = u_{j,n+1}, f(v_{j,n}) = v_{j,n+1}, f(w_{j,n}) = w_{j,n+1}, \, \mbox{ for } 1\leq j\leq 3,  n\geq
0 \, \mbox{ and } \nonumber \\ && g(u_{1,n}) = u_{2,2n}, \, g(u_{2,n}) = u_{3,2n}, \, g(u_{3,n}) = u_{1,2n}, \,
g(v_{1,n}) = v_{2,2n}, \, g(v_{2,n}) = v_{3,2n}, \nonumber \\ && g(v_{3,n}) = v_{1,2n}, \, g(w_{1,n}) = w_{2,2n},
\, g(w_{2,n}) = w_{3,2n}, \, g(w_{3,n}) = w_{1,2n}, \, \mbox{ for } \, n \geq 0.
\end{eqnarray}
Then, $f|_{\ZZ_7}$ is an automorphism of $T_j$ for $1\leq j\leq 3$, \, $g(T_1) = T_{2}$, \, $g(T_2) = T_{3}$, \,
$g(T_3) = T_{1}$ and $g \circ f = f^2 \circ g$. Let
\begin{eqnarray} \label{ET123n}
T_{j, 0}& := & g^{j-1}(T_{1,0}) \hspace{5mm} \mbox{ for } \, j=2, 3, \nonumber \\
T^{\hspace{.3mm}\prime}_{j, 0} & := & g^{j-
1}(T^{\hspace{.3mm}\prime}_{1, 0}) \hspace{5mm} \mbox{ for }
\, j=2, 3, \nonumber \\
T_{j,n} & := & \left\{
\begin{array}{ll}
f^n(T_{j,0}) & \mbox{ for }  \, 1\leq j\leq 3, \, 1\leq n \leq 6,
\\[1mm]
f^n(T^{\hspace{.3mm}\prime}_{j,0}) & \mbox{ for }  \, 1\leq j\leq 3, \, n \geq 7.
\end{array}
\right.
\end{eqnarray}
For $1\leq j\leq 3$, the number of vertices in $T_{j, n}$ is $7+2=9$ if $0 \leq n\leq 6$ and the number of
vertices in $T_{j, n}$ is $7+3=10$ if $n\geq 7$. }\qed
\end{eg}

\begin{lemma} \label{LT123n}
Let $\mathcal{T}$ be the torus given by Eq. \eqref{ET} and $T_1, T_2, T_3$ be given by the Eq. \eqref{ET123}. For
$1\leq j\leq 3$ and $n\geq 0$, the solid tori $T_{j, n}$ defined in Example $\ref{ETjn}$ satisfy the following.
\begin{enumerate}[{\rm (i)}]
\item $T_{j, n}$ is a subdivision of $T_j$ for  $1\leq j\leq 3$ and $n \geq  0$.
\item $T_{i,k} \cong T_{j, \ell}$, $T_{i, m} \cong T_{j, n}$ for $1\leq i, j \leq 3$, $0 \leq k, \ell \leq 6$,
$m, n \geq 7$.
\item $\partial T_{j, n} = \mathcal{T}$ for $1\leq j\leq 3$ and $n \geq   0$.
\item $T_{i} \cap T_{j, n} = \mathcal{T} = T_{i, m} \cap T_{j, n}$ for $1\leq i \neq j \leq 3$ and $m, n\geq 0$.
\item $T_{j,m} \cap T_{j, n} = \mathcal{T} = T_{j} \cap T_{j, \hspace{.2mm}l}$ for $1 \leq j\leq 3$, $m \neq
n\geq 0$ and $\ell\geq 7$.
\item $T_{i} \cup T_{j, n}$, $T_{i, m} \cup T_{j, n}$ triangulate  $S^{\hspace{.2mm}3}$ for $1 \leq i \neq j \leq
3$ and $m, n \geq  0$.
\item $T_{j, m} \cup T_{j, n}$, $T_{j} \cup T_{j, \ell}$ triangulate  $S^{\hspace{.25mm}2} \times S^1$ for $1
\leq j \leq 3$,  $m \neq n \geq 0$ and $\ell \geq  7$.
\end{enumerate}
\end{lemma}

\begin{proof}
Parts (i) and (ii) follow from the definition of $T_{j, n}$ in \eqref{ET123n}. From \eqref{T10} it follows that
$\partial T_{1, 0} = \mathcal{T} = \partial T^{\hspace{.3mm}\prime}_{1, 0}$. Part (iii) now follows from the
definition of $T_{j,n}$ in \eqref{ET123n}.

To prove part (iv), assume that $(i, j) = (1, 2)$. (Similar arguments work for other values of $(i, j)$.)
Clearly, 
$$
V(T_{1, m}) \subseteq \ZZ_7\cup \{u_{1,m}, v_{1,m}, w_{1,m}\}~ \mbox{and}~ V(T_{2, n}) \subseteq
\ZZ_7\cup \{u_{2,n}, v_{2,n}, w_{2,n}\}.
$$ 
Let $\sigma$ be a face in $T_{1, m} \setminus \mathcal{T}$ (resp., in
$T_1\setminus \mathcal{T}$). Then, either $\sigma = \{m, m+1, m+2\}$ (when $m \leq 6$) or $\sigma$ contains a
vertex from $\{u_{1,m}, v_{1,m}, w_{1,m}\}$ (resp., $\sigma = \{i, i+1, i+2\}$ for some $i\in \ZZ_7$). Similarly,
if $\tau$ be a face in $T_{2, n} \setminus \mathcal{T}$ then either $\tau = \{n, n+2, n+4\}$ (when $n \leq 6$) or
$\tau$ contains a vertex from $\{u_{2,n}, v_{2,n}, w_{2,n}\}$. (Addition here is mod $7$.)  Thus, $\sigma \neq
\tau$. This proves part (iv)

To prove part (v), assume that $j=1$. (Similar arguments work for other values of $j$.)  Any face in $T_{1,m}
\setminus \mathcal{T}$ is either $\{m, m+1, m+2\}$ (when $m \leq 6$) or contains a vertex from $\{u_{1,m},
v_{1,m}, w_{1,m}\}$ and any face in $T_{1,n}\setminus \mathcal{T}$ is either $\{n, n+1, n+2\}$ (when $n \leq 6$)
or contains a vertex from $\{u_{1,n}, v_{1,n}, w_{1,n}\}$. Since $m \neq n$, there is no common face in
$T_{1,m}\setminus \mathcal{T}$ and in $T_{1,n}\setminus \mathcal{T}$. Thus, $T_{1, m} \cap T_{1, n} =
\mathcal{T}$. Similarly, $T_{1} \cap T_{1, \ell} = \mathcal{T}$ for $\ell\geq 7$. This proves part (v).

Since $T_{i, m}$ is a subdivision of $T_i$ for $1\leq i \leq 3$ and $m\geq 0$, $T_{i, m}\cup T_{j, n}$  is a
subdivision of $T_i \cup T_j = S_{ij}$ for $1\leq i \neq j \leq 3$. Part (vi) now follows since $S_{ij}$ is a
triangulation of $S^{\hspace{.2mm}3}$.

Consider the oriented $3$-dimensional weak pseudomanifold $X$ whose vertex set is $\{d_0, d_1$, $\dots, d_9, u,
x, a, b\}$ and oriented facets are
\begin{align*}
&d_{i}d_{i+1}d_{i+2}d_{i+3}, d_{0}d_{1}ua,
d_{0}d_{2}au, d_{1}d_{2}ua, d_jd_{j +1}d_{j+3}u, ud_{j}d_{j+2}d_{j+3},  \mbox{ for } 0\leq i \leq 6,
0\leq j\leq 2,
\\
& d_{k}d_{k
+1}d_{k+3}x,  xd_{k}d_{k+2}d_{k+3} \mbox{ for }  3\leq k\leq 6, ~ d_3d_4xu, d_4d_5xu, d_3d_5ux, d_{7}d_{8}bx,
d_7d_9xb, d_8d_9bx.
\end{align*}
Then, $X$ triangulates $[0, 1] \times S^{\hspace{.2mm}2}$ and the 2-spheres
$S^{\hspace{.2mm}2}_4(\{d_0, d_1, d_2, a\})$ and $S^{\hspace{.2mm}2}_4(\{d_7, d_8, d_9, b\})$ are induced and in
$\partial X$. The orientation on $X$ induces orientation on $S^{\hspace{.2mm}2}_4(\{d_0, d_1, d_2, a\})$ (namely
$\alpha$ is a positively oriented 2-face if $a\alpha$ is positively oriented in $X$) with the positively oriented
2-faces $d_0d_2d_1$, $d_2d_0a$, $d_1d_2a$, $d_0d_1a$. Similarly (with the induced orientation), the positively
oriented 2-faces in $S^{\hspace{.2mm}2}_4(\{d_7, d_8, d_9, b\})$ are $d_7d_8d_9$, $d_7d_9b$, $d_9d_8b$ and
$d_8d_7b$.

Now, consider the map 
$$
\varphi \colon \{d_0, d_1, d_2, a\} \to \{d_7, d_8, d_9, b\}
$$ 
given by $\varphi
(d_i) = d_{i+7}$ for $0 \leq i\leq 2$ and $\varphi(a) = b$. Then, $\varphi$ defines an equivalence relation $\sim$
on $V(X)$ (namely, $y\sim z$ if $y=z$, $\varphi(y) = z$ or $y = \varphi(z)$). Observe that if $\alpha$ is a
positively oriented 2-face, then $\varphi(\alpha)$ is negatively oriented. So, the quotient complex
$X/\hspace{-1.5mm}\sim$ is orientable and triangulates $S^{\hspace{.2mm}2} \times S^1$. It is easy to see that
$X/\hspace{-1.5mm}\sim$ is isomorphic to $T_1 \cup T_{1, 0}^{\hspace{.2mm}\prime}$. Therefore, $T_1 \cup T_{1,
0}^{\hspace{.2mm}\prime}$ triangulates $S^{\hspace{.2mm}2} \times S^1$. Since $T_{1, m} \cup T_{1,n}$ (resp.,
$T_{1} \cup T_{1,\ell}$) is stellar equivalent (resp., is isomorphic) to $T_1 \cup T_{1,0}^{\hspace{.2mm}\prime}$
for $m\neq n$ (resp. $\ell\geq 7$), it follows that $T_{1, m} \cup T_{1,n}$ and $T_{1} \cup T_{1,\ell}$
triangulate $S^{\hspace{.2mm}2}\times  S^1$. This proves part (vii) since $g^2 \colon T_{2, m} \cup T_{2, n} \to
T_{1, 4m} \cup T_{1,4n}$, $g \colon T_{3, m} \cup T_{3, n} \to T_{1, 2m} \cup T_{1,2n}$ are isomorphisms for
$m\neq n$ and $g^2 \colon T_{2} \cup T_{2, \ell} \to T_{1} \cup T_{1,4n}$, $g \colon T_{3} \cup T_{3, \ell} \to
T_{1} \cup T_{1,2l}$ are isomorphisms for $\ell \geq 7$.
\end{proof}

\begin{remark} \label{R2}
{\rm For $1\leq j \leq 3$ and $m \neq n$, $|T_{j,m}|$ and $|T_{j,n}|$ are two copies of the same solid torus
$|T_1|$ and (since $|T_{j,m}|$ and $|T_{j,n}|$ have no common interior face) $|T_{j,m} \cup T_{j,n}|$ is the
space obtained from the disjoint union of $|T_{j,m}|$ and $|T_{j,n}|$ by identifying their boundaries with the
identity map. This implies $|T_{j,m} \cup T_{j,n}|$ is homeomorphic to $S^{\hspace{.2mm}2}\times  S^1$.
Similarly, $|T_{j} \cup T_{j,\ell}|$ is homeomorphic to $S^{\hspace{.2mm}2}\times  S^1$ for $\ell\geq 7$. This
gives another proof of Lemma \ref{LT123n} (vii).}
\end{remark}

\setlength{\unitlength}{3mm}

\begin{picture}(48,33)(1,-1.5)

\thicklines

\put(2,5){\line(1,0){10}} \put(2,5){\line(0,1){22}}
\put(2,27){\line(1,0){10}} \put(12,5){\line(0,1){22}}
\put(2,5){\line(3,-2){3}}  \put(5,3){\line(1,0){4}}
\put(9,3){\line(3,2){3}} \put(12,5){\line(-1,1){2}}
\put(10,7){\line(-3,1){3}} \put(4,7){\line(3,1){3}}
\put(2,5){\line(1,1){2}}

\put(2,27){\line(3,-2){3}}  \put(5,25){\line(1,0){4}}
\put(9,25){\line(3,2){3}} \put(12,27){\line(-1,1){2}}
\put(10,29){\line(-3,1){3}} \put(4,29){\line(3,1){3}}
\put(2,27){\line(1,1){2}}

\put(7,5){\line(3,2){3}}\put(7,5){\line(1,-1){2}}
\put(7,5){\line(-3,2){3}} \put(5,3){\line(1,1){2}}
\put(7,5){\line(0,1){3}}

\put(7,27){\line(3,2){3}}\put(7,27){\line(1,-1){2}}
\put(7,27){\line(-3,2){3}} \put(5,25){\line(1,1){2}}
\put(7,27){\line(0,1){3}}


\put(18,3){\line(1,0){21}} \put(18,3){\line(0,1){9}}
\put(18,12){\line(1,0){21}} \put(39,3){\line(0,1){9}}

\put(30,0){\line(-4,1){12}} \put(30,0){\line(3,1){9}}
\put(18,12){\line(3,1){9}} \put(39,12){\line(-4,1){12}}


\put(22,19){\line(1,0){21}} \put(22,19){\line(0,1){9}}
\put(22,28){\line(1,0){21}} \put(43,19){\line(0,1){9}}

\put(34,16){\line(-4,1){12}} \put(34,16){\line(3,1){9}}
\put(22,28){\line(3,1){9}} \put(43,28){\line(-4,1){12}}

\thinlines

\put(2,5){\line(1,5){2}} \put(12,5){\line(-1,5){2}}
\put(4,15){\line(-1,6){2}} \put(10,15){\line(1,6){2}}

\put(4,15){\line(2,-1){2}} \put(6,14){\line(1,0){2}}
\put(8,14){\line(2,1){2}} \put(10,15){\line(-1,2){1}}
\put(9,17){\line(-2,1){2}} \put(4,15){\line(1,2){1}}
\put(5,17){\line(2,1){2}}

\put(7,22){\line(2,3){2}} \put(7,22){\line(1,1){5}}
\put(7,22){\line(-2,3){2}} \put(7,22){\line(-1,1){5}}
\put(7,18){\line(0,1){6.6}} \put(7,25.4){\line(0,1){1.6}}

\put(5,17){\line(2,5){2}} \put(9,17){\line(-2,5){2}}

\put(4,7){\line(1,3){2.15}} \put(10,7){\line(-1,3){2.15}}
\put(7,8){\line(0,1){5.5}}

\put(7,16){\line(-1,-3){0.6}} \put(7,16){\line(1,-3){0.6}}
\put(7,16){\line(0,-1){1.7}}

\put(4,15){\line(3,1){3}} \put(10,15){\line(-3,1){3}}
\put(6,14){\line(1,2){1}} \put(8,14){\line(-1,2){1}}
\put(7,16){\line(-2,1){2}} \put(7,16){\line(2,1){2}}
\put(7,16){\line(0,1){2}}


\put(21,3){\line(0,1){9}} \put(24,3){\line(0,1){9}}
\put(27,3){\line(0,1){9}} \put(30,3){\line(0,1){9}}
\put(33,3){\line(0,1){9}} \put(36,3){\line(0,1){9}}

\put(18,3){\line(1,3){3}} \put(21,3){\line(1,3){3}}
\put(24,3){\line(1,3){3}} \put(27,3){\line(1,3){3}}
\put(30,3){\line(1,3){3}} \put(33,3){\line(1,3){3}}
\put(36,3){\line(1,3){3}}

\put(30,0){\line(-3,1){9}} \put(30,0){\line(-2,1){6}}
\put(30,0){\line(-1,1){3}} \put(30,0){\line(0,1){3}}
\put(30,0){\line(2,1){6}} \put(30,0){\line(1,1){3}}

\put(21,12){\line(2,1){6}} \put(24,12){\line(1,1){3}}
\put(27,12){\line(0,1){3}} \put(30,12){\line(-1,1){3}}
\put(33,12){\line(-2,1){6}} \put(36,12){\line(-3,1){9}}


\put(22,23.5){\line(1,0){21}}

\put(25,19){\line(0,1){9}} \put(28,19){\line(0,1){9}}
\put(31,19){\line(0,1){9}} \put(34,19){\line(0,1){9}}
\put(37,19){\line(0,1){9}} \put(40,19){\line(0,1){9}}

\put(22,19){\line(2,3){6}} \put(25,19){\line(2,3){6}}
\put(28,19){\line(2,3){6}} \put(31,19){\line(2,3){6}}
\put(34,19){\line(2,3){6}} \put(37,19){\line(2,3){6}}
\put(22,23.5){\line(2,3){3}} \put(40,19){\line(2,3){3}}

\put(34,16){\line(-3,1){9}} \put(34,16){\line(-2,1){6}}
\put(34,16){\line(-1,1){3}} \put(34,16){\line(0,1){3}}
\put(34,16){\line(2,1){6}} \put(34,16){\line(1,1){3}}

\put(25,28){\line(2,1){6}} \put(28,28){\line(1,1){3}}
\put(31,28){\line(0,1){3}} \put(34,28){\line(-1,1){3}}
\put(37,28){\line(-2,1){6}} \put(40,28){\line(-3,1){9}}



\put(5.8,28.2){\mbox{$u_{4,n}$}} \put(7.8,21.8){\mbox{$v_{4,n}$}}
\put(7.3,16.8){\mbox{$w_{4,n}$}}
\put(8.2,4.1){\mbox{$u_{4,n}^{\hspace{.3mm}\prime}$}}

\put(1,3.8){\mbox{$p_{0}^{\hspace{.3mm}\prime}$}}
\put(3.5,2){\mbox{$p_{1}^{\hspace{.3mm}\prime}$}}
\put(9.2,2){\mbox{$p_{2}^{\hspace{.3mm}\prime}$}}
\put(11.6,3.8){\mbox{$p_{3}^{\hspace{.3mm}\prime}$}}
\put(10,7.5){\mbox{$p_{4}^{\hspace{.3mm}\prime}$}}
\put(7.4,8.5){\mbox{$p_{5}^{\hspace{.3mm}\prime}$}}
\put(2.8,7.5){\mbox{$p_{6}^{\hspace{.3mm}\prime}$}}

\put(1,27.8){\mbox{$p_{2}$}} \put(4.5,26){\mbox{$p_{3}$}}
\put(8.7,26){\mbox{$p_{4}$}} \put(11.7,27.8){\mbox{$p_{5}$}}
\put(10.2,29.3){\mbox{$p_{6}$}} \put(7.4,30.5){\mbox{$p_{0}$}}
\put(3,29.4){\mbox{$p_{1}$}}

\put(2.15,15){\mbox{$q_{n,1}$}} \put(4.3,13.5){\mbox{$q_{n, 2}$}}
\put(8.2,13.2){\mbox{$q_{n,3}$}} \put(10.1,15){\mbox{$q_{n,4}$}}
\put(9.2,17.7){\mbox{$q_{n,5}$}} \put(6.2,18.5){\mbox{$q_{n,6}$}}
\put(3.3,17.5){\mbox{$q_{n,0}$}}


\put(16.5,3.7){\mbox{$p_{0}^{\hspace{.3mm}\prime}$}}
\put(19.5,3.7){\mbox{$p_{1}^{\hspace{.3mm}\prime}$}}
\put(22.,3.7){\mbox{$p_{2}^{\hspace{.3mm}\prime}$}}
\put(25.5,3.7){\mbox{$p_{3}^{\hspace{.3mm}\prime}$}}
\put(28.5,3.7){\mbox{$p_{4}^{\hspace{.3mm}\prime}$}}
\put(31.5,3.7){\mbox{$p_{5}^{\hspace{.3mm}\prime}$}}
\put(34.5,3.7){\mbox{$p_{6}^{\hspace{.3mm}\prime}$}}
\put(37.5,3.7){\mbox{$p_{0}^{\hspace{.3mm}\prime}$}}

\put(18.5,11){\mbox{$p_{2}$}} \put(21.5,11){\mbox{$p_{3}$}}
\put(24.5,11){\mbox{$p_{4}$}} \put(27.5,11){\mbox{$p_{5}$}}
\put(30.5,11){\mbox{$p_{6}$}} \put(33.5,11){\mbox{$p_{0}$}}
\put(36.5,11){\mbox{$p_{1}$}} \put(39.5,11){\mbox{$p_{2}$}}

\put(28,15.2){\mbox{$u_{4,n}$}} \put(31.5,-0.5){\mbox{$u_{4,
n}^{\hspace{.3mm}\prime}$}}


\put(20.5,19.7){\mbox{$p_{0}^{\hspace{.3mm}\prime}$}}
\put(23.5,19.7){\mbox{$p_{1}^{\hspace{.3mm}\prime}$}}
\put(26.5,19.7){\mbox{$p_{2}^{\hspace{.3mm}\prime}$}}
\put(29.5,19.7){\mbox{$p_{3}^{\hspace{.3mm}\prime}$}}
\put(32.5,19.7){\mbox{$p_{4}^{\hspace{.3mm}\prime}$}}
\put(35.5,19.7){\mbox{$p_{5}^{\hspace{.3mm}\prime}$}}
\put(38.5,19.7){\mbox{$p_{6}^{\hspace{.3mm}\prime}$}}
\put(41.5,19.7){\mbox{$p_{0}^{\hspace{.3mm}\prime}$}}

\put(22.2,22.6){\mbox{$q_{n,1}$}}
\put(25.2,22.6){\mbox{$q_{n,2}$}}
\put(28.2,22.6){\mbox{$q_{n,3}$}}
\put(31.2,22.6){\mbox{$q_{n,4}$}}
\put(34.2,22.6){\mbox{$q_{n,5}$}}
\put(37.2,22.6){\mbox{$q_{n,6}$}}
\put(40.2,22.6){\mbox{$q_{n,0}$}}
\put(43.2,22.6){\mbox{$q_{n,1}$}}

\put(22.5,27){\mbox{$p_{2}$}} \put(25.5,27){\mbox{$p_{3}$}}
\put(28.5,27){\mbox{$p_{4}$}} \put(31.5,27){\mbox{$p_{5}$}}
\put(34.5,27){\mbox{$p_{6}$}} \put(37.5,27){\mbox{$p_{0}$}}
\put(40.5,27){\mbox{$p_{1}$}} \put(43.5,27){\mbox{$p_{2}$}}

\put(27,31){\mbox{$u_{4,n}$}} \put(35.5,15.5){\mbox{$u_{4,
n}^{\hspace{.3mm}\prime}$}}

\put(6,0.5){\mbox{${\mathcal B}_{4,n}$}} \put(22,0){\mbox{$\partial
{\mathcal B}_{4,n}$}}

\put(20,16.5){\mbox{$\partial {\mathcal B}^{\hspace{.3mm}\prime}
$}}

\put(14,0){\mbox{\bf Figure\,2b}}

\end{picture}

\begin{eg} \label{EB4n}
{\rm For each integer $n \geq 0$, consider the 3-dimensional simplicial complex ${\mathcal B}_{4,n} = {\mathcal
B}^{\hspace{.3mm}\prime} \cup {\mathcal B}^{\hspace{.3mm}\prime\prime}$ on the vertex set $\{p_0, p_1, \dots,
p_6\} \cup\{p_0^{\hspace{.3mm}\prime}, \dots, p_6^{\hspace{.3mm}\prime}\} \cup \{q_{n,0}, q_{n,1}, \dots, q_{n,
6}\} \cup $ $\{u_{4,n}, v_{4,n}, w_{4,n}, u_{4, n}^{\hspace{.2mm}\prime}\}$.

\begin{eqnarray*}
{\mathcal B}^{\hspace{.3mm}\prime} & = & \{w_{4,n}u_{4,n}^{\hspace{.3mm}\prime}p_i^{\hspace{.3mm}\prime} p_{i
+1}^{\hspace{.3mm}\prime}, w_{4,n} p_i^{\hspace{.3mm}\prime} p_{i+1}^{\hspace{.3mm}\prime}q_{n,i+2}, w_{4,n}
p_{i-1}^{\hspace{.3mm}\prime}q_{n,i}q_{n, i+1}, v_{4,n}w_{4,n}q_{n,i}q_{n,i+1}, \\ && ~~  v_{4,n}q_{n,i}p_{i+1}
p_{i+2}, v_{4,n}q_{n,i}q_{n,i+1}p_{i+2},
u_{4,n}v_{4,n}p_ip_{i+1}  : \, i \in \ZZ_7\},  \\
{\mathcal B}^{\hspace{.3mm}\prime\prime} &= & \{p_{i-1}^{\hspace{.3mm}\prime}q_{n,i}q_{n,i+1}p_{i+2},
p_{i-2}^{\hspace{.3mm}\prime}p_{i-1}^{\hspace{.3mm}\prime}q_{n, i}p_{i+1},
p_{i-1}^{\hspace{.3mm}\prime}q_{n,i}p_{i+1}p_{i+2}, \, : \, i \in \ZZ_7\}.
\end{eqnarray*}
Consider the equivalence relation `$\sim$' on the vertex set $V({\mathcal B}_{4,n})$ generated by
$u_{4,n}^{\hspace{.3mm}\prime} \sim u_{4,n}$, $p_i^{\hspace{.3mm}\prime} \sim p_i$ for $0\leq i \leq 6$. Let
\begin{eqnarray} \label{ET4n}
T_{4, n} = {\mathcal B}_{4,n}/\hspace{-1.5mm}\sim.
\end{eqnarray}
We identify the vertices $[p_i]$, $[u_{4,n}]$, $[v_{4,n}]$ and $[w_{4,n}]$ in $T_{4,n}$ with $i$, $u_{4,n}$,
$v_{4,n}$ and $w_{4,n}$ respectively. So, $V(T_{4,n}) = \ZZ_7\cup \{q_{n,0}, \dots, q_{n,6}\} \cup\{u_{4,n},
v_{4,n}, w_{4,n}\}$ for $n\geq 0$. We assume that $V(T_{4,n}) \cap V(T_{4,m}) = \ZZ_7 = V(T_{4,n}) \cap
V(T_{j,\hspace{.1mm}l})$ for all $\ell, m\neq n \geq 0$, $1\leq j\leq 3$. }\qed
\end{eg}

\begin{lemma} \label{LT4n}
For $n \geq 0$, the simplicial complexes ${\mathcal B}_{4,n}$ and $T_{4, n}$ defined in Example $\ref{EB4n}$
satisfy the following.
\begin{enumerate}[{\rm (i)}]
\item ${\mathcal B}_{4,n}$ is a triangulated $3$-ball with boundary $\partial {\mathcal B}_{4,n}  = {\mathcal
C}_1 \cup {\mathcal C}_2 \cup {\mathcal C}_3$, where ${\mathcal C}_1 = \{u_{4, n}p_ip_{i + 1} : i \in \ZZ_7\}$,
${\mathcal C}_2 = \{u_{4, n}^{\hspace{.3mm}\prime} p_i^{\hspace{.3mm}\prime} p_{i + 1}^{\hspace{.3mm}\prime} : i
\in \ZZ_7\}$, ${\mathcal C}_3 = \{p_{i - 2}^{\hspace{.3mm}\prime}p_{i - 1}^{\hspace{.3mm}\prime}p_{i + 1}, p_{i -
1}^{\hspace{.3mm}\prime} p_{i + 1}p_{i + 2} : i \in \ZZ_7\}$.

\item $T_{4, n}$ is a triangulated solid torus with boundary $\partial T_{4, n} = \mathcal{T}$.

\item If $\alpha_1$, $\alpha_2$ are as in Remark $\ref{R1}$ and $\alpha_4 = [c_4] \in \pi_1(\mathcal{T}, 0)$,
where $c_4$ is the loop $01234560$ in $\mathcal{T}$, then $\alpha_4 = 3\alpha_1 +\alpha_2$ and $\alpha_4 =0$ in
$\pi_1(T_{4,n}, 0)$.
\end{enumerate}
\end{lemma}

\begin{proof}
Consider the seventeen triangulated $3$-balls 

${\mathcal B}_1 = \{w_{4,n}u_{4,n}^{\hspace{.3mm}\prime}p_i^{\hspace{.3mm}\prime}p_{i +1}^{\hspace{.3mm}\prime}$, ~$w_{4,
n}p_i^{\hspace{.3mm}\prime}p_{i +1}^{\hspace{.3mm}\prime}q_{n, i+2}$, 
~$w_{4,n}p_{i-1}^{\hspace{.3mm}\prime}q_{n,i}q_{n, i+1} \hspace{.2mm} : \hspace{.2mm} i \in \ZZ_7\}$,\newline

${\mathcal B}_2 = \{v_{4, n}q_{n,i}p_{i+1}p_{i+2}, ~v_{4,n}q_{n,i}q_{n,i+1}p_{i+2}, ~v_{4, n}w_{4,n}q_{n,i}q_{n,i+1}
\hspace{.2mm} :  \hspace{.2mm} i \in \ZZ_7\}$,\newline

${\mathcal B}_3 = \{u_{4,n}v_{4,n}p_ip_{i+1} \hspace{.2mm} :
\hspace{.2mm} i \in \ZZ_7\}$, $~{\mathcal B}_4 = \{p_{0}^{\hspace{.3mm}\prime} q_{n,1}q_{n,2}p_{3}\}, \dots,
{\mathcal B}_{10} = \{p_{6}^{\hspace{.3mm}\prime} q_{n,0}q_{n,1}p_{2}\}$,\newline

${\mathcal B}_{11} = \{p_{0}^{\hspace{.3mm}\prime} p_{1}^{\hspace{.3mm}\prime} q_{n, 2}p_{3}, ~p_{1}^{\hspace{.3mm}\prime}
q_{n,2}p_{3}p_{4}\}, \dots, {\mathcal B}_{17} = \{p_{6}^{\hspace{.3mm}\prime}p_{0}^{\hspace{.3mm}\prime} q_{n,
1}p_{2}, ~p_{0}^{\hspace{.3mm}\prime}q_{n,1}p_{2}p_{3}\}$.

Clearly, ${\mathcal B}_{4,n} = {\mathcal B}_1 \cup \cdots \cup {\mathcal B}_{17}$. Observe that
${\mathcal B}_i \cap ({\mathcal B}_1 \cup \cdots \cup {\mathcal B}_{i-1})$ is a triangulated 2-disc
for $2 \leq i \leq 17$. This implies that ${\mathcal B}_{4,n} = {\mathcal B}_1 \cup \cdots \cup {\mathcal B}_{17}$
is a triangulated 3-ball. This proves first part of (i). (Observe that
${\mathcal B}^{\hspace{.3mm}\prime} = {\mathcal B}_1 \cup {\mathcal B}_2 \cup {\mathcal B}_3$ is also a 3-ball.)
Second part of (i) follows from the definition of ${\mathcal B}_{4,n}$.

There is no vertex in ${\mathcal B}_{4,n}$ which is a common neighbour of $p_{i}^{\hspace{.3mm}\prime}$ and
$p_{i}$ for $i\in \ZZ_7$ and there is no vertex in ${\mathcal B}_{4,n}$ which is a common neighbour of $u_{4,
n}^{\hspace{.3mm}\prime}$ and $u_{4,n}$. This implies that the quotient complex $T_{4,n}$ is a triangulated
3-manifold with boundary. Clearly, the boundary $\partial T_{4, n}$ is the quotient ${\mathcal
C}_3/\hspace{-1.5mm}\sim \, = T$. Since the space $|T_{4,n}|$ can be obtained from $|{\mathcal B}_{4,n}|$ by
identifying the 2-disc $|{\mathcal C}_1|$ with the 2-disc $|{\mathcal C}_2|$ (via the simplicial isomorphism from
${\mathcal C}_1$ to ${\mathcal C}_2$ given by $u_{4,n} \mapsto u_{4,n}^{\hspace{.3mm}\prime}$, $p_i \mapsto
p_i^{\hspace{.3mm}\prime} $ for $0\leq i \leq 6$), $|T_{4,n}|$ is a solid torus. This proves (ii).

It is clear that, $0160340160 \simeq 01603405160 \simeq 0160340560 \simeq 016034560 \simeq 0160234560 \simeq 016234560 \simeq
01234560$ (see Fig. 2a). Thus, $\alpha_4= [c_4] = [0160] + [0340] + [0160] = \alpha_1 + \alpha_3 + \alpha_1 =
3\alpha_1 + \alpha_2$. Let ${\mathcal C}$ be the 2-disc in $T_{4,n}$ corresponding to the 2-disc ${\mathcal C}_1$
in ${\mathcal B}_{4,n}$. Then, $c_4$ is the boundary of ${\mathcal C}$ and hence $\alpha_4 = [c_4] =0$ in
$\pi_1(T_{4,n}, 0)$. This completes the proof of (iii).
\end{proof}

\begin{eg} \label{ET56n}
{\rm Let $Q_n =
\{q_{n,0}, \dots, q_{n,6}\}$, $U_4 = \{u_{4,n} : n\geq 0\}$, $V_4 = \{v_{4,n} : n\geq 0\}$, $W_4 = \{w_{4,n} :
n\geq 0\}$ be as in Example \ref{EB4n}. Let $Q = \cup_{n\geq 0}Q_n$. Take new (disjoint) sets of vertices
\begin{align*}
R & = \cup_{n\geq 0}R_n = \cup_{n\geq 0}\{r_{n,0}, \dots, r_{n,6}\}, ~S = \cup_{n\geq 0}S_n = \cup_{n\geq 0}\{s_{n,0},
\dots, s_{n,6}\}, \\
U_j & = \{u_{j,n} \, : n\geq 0\}, ~V_j = \{v_{j,n} \, : n\geq 0\}, ~W_j = \{w_{j,n} \, :
n\geq 0\} ~\mbox{for} ~5 \leq j\leq 6.
\end{align*}

Let $f$ and $g$ be as in \eqref{maps_fg}. Consider the following extensions of
$f$ and $g$ (also denoted by $f$ and $g$ respectively) $ {f},  {g} \, \colon \ZZ_7 \sqcup (Q\sqcup R\sqcup S)
\sqcup (\sqcup_{i=1}^6(U_{i} \sqcup V_{i} \sqcup W_{i})) \to \ZZ_7 \sqcup (Q\sqcup R\sqcup S) \sqcup
(\sqcup_{i=1}^6(U_{i} \sqcup  V_{i} \sqcup W_{i}))$  as\,:
\begin{eqnarray} \label{maps_fg456}
&&  {f}(u_{j,n}) = u_{j,n+1},  {f}(v_{j,n}) =
v_{j,n+1},  {f}(w_{j,n}) = w_{j,n+1}, \, \mbox{ for }
4\leq j\leq 6,  n\geq 0,  \nonumber \\
&&  {f}(q_{n,i}) = q_{n+1,i},  {f}(r_{n,i}) =
r_{n+1,i},  {f}(s_{n,i}) = s_{n+1,i}, \, \mbox{ for }
0\leq i\leq 6,  n\geq 0 \, \mbox{ and } \nonumber \\
&&  {g}(u_{4,n}) = u_{5,2n}, \,  {g}(u_{5,n}) =
u_{6,2n}, \,  {g}(u_{6,n}) = u_{4,2n}, \,  {g}(v_{4,n})
= v_{5,2n}, \,  {g}(v_{5,n}) = v_{6,2n},  \nonumber \\
&&  {g}(v_{6,n}) = v_{4,2n}, \,  {g}(w_{4,n}) =
w_{5,2n}, \,  {g}(w_{5,n}) = w_{6,2n}, \,
 {g}(w_{6,n}) = w_{4,2n}, \,
\mbox{ for } \, n \geq 0, \nonumber \\
&&  {g}(q_{n,i}) = r_{2n,i}, \,  {g}(r_{n,i}) =
s_{2n,i}, \,  {g}(s_{n,i}) = q_{2n,i}, \, 0\leq i\leq 6,  n\geq 0.
\end{eqnarray}
Then, $ {g} \circ  {f} = { {f}}^2 \circ  {g}$. Let
\begin{eqnarray} \label{T56n}
T_{5, 0} :=  {g}(T_{4,0}), &&
T_{6, 0} := { {g}}^{2}(T_{4,0}), \nonumber \\
T_{j,n} := { {f}}^{\hspace{.2mm}n}(T_{j,0}) && \mbox{for } \, n
\geq 1 \, \mbox{ and } \, j=5, 6.
\end{eqnarray}
For $4\leq j\leq 6$ and $n\geq 0$, the number of vertices in $T_{j, n}$ is $7+7+3=17$. } \qed
\end{eg}

\begin{lemma} \label{LT56n}
For $n \geq 0$, $T_{5, n}$, $T_{6, n}$ defined in Example $\ref{ET56n}$ satisfy the following.
\begin{enumerate}[{\rm (i)}]
\item $T_{i, m}\cong T_{j, n}$ for $4\leq i, j\leq 6$ and $m, n\geq 0$.

\item $T_{j, n}$ is a triangulated solid torus with boundary $\partial T_{j, n} = \mathcal{T}$ for $j = 5, 6$,
$n\geq 0$.

\item Let $\alpha_1$, $\alpha_2$ be as in Remark $\ref{R1}$ and $\alpha_5 = [c_5], \alpha_6 = [c_6] \in
\pi_1(\mathcal{T}, 0)$, where $c_5=05316420$, $c_6=03625140$ are loops in $\mathcal{T}$. Then, $\alpha_5 =
\alpha_1 - 2\alpha_2$, $\alpha_6 = 2\alpha_1 + 3\alpha_2$ and $\alpha_j =0$ in $\pi_1(T_{j,n}, 0)$ for $j=5,6$,
$n\geq 0$.
\end{enumerate}
\end{lemma}

\begin{proof}
From the definition in \eqref{T56n}, $T_{i,m} \cong T_{i,0}\cong T_{4,0} \cong T_{j,0} \cong T_{j,n}$ for $4\leq
i, j\leq 6$ and $m, n\geq 0$. This proves part (i).

By part (i), $T_{j, n}\cong T_{4,0}$. Therefore, by Lemma \ref{LT4n} (ii), $T_{j,n}$ triangulates the solid
torus. Since $f(\mathcal{T}) = \mathcal{T} = g(\mathcal{T})$ and $\partial T_{4,0} =\mathcal{T}$, it follows that
$\partial T_{j, n} = \mathcal{T}$. This proves part (ii).

Observe that $g(c_4) = 02461350 = \overline{c}_5$ and $g(c_5) = c_6$, where $\overline{c}_5$ is the 
loop $c_5$ with opposite orientation. Therefore, $g_{\ast}(\alpha_4) = -\alpha_5$
and $g_{\ast}(\alpha_5) = \alpha_6$, where $g_{\ast} \colon \pi_1(\mathcal{T}, 0) \to \pi_1(\mathcal{T}, 0)$ is
induced by the map $g$. Since $g_{\ast}(\alpha_1) = [g(0160)] = [0250] = \alpha_2$ and $g_{\ast}(\alpha_2) =
[g(0250)] = [0430] = -\alpha_1-\alpha_2$, it follows that $\alpha_5=-g_{\ast}(\alpha_4) = -g_{\ast}(3\alpha_1
+\alpha_2) = -3\alpha_2 + (\alpha_1 + \alpha_2) = \alpha_1- 2\alpha_2$. Therefore, $\alpha_6 = g_{\ast}(\alpha_5)
= g_{\ast}(\alpha_1 -2\alpha_2) = \alpha_2 - 2(-\alpha_1 -\alpha_2) = 2\alpha_1 +3\alpha_2$. This proves the
first part of (iii).

Since $g \colon T_{4,0} \to T_{5,0}$ is an isomorphism, $g_{\ast} \colon \pi_1(T_{4,0}, 0) \to \pi_1(T_{5,0}, 0)$
is an isomorphism. Therefore, $\alpha_5 = g_{\ast}(-\alpha_4) = g_{\ast}(0) =0$ in $\pi_1(T_{5,0},0)$. (In fact,
$c_5$ is the boundary of the $2$-disc $g({\mathcal C})$ in $T_{5,0}$, where ${\mathcal C}$ is the $2$-disc
defined in the proof of Lemma \ref{LT4n} (iii).) Now, $f^n \colon T_{5,0} \to T_{5,n}$ is an isomorphism and
$f^n(c_5) = c_5$. This implies that $\alpha_5 = [c_5] =0$ in $\pi_1(T_{5,n},0)$. Similarly, $\alpha_6 = [c_6] =0$
in $\pi_1(T_{6,n},0)$. This completes the proof of (iii).
\end{proof}

\begin{cor} \label{CT4n}
For $n\geq 0$, the solid tori $T_1, T_2, T_3$ defined in \eqref{ET123}, $T_{1,n}, T_{2,n}, T_{3,n}$ defined in
\eqref{ET123n}, $T_{4,n}$ defined in \eqref{ET4n} and $T_{5,n}$, $T_{6,n}$ in \eqref{T56n} satisfy the following.
\begin{enumerate}[{\rm (i)}]
\item $T_{1} \cup T_{4,n}$, $T_{1,m} \cup T_{4,n}$, $T_{2} \cup T_{5,n}$, $T_{2,m} \cup T_{5,n}$, $T_{3} \cup
T_{6,n}$, $T_{3,m} \cup T_{6,n}$ triangulate the $3$-sphere $S^{\hspace{.2mm}3}$ for $m, n\geq 0$.
\item $T_{2} \cup T_{4,n}$, $T_{2,m} \cup T_{4,n}$, $T_{3} \cup T_{5,n}$, $T_{3,m} \cup T_{5,n}$, $T_{1} \cup
T_{6,n}$, $T_{1,m} \cup T_{6,n}$ triangulate the lens space $L(3,1)$ for $m, n\geq 0$.
\item $T_{3} \cup T_{4,n}$, $T_{3,m} \cup T_{4,n}$, $T_{1} \cup T_{5,n}$, $T_{1,m} \cup T_{5,n}$, $T_{2} \cup
T_{6,n}$, $T_{2,m} \cup T_{6,n}$ triangulate the projective space $\mathbb{RP}^{\hspace{.2mm}3}$ for $m, n\geq
0$.
\item $T_{4,m} \cup T_{5,n}$, $T_{4,m} \cup T_{6,n}$, $T_{5,m} \cup T_{6,n}$ triangulate the lens space $L(7,2)$
for $m, n\geq 0$.
\item $T_{j,m} \cup T_{j,n}$ triangulates $S^{\hspace{.2mm}2} \times S^1$ for $4 \leq j \leq 6$ and $m \neq n
\geq 0$.
\end{enumerate}
\end{cor}

\begin{proof}
From the definition of $T_{4,n}$, it follows that $T_{j} \cap T_{4,n} = \mathcal{T} = T_{j,m} \cap T_{4,n}$ for
$1\leq j\leq 3$ and $m, n\geq 0$. For $1 \leq j\leq 3$, let $M_{jmn}= T_{j,m} \cup T_{4,n}$ and let $M_{jn}=
T_{j} \cup T_{4,n}$. If $u$ is in $V(T_{j,m}) \setminus V(\partial T_{j,m})$ (resp., in $V(T_{4,n}) \setminus
V(\partial T_{4,n})$ then, ${\rm lk}_{M_{jmn}}(u)$ is same as ${\rm lk}_{T_{j,m}}(u)$ (resp., ${\rm
lk}_{T_{4,n}}(u)$) and hence is a triangulated $2$-sphere. If $u \in \ZZ_7$ then, ${\rm lk}_{M_{jmn}}(u)$ is the
union of the $2$-discs ${\rm lk}_{T_{j,m}}(u)$ and ${\rm lk}_{T_{4,n}}(u)$ and  ${\rm lk}_{T_{j,m}}(u) \cap {\rm
lk}_{T_{4,n}}(u) = {\rm lk}_{T}(u)$ is a cycle. Therefore, ${\rm lk}_{M_{jmn}}(u)$ is a triangulated 2-sphere in
this case also. So, $M_{jmn}= T_{j,m} \cup T_{4,n}$ is a triangulated 3-manifold without boundary for $1\leq
j\leq 3$ and $m, n\geq 0$. Similarly, $M_{jn}= T_{j} \cup T_{4,n}$ is a triangulated 3-manifold without boundary
for $1\leq j\leq 3$ and $n\geq 0$.

The elements $\alpha_1, \dots, \alpha_4$ in $\pi_1(\mathcal{T}, 0)$ mentioned in Remark \ref{R1} and Lemma
\ref{LT4n} satisfy the following\,:
\begin{eqnarray}
\alpha_1 = 0 & \mbox{in } \, \pi_1(T_1, 0)  \mbox{ and in } \,
\pi_1(T_{1,m}, 0) \, \mbox{ for all } \, m\geq 0, \nonumber\\
\alpha_2 = 0 & \mbox{in } \, \pi_1(T_2, 0)  \mbox{ and in } \,
\pi_1(T_{2,m}, 0) \, \mbox{ for all } \, m\geq 0, \nonumber\\
\alpha_1 +\alpha_2 = \alpha_3 = 0 & \mbox{in } \, \pi_1(T_3, 0)
\mbox{ and in } \,
\pi_1(T_{3,m}, 0) \, \mbox{ for all } \, m\geq 0, \nonumber\\
2\alpha_1 + \alpha_3 = 3\alpha_1 + \alpha_2 = \alpha_4 = 0 &
\mbox{in } \, \pi_1(T_{4,m}, 0)  \, \mbox{ for all } \, m\geq 0, \nonumber\\
7\alpha_1 - 2\alpha_4 = \alpha_1 - 2\alpha_2 =  \alpha_5 = 0 &
\mbox{in } \, \pi_1(T_{5,m}, 0)  \, \mbox{ for all } \, m\geq 0.
\end{eqnarray}

Therefore, by Seifert-Van Kampen theorem, we have
\begin{eqnarray}
\pi_1(T_{1,m} \cup T_{4,n}, 0) =  \pi_1(T_{1} \cup T_{4,n}, 0) &=&
\langle\{\alpha_1, \alpha_2\}; \, \{\alpha_1,
3\alpha_1+\alpha_2\}\rangle = \{0\}, \label{T14mn} \\
\pi_1(T_{2,m} \cup T_{4,n}, 0)  =  \pi_1(T_{2} \cup T_{4,n}, 0) &=&
\langle\{\alpha_1, \alpha_2\}; \, \{\alpha_2, 3\alpha_1+
\alpha_2\}\rangle \cong \ZZ_3, \label{T24mn} \\
\pi_1(T_{3,m} \cup T_{4,n}, 0)  =  \pi_1(T_{3} \cup T_{4,n}, 0)
&=& \langle\{\alpha_1, \alpha_2\}; \, \{\alpha_1+\alpha_2,
3\alpha_1+\alpha_2\}\rangle \nonumber \\
&=& \langle\{\alpha_1, \alpha_3\}; \, \{\alpha_3,
2\alpha_1+\alpha_3\}
\rangle \cong \ZZ_2, \label{T34mn} \\
\pi_1(T_{4,m} \cup T_{5,n}, 0)
&=& \langle\{\alpha_1, \alpha_2\}; \, \{3\alpha_1+\alpha_2,
\alpha_1 - 2\alpha_2\}\rangle \nonumber \\
&=& \langle\{\alpha_1, \alpha_4\}; \, \{7\alpha_1-2\alpha_4,
\alpha_4\} \rangle \cong \ZZ_7. \label{T45mn}
\end{eqnarray}
By \eqref{T14mn}, $T_{1} \cup T_{4,n}$, $T_{1,m} \cup T_{4,n}$ triangulate the $3$-sphere $S^{\hspace{.2mm}3}$
for $m, n\geq 0$. Part (i) follows from this since $g^2\colon T_{2} \cup T_{5,n} \to T_{1} \cup T_{4,4n}$,
$g^2\colon T_{2,m} \cup T_{5,n} \to T_{1, 4m} \cup T_{4,4n}$, $g\colon T_{3} \cup T_{6,n} \to T_{1} \cup
T_{4,2n}$, $g\colon T_{3,m} \cup T_{6,n} \to T_{1, 2m} \cup T_{4,2n}$ are isomorphisms.

By \eqref{T24mn}, $T_{2} \cup T_{4,n}$, $T_{2,m} \cup T_{4,n}$ triangulate the lens space $L(3,1)$ for $m, n\geq
0$. Part (ii) follows from this since $g^2\colon T_{3} \cup T_{5,n} \to T_{2} \cup T_{4,4n}$, $g^2\colon T_{3,m}
\cup T_{5,n} \to T_{2, 4m} \cup T_{4,4n}$, $g\colon T_{1} \cup T_{6,n} \to T_{2} \cup T_{4,2n}$, $g\colon T_{1,m}
\cup T_{6,n} \to T_{2, 2m} \cup T_{4,2n}$ are isomorphisms.

By \eqref{T34mn}, $T_{3} \cup T_{4,n}$, $T_{3,m} \cup T_{4,n}$ triangulate the projective space
$\mathbb{RP}^{\hspace{.2mm}3}$ for $m, n\geq 0$. Part (iii) follows from this since $g^2\colon T_{1} \cup T_{5,n}
\to T_{3} \cup T_{4,4n}$, $g^2\colon T_{1,m} \cup T_{5,n} \to T_{3, 4m} \cup T_{4,4n}$, $g\colon T_{2} \cup
T_{6,n} \to T_{3} \cup T_{4,2n}$, $g\colon T_{2,m} \cup T_{6,n} \to T_{3, 2m} \cup T_{4,2n}$ are isomorphisms.

By \eqref{T45mn}, $T_{4,m} \cup T_{5,n}$ triangulates the lens space $L(7,-2) = L(7,2) $ for $m, n\geq 0$. Part
(iv) follows from this since $g^2\colon T_{5,m} \cup T_{6,n} \to T_{4, 4m} \cup T_{5,4n}$, $g\colon T_{4,m} \cup
T_{6,n} \to T_{5, 2m} \cup T_{4,2n}$ are isomorphisms.

Part (v) follows by the similar arguments as in Remark \ref{R2}.
\end{proof}

Similar to $T_{4,n}$, we now present another sequence of solid tori $T_{7,n}$, $n\geq 0$.

\begin{eg} \label{EB7n}
{\rm For each non-negative integer $n\geq 0$, consider the $3$-dimensional simplicial complex ${\mathcal B}_{7,n}
= {\mathcal D}^{\hspace{.3mm}\prime} \cup {\mathcal D}^{\hspace{.3mm}\prime\prime}$ on the vertex set $\{e_0,
e_1, \dots, e_6\}$ $\cup\{e_0^{\hspace{.3mm}\prime}, \dots, e_4^{\hspace{.3mm}\prime}\} \cup \{a_{n,1}, a_{n,2},
a_{n,3}\} $ $\cup \{u_{7,n}, v_{7,n}, w_{7,n}, u_{7,n}^{\hspace{.2mm}\prime}\}$. Where
\begin{eqnarray*} 
{\mathcal D}^{\hspace{.3mm}\prime}  & \!=\! & \{w_{7,n}u_{7,n}^{\hspace{.3mm}\prime}e_0^{\hspace{.3mm}\prime}
e_{1}^{\hspace{.3mm}\prime}, \dots, w_{7,n}u_{7,n}^{\hspace{.3mm}\prime} e_3^{\hspace{.3mm}\prime}
e_{4}^{\hspace{.3mm}\prime}, w_{7,n}u_{7,n}^{\hspace{.3mm}\prime}e_0^{\hspace{.3mm}\prime}
e_{4}^{\hspace{.3mm}\prime}, w_{7,n}e_0^{\hspace{.3mm}\prime} e_{1}^{\hspace{.3mm}\prime}a_{n,2},
w_{7,n}e_1^{\hspace{.3mm}\prime}e_{2}^{\hspace{.3mm}\prime}a_{n,3}, \\
&& ~~ w_{7,n}e_2^{\hspace{.3mm}\prime}e_{3}^{\hspace{.3mm}\prime}e_{5},
w_{7,n}e_3^{\hspace{.3mm}\prime}e_{4}^{\hspace{.3mm}\prime}e_{6},
w_{7,n}e_0^{\hspace{.3mm}\prime}e_{4}^{\hspace{.3mm}\prime}e_{6},
w_{7,n}e_{0}^{\hspace{.3mm}\prime}a_{n,1}a_{n,2},
w_{7,n}e_{1}^{\hspace{.3mm}\prime}a_{n,2}a_{n,3},
w_{7,n}e_{2}^{\hspace{.3mm}\prime}a_{n,3}e_{5}, \\
&& ~~ w_{7,n}e_{3}^{\hspace{.3mm}\prime}e_{5}e_{6},
 w_{7,n}e_{0}^{\hspace{.3mm}\prime}e_{6}a_{n,1},
v_{7,n}w_{7,n}a_{n,1}a_{n,2}, v_{7,n}w_{7,n}a_{n,2}a_{n,3},
v_{7,n}w_{7,n}a_{n,3}e_{5}, \\
&& ~~ v_{7,n}w_{7,n}e_{5}e_{6}, v_{7,n}w_{7,n}e_{6}a_{n,1},
v_{7,n}a_{n,1}e_{2}e_{3}, v_{7,n}a_{n,2}e_{3}e_{4},
v_{7,n}e_{0}e_{4}e_{5}, v_{7,n}e_{0}e_{1}e_{5}, \\
&& ~~ v_{7,n}e_{1}e_{2}e_{6}, v_{7,n}a_{n,1}a_{n,2}e_{3},
v_{7,n}a_{n,2}a_{n,3}e_{4}, v_{7,n}a_{n,3}e_{4}e_{5},
v_{7,n}e_{1}e_{5}e_{6}, v_{7,n}a_{n,1}e_{2}e_{6}, \\
&& ~~ u_{7,n}v_{7,n}e_0e_{1}, \dots, u_{7,n}v_{7,n}
e_3e_{4}, u_{7,n}v_{7,n}e_0e_{4}\}, \\
{\mathcal D}^{\hspace{.3mm}\prime\prime} &\!=\!& \{e_{0}^{\hspace{.3mm}\prime}a_{n,1}a_{n,2}e_{3},
e_{1}^{\hspace{.3mm}\prime}a_{n,2}a_{n,3}e_{4}, e_{0}^{\hspace{.3mm}\prime}a_{n,1}e_{2}e_{3},
e_{0}^{\hspace{.3mm}\prime}a_{n,1}e_{2}e_{6}, e_{0}^{\hspace{.3mm}\prime}
e_{1}^{\hspace{.3mm}\prime}a_{n,2}e_{3},
e_{1}^{\hspace{.3mm}\prime}a_{n,2}e_{3}e_{4}, \\
&& ~~
e_{1}^{\hspace{.3mm}\prime}e_{2}^{\hspace{.3mm}\prime}a_{n,3}e_{4},
e_{2}^{\hspace{.3mm}\prime}a_{n,3}e_{4}e_{5}\}.
\end{eqnarray*}
Consider the equivalence relation `$\sim$' on the vertex set $V({\mathcal B}_{7,n})$ generated by
$u_{7,n}^{\hspace{.3mm}\prime} \sim u_{7,n}$, $e_i^{\hspace{.3mm}\prime} \sim e_i$ for $0\leq i \leq 4$. Let
\begin{eqnarray} \label{ET7n}
T_{7, n} = {\mathcal B}_{7,n}/\hspace{-1.5mm}\sim.
\end{eqnarray}
We identify the vertices $[e_i]$, $[u_{7,n}]$, $[v_{7,n}]$ and $[w_{7,n}]$ in $T_{7,n}$ with $i$, $u_{7,n}$,
$v_{7,n}$ and $w_{7,n}$ respectively. So, $V(T_{7,n}) = \ZZ_7\cup \{a_{n,1}, \dots, a_{n,3}\} \cup\{u_{7,n},
v_{7,n}, w_{7,n}\}$ for $n\geq 0$. We assume that $V(T_{7,n}) \cap V(T_{7,m}) = \ZZ_7 = V(T_{7,n}) \cap
V(T_{j,\hspace{.1mm}l})$ for all $\ell, m\neq n \geq 0$, $1\leq j\leq 6$. Then
\begin{eqnarray*} \label{EXT7n}
T_{7,n}  & \!\!=\!\! & \{uw01, \dots, uw34, uw04, wb01, wc12, w235,
w346,  w046, wab0, wbc1,  wc25, w356,  \\
&& wa06,   vwab, vwbc, vwc5, vw56,
 vwa6, va23, vb34, v045, v015, v126, vab3, vbc4, vc45,   \\
&& v156, va26, uv01, \dots, uv34, uv04, ab03, bc14, a023, a026,
b013, b134, c124, c245\},
\end{eqnarray*}
where $u= u_{7,n}$, $v= v_{7,n}$, $w= w_{7,n}$, $a=a_{n,1}$,
$b=a_{n,2}$,  $c=a_{n,3}$.} \qed
\end{eg}


\setlength{\unitlength}{2.7mm}

\begin{picture}(50,26.5)(-10,1)

\thicklines

\put(19,9){\line(1,0){20}} \put(19,9){\line(0,1){12}}
\put(19,21){\line(1,0){20}} \put(39,9){\line(0,1){12}}

\put(29,5){\line(-5,2){10}} \put(29,5){\line(5,2){10}}
\put(29,25){\line(-5,-2){10}} \put(29,25){\line(5,-2){10}}

\put(-7,9){\line(1,0){20}} \put(-7,9){\line(0,1){12}}
\put(-7,21){\line(1,0){20}} \put(13,9){\line(0,1){12}}

\put(3,5){\line(-5,2){10}} \put(3,5){\line(5,2){10}}
\put(3,25){\line(-5,-2){10}} \put(3,25){\line(5,-2){10}}

\thinlines


\put(31,15){\line(1,0){4}}

\put(23,9){\line(0,1){12}} \put(27,9){\line(0,1){12}}
\put(31,9){\line(0,1){12}} \put(35,9){\line(0,1){12}}

\put(19,9){\line(1,3){4}}
\put(23,9){\line(1,3){4}} \put(27,9){\line(2,3){8}}
\put(31,9){\line(2,3){8}} \put(31,15){\line(-2,3){4}}
\put(39,9){\line(-2,3){4}}

\put(29,5){\line(-3,2){6}} \put(29,5){\line(-1,2){2}}
\put(29,5){\line(3,2){6}} \put(29,5){\line(1,2){2}}

\put(29,25){\line(-3,-2){6}} \put(29,25){\line(-1,-2){2}}
\put(29,25){\line(3,-2){6}} \put(29,25){\line(1,-2){2}}


\put(-7,15){\line(1,0){20}}

\put(-3,9){\line(0,1){12}} \put(1,9){\line(0,1){12}}
\put(5,9){\line(0,1){12}} \put(9,9){\line(0,1){12}}

\put(-7,9){\line(2,3){8}} \put(-7,15){\line(2,3){4}}
\put(-3,9){\line(2,3){4}} \put(1,9){\line(2,3){8}}
\put(5,9){\line(2,3){8}} \put(5,15){\line(-2,3){4}}
\put(13,9){\line(-2,3){4}}

\put(3,5){\line(-3,2){6}} \put(3,5){\line(-1,2){2}}
\put(3,5){\line(3,2){6}} \put(3,5){\line(1,2){2}}

\put(3,25){\line(-3,-2){6}} \put(3,25){\line(-1,-2){2}}
\put(3,25){\line(3,-2){6}} \put(3,25){\line(1,-2){2}}




\put(17.5,9.5){\mbox{$e_{0}^{\hspace{.3mm}\prime}$}}
\put(21.5,9.5){\mbox{$e_{1}^{\hspace{.3mm}\prime}$}}
\put(25.5,9.5){\mbox{$e_{2}^{\hspace{.3mm}\prime}$}}
\put(29.5,9.5){\mbox{$e_{3}^{\hspace{.3mm}\prime}$}}
\put(33.5,9.5){\mbox{$e_{4}^{\hspace{.3mm}\prime}$}}
\put(39.4,9.5){\mbox{$e_{0}^{\hspace{.3mm}\prime}$}}

\put(29,15){\mbox{$e_{5}$}}
\put(35.8,14.8){\mbox{$e_{6}$}}

\put(17.5,20){\mbox{$e_{2}$}}
\put(23.4,20){\mbox{$e_{3}$}}
\put(28,20){\mbox{$e_{4}$}} \put(31.5,20){\mbox{$e_{0}$}}
\put(35.5,20){\mbox{$e_{1}$}} \put(39.4,20){\mbox{$e_{2}$}}

\put(30,25){\mbox{$u_{7,n}$}} \put(29.5,4){\mbox{$u_{7,
n}^{\hspace{.3mm}\prime}$}}

\put(22,4){\mbox{$\partial {\mathcal B}_{7,n}$}}


\put(-8.7,9.7){\mbox{$e_{0}^{\hspace{.3mm}\prime}$}}
\put(-4.7,9.7){\mbox{$e_{1}^{\hspace{.3mm}\prime}$}}
\put(-0.7,9.7){\mbox{$e_{2}^{\hspace{.3mm}\prime}$}}
\put(3.3,9.7){\mbox{$e_{3}^{\hspace{.3mm}\prime}$}}
\put(7.3,9.7){\mbox{$e_{4}^{\hspace{.3mm}\prime}$}}
\put(13.5,9.7){\mbox{$e_{0}^{\hspace{.3mm}\prime}$}}

\put(-9.5,15.5){\mbox{$a_{n,1}$}}
\put(-5.5,15.5){\mbox{$a_{n,2}$}}
\put(-1.5,15.5){\mbox{$a_{n,3}$}}
\put(2.8,15.5){\mbox{$e_{5}$}}
\put(7.3,15.5){\mbox{$e_{6}$}}
\put(10.5,15.5){\mbox{$a_{n,1}$}}

\put(-8.5,20){\mbox{$e_{2}$}}
\put(-2.6,20){\mbox{$e_{3}$}}
\put(2,20){\mbox{$e_{4}$}} \put(5.5,20){\mbox{$e_{0}$}}
\put(9.5,20){\mbox{$e_{1}$}} \put(13.4,20){\mbox{$e_{2}$}}

\put(4,25){\mbox{$u_{7,n}$}} \put(3.5,4){\mbox{$u_{7,
n}^{\hspace{.3mm}\prime}$}}

\put(-5,4){\mbox{$\partial {\mathcal D}^{\hspace{.3mm}\prime}$}}

\put(13,2){\mbox{\bf Figure\,2c}}
\end{picture}

\begin{lemma} \label{LT7n}
For $n \geq 0$, the simplicial complexes ${\mathcal B}_{7,n}$ and $T_{7, n}$ defined in Example $\ref{EB7n}$
satisfy the following.
\begin{enumerate}[{\rm (i)}]
\item ${\mathcal B}_{7,n}$ is a triangulated $3$-ball with boundary $\partial {\mathcal B}_{7,n}  = {\mathcal
E}_1 \cup {\mathcal E}_2 \cup {\mathcal E}_3$, where 
$$
{\mathcal E}_1 = \{u_{7, n}e_0e_{4}, u_{7, n}e_ie_{i + 1}
: 1\leq i \leq 3\}, {\mathcal E}_2 = \{u_{7, n}^{\hspace{.3mm}\prime} e_0^{\hspace{.3mm}\prime}
e_{4}^{\hspace{.3mm}\prime}, u_{7, n}^{\hspace{.3mm}\prime} e_i^{\hspace{.3mm}\prime} e_{i +
1}^{\hspace{.3mm}\prime} : 0\leq i \leq 3\},
$$ 
${\mathcal E}_3 = \{e_{i}^{\hspace{.3mm}\prime}e_{i+
1}^{\hspace{.3mm}\prime}e_{i + 3}, e_{i}^{\hspace{.3mm}\prime} e_{i + 2}e_{i + 3} : 0 \leq i \leq 3 \} \cup
\{e_{0}^{\hspace{.3mm}\prime} e_{4}^{\hspace{.3mm}\prime} e_{6}, e_{0}^{\hspace{.3mm}\prime} e_{2}e_{6},
e_0e_1e_5, e_0e_4e_5, e_1e_2e_6, e_1e_5e_6\}.$

\item $T_{7, n}$ is a triangulated solid torus with boundary $\partial T_{7, n} = \mathcal{T}$.

\item If $\alpha_1$, $\alpha_2$ are as in Remark $\ref{R1}$ and $\alpha_7 = [c_7] \in \pi_1(\mathcal{T}, 0)$,
where $c_7$ is the loop $012340$ in $\mathcal{T}$, then $\alpha_7 = 2\alpha_1 +\alpha_2$ and $\alpha_7 =0$ in
$\pi_1(T_{7,n}, 0)$.
\end{enumerate}
\end{lemma}

\begin{proof}
Consider the eight triangulated 3-balls

${\mathcal D}_1 := \{w_{7,n}\}\ast \{u_{7,n}^{\hspace{.2mm}\prime}
e_0^{\hspace{.2mm}\prime} e_{1}^{\hspace{.2mm}\prime}, \dots, u_{7,n}^{\hspace{.2mm}\prime}
e_3^{\hspace{.2mm}\prime} e_{4}^{\hspace{.2mm}\prime}, ~ u_{7, n}^{\hspace{.2mm}\prime} e_0^{\hspace{.2mm}\prime}
e_{4}^{\hspace{.2mm}\prime}$,
$~e_0^{\hspace{.3mm}\prime}e_{1}^{\hspace{.3mm}\prime}a_{n,2}, ~e_1^{\hspace{.3mm}\prime}
e_{2}^{\hspace{.3mm}\prime}a_{n,3}, ~e_2^{\hspace{.3mm}\prime}e_{3}^{\hspace{.3mm}\prime}e_{5},
~e_3^{\hspace{.3mm}\prime}e_{4}^{\hspace{.3mm}\prime}e_{6},$ $ ~e_0^{\hspace{.3mm}\prime}
e_{4}^{\hspace{.3mm}\prime}e_{6}, ~e_{0}^{\hspace{.3mm}\prime}a_{n,1}a_{n,2}, ~e_{1}^{\hspace{.3mm}\prime}
a_{n,2}a_{n,3}, ~e_{2}^{\hspace{.3mm}\prime} a_{n,3}e_{5}, ~e_{3}^{\hspace{.3mm}\prime} e_{5}e_{6},
~e_{0}^{\hspace{.3mm}\prime} e_{6}a_{n,1}\}$, \newline

${\mathcal D}_2 := \{v_{7,n}\}\ast \{a_{n,1}e_{2}e_{3}, ~a_{n,2}e_{3}e_{4}, ~e_{0}e_{4}e_{5}, ~e_{0}e_{1}e_{5},
~e_{1}e_{2}e_{6}, ~a_{n,1}a_{n,2}e_{3}, ~a_{n,2}a_{n,3}e_{4}, ~a_{n,3}e_{4}e_{5},$ $ ~e_{1}e_{5}e_{6}, 
~a_{n,1}e_{2}e_{6}, ~w_{7,n}a_{n,1}a_{n,2}, ~w_{7,n}a_{n,2}a_{n,3}, ~w_{7,n}a_{n,3}e_{5}, ~w_{7,n}e_{5}e_{6},
~w_{7,n}e_{6}a_{n,1}\}$,

${\mathcal D}_3 := \{u_{7,n}\}\ast \{v_{7,n}e_0e_{1}$, $~\dots, v_{7,n}e_3e_{4}, ~v_{7,n}e_0e_{4}\}$,
\quad ${\mathcal D}_4 := \{e_{0}^{\hspace{.3mm}\prime} a_{n,1}a_{n,2}e_{3}\}$,\newline

${\mathcal D}_5 := \{e_{1}^{\hspace{.3mm}\prime} a_{n,2}a_{n,3}e_{4}\}$,
\quad ${\mathcal D}_6 := \{e_{0}^{\hspace{.3mm}\prime}\}\ast \{a_{n,1}e_{2}e_{3}$, $~a_{n,1}e_{2}e_{6}\}$,\newline

${\mathcal D}_7 := \{e_{0}^{\hspace{.3mm}\prime}e_{1}^{\hspace{.3mm}\prime}a_{n,2}e_{3},
~e_{1}^{\hspace{.3mm}\prime}a_{n,2}e_{3}e_{4}\}$,
~ and ~${\mathcal D}_8 := \{e_{1}^{\hspace{.3mm}\prime}e_{2}^{\hspace{.3mm}\prime}a_{n,3}e_{4},
~e_{2}^{\hspace{.3mm}\prime}a_{n,3}e_{4}e_{5}\}$.

Clearly, ${\mathcal B}_{7,n} = {\mathcal D}_1 \cup \cdots \cup
{\mathcal D}_{8}$. Observe that ${\mathcal D}_i \cap ({\mathcal D}_1 \cup \cdots \cup {\mathcal D}_{i-1})$ is a
triangulated 2-disc for $2 \leq i \leq 8$. This implies that ${\mathcal B}_{7,n} = {\mathcal D}_1 \cup \cdots
\cup {\mathcal D}_{8}$ is a triangulated 3-ball. This proves the first part of (i). The second part of (i)
follows from the definition of ${\mathcal B}_{7,n}$.

There is no vertex in ${\mathcal B}_{7,n}$ which is a common neighbour of $e_{i}^{\hspace{.3mm}\prime}$ and
$e_{i}$ for $0\leq i \leq 3$ and there is no vertex in ${\mathcal B}_{7,n}$ which is a common neighbour of
$u_{7,n}^{\hspace{.3mm}\prime}$ and $u_{7,n}$. This implies that the quotient complex $T_{7,n}$ is a triangulated
$3$-manifold with boundary. Clearly, the boundary $\partial T_{7, n}$ is the quotient ${\mathcal
E}_3/\hspace{-1.5mm}\sim \, = \mathcal{T}$. Since the space $|T_{7,n}|$ can be obtained from $|{\mathcal
B}_{7,n}|$ by identifying the 2-disc $|{\mathcal E}_1|$ with the 2-disc $|{\mathcal E}_2|$ (via the simplicial
isomorphism from ${\mathcal E}_1$ to ${\mathcal E}_2$ given by $u_{7,n} \mapsto u_{7,n}^{\hspace{.3mm}\prime}$,
$e_i \mapsto e_i^{\hspace{.3mm}\prime} $ for $0\leq i \leq 4$), $|T_{7,n}|$ is a solid torus. This proves (ii).

Now, $012340 \simeq 0162340 \simeq 01602340 \simeq 0160340$ (see Figure 2a). Thus, $\alpha_7= [c_7] = [012340] =
[0160] + [0340] = \alpha_1 + (\alpha_1 + \alpha_2) = 2\alpha_1 + \alpha_2$. Let ${\mathcal E}$ be the 2-disc in
$T_{7,n}$ corresponding to the 2-disc ${\mathcal E}_1$ in ${\mathcal B}_{7,n}$. Then, $c_7$ is the boundary of
${\mathcal E}$ and hence $\alpha_7 = [c_7] =0$ in $\pi_1(T_{7,n}, 0)$.
\end{proof}

\begin{eg} \label{ET89n}
{\rm Let $A_n = \{a_{n,1}, a_{n,2}, a_{n,3}\}$, $U_7 = \{u_{7,n} : n\geq 0\}$, $V_7 = \{v_{7,n} : n\geq 0\}$,
$W_7 = \{w_{7,n} : n\geq 0\}$ be as in Example \ref{EB7n}. Let $A = \cup_{n\geq 0}A_n$.
Take new (disjoint) sets of vertices
\begin{align*}
B &= \cup_{n\geq 0}B_n = \cup_{n\geq 0}\{b_{n,1}, b_{n,2}, b_{n,3}\}, \quad C = \cup_{n\geq 0}C_n = \cup_{n\geq
0}\{c_{n,1}, c_{n,2}, c_{n,3}\}, \\
U_j &= \{u_{j,n} \, : n\geq 0\}, \quad V_j = \{v_{j,n} \, : n\geq 0\}, \quad W_j = \{w_{j,n} \, : n\geq 0\}
~\mbox{for} ~8 \leq j\leq 9.
\end{align*}

Let $f$ and $g$ be as in \eqref{maps_fg456}. Consider the
following extensions of $f$ and $g$ (also denoted by $f$ and $g$ respectively) $ {f},  {g} \, \colon \ZZ_7 \sqcup
(Q\sqcup R\sqcup S) \sqcup (A\sqcup B\sqcup C) \sqcup (\sqcup_{i=1}^9(U_{i} \sqcup V_{i} \sqcup W_{i})) \to \ZZ_7
\sqcup (Q\sqcup R\sqcup S) \sqcup (A\sqcup B\sqcup C) \sqcup (\sqcup_{i=1}^9(U_{i} \sqcup  V_{i} \sqcup W_{i}))$
as\,:
\begin{eqnarray} \label{maps_fg789}
&&  {f}(u_{j,n}) = u_{j,n+1},  {f}(v_{j,n}) = v_{j,n+1}, {f}(w_{j,n}) = w_{j,n+1}, \, \mbox{ for } 7\leq j \leq
9, n\geq 0,  \nonumber \\
&&  {f}(a_{n,i}) = a_{n+1,i},  {f}(b_{n,i}) =
b_{n+1,i},  {f}(c_{n,i}) = c_{n+1,i}, \, \mbox{ for }
1\leq i\leq 3,  n\geq 0 \, \mbox{ and } \nonumber \\
&&  {g}(u_{7,n}) = u_{8,2n}, \,  {g}(u_{8,n}) =
u_{9,2n}, \,  {g}(u_{9,n}) = u_{7,2n}, \,  {g}(v_{7,n})
= v_{8,2n}, \,  {g}(v_{8,n}) = v_{9,2n},  \nonumber \\
&&  {g}(v_{9,n}) = v_{7,2n}, \,  {g}(w_{7,n}) = w_{8,2n}, \,
{g}(w_{8,n}) = w_{9,2n}, \, {g}(w_{9,n}) = w_{7,2n}, \,
\mbox{ for } \, n \geq 0, \nonumber \\
&&  {g}(a_{n,i}) = b_{2n,i}, \,  {g}(b_{n,i}) = c_{2n,i}, \,
{g}(c_{n,i}) = a_{2n,i}, \, 1\leq i\leq 3,  n\geq 0.
\end{eqnarray}
Then, $ {g} \circ  {f} = { {f}}^2 \circ  {g}$. Let
\begin{eqnarray} \label{T89n}
T_{8, 0} :=  {g}(T_{7,0}), &&
T_{9, 0} := { {g}}^{2}(T_{7,0}), \nonumber \\
T_{j,n} := { {f}}^{\hspace{.2mm}n}(T_{j,0}) && \mbox{for } \, n
\geq 1 \, \mbox{ and } \, j= 8, 9.
\end{eqnarray}
For $7\leq j\leq 9$ and $n \geq 0$, the number of vertices in $T_{j, n}$ is $7+3+3=13$.} \qed
\end{eg}

\begin{lemma} \label{LT89n}
For $n \geq 0$, $T_{7,n}$ defined in Example $\ref{EB7n}$ and $T_{8, n}$, $T_{9, n}$ defined in
Example $\ref{ET89n}$ satisfy the following.
\begin{enumerate}[{\rm (i)}]
\item $T_{i, m}\cong T_{j, n}$ for $7\leq i, j\leq 9$ and $m, n\geq 0$.

\item $T_{j, n}$ is a triangulated solid torus with boundary $\partial T_{j, n} = \mathcal{T}$ for $j = 8, 9$,
$n\geq 0$.

\item Let $\alpha_1$, $\alpha_2$ be as in Remark $\ref{R1}$ and $\alpha_8 = [c_8], \alpha_9 = [c_9] \in
\pi_1(\mathcal{T}, 0)$, where $c_8=016420$, $c_9=025140$ are loops in $\mathcal{T}$. Then, $\alpha_8 = \alpha_1 -
\alpha_2$, $\alpha_9 = \alpha_1 + 2 \alpha_2$ and $\alpha_j =0$ in $\pi_1(T_{j,n}, 0)$ for $j=8,9$, $n\geq 0$.
\end{enumerate}
\end{lemma}

\begin{proof}
From the definition in \eqref{T89n}, $T_{i,m} \cong T_{i,0}\cong T_{7,0} \cong T_{j,0} \cong T_{j,n}$ for $7\leq
i, j\leq 9$ and $m, n\geq 0$. This proves part (i).

By part (i), $T_{j, n}\cong T_{7,0}$. Therefore, by Lemma \ref{LT7n} (ii), $T_{j,n}$ triangulates the solid
torus. Since $f(\mathcal{T}) = \mathcal{T} = g(\mathcal{T})$ and $\partial T_{7,0} =\mathcal{T}$, it follows that
$\partial T_{j, n} = \mathcal{T}$. This proves (ii).

Observe that $g(c_7) = g(012340) = 024610 = \overline{c}_8$ and $g(c_8) = c_9$, where $\overline{c}_8$ denotes 
the loop $c_8$ with opposite orientation. Therefore, $g_{\ast}(\alpha_7) =
-\alpha_8$ and $g_{\ast}(\alpha_8) = \alpha_9$, where $g_{\ast} \colon \pi_1(\mathcal{T}, 0) \to
\pi_1(\mathcal{T}, 0)$ is induced by the map $g$. Therefore, $\alpha_8 = - g_{\ast}(\alpha_7) =
-g_{\ast}(2\alpha_1 +\alpha_2) = -2\alpha_2 + (\alpha_1 + \alpha_2) = \alpha_1- \alpha_2$ and hence $\alpha_9 =
g_{\ast}(\alpha_8) = g_{\ast}(\alpha_1 -\alpha_2) = \alpha_2 + (\alpha_1 +\alpha_2) = \alpha_1 + 2\alpha_2$. This
proves the first part of (iii).

Since $g \colon T_{7,0} \to T_{8,0}$ is an isomorphism, $g_{\ast} \colon \pi_1(T_{7,0}, 0) \to \pi_1(T_{8,0}, 0)$
is an isomorphism. Therefore, $\alpha_8 = g_{\ast}(-\alpha_7) = g_{\ast}(0) =0$ in $\pi_1(T_{8,0},0)$. (In fact,
$c_8$ is the boundary of the 2-disc $g({\mathcal E})$ in $T_{8,0}$, where ${\mathcal E}$ is the 2-disc defined in
the proof of Lemma \ref{LT7n} (iii).) Now, $f^n \colon T_{8,0} \to T_{8,n}$ is an isomorphism and $f^n(c_8) =
c_8$. This implies that $\alpha_8 = [c_8] =0$ in $\pi_1(T_{8,n},0)$. Similarly, $\alpha_9 = [c_9] =0$ in
$\pi_1(T_{9, n},0)$. This completes the proof of (iii).
\end{proof}

\begin{cor} \label{CT7n}
For $n\geq 0$, the solid tori $T_1, T_2, T_3$ defined in \eqref{ET123}, $T_{1,n}, T_{2,n}, T_{3,n}$ defined in
\eqref{ET123n}, $T_{4,n}$ defined in \eqref{ET4n}, $T_{5,n}$, $T_{6,n}$ in \eqref{T56n}, $T_{7,n}$ defined in
\eqref{ET7n}, $T_{8,n}$, $T_{9,n}$ defined in \eqref{T89n} satisfy the following.
\begin{enumerate}[{\rm (i)}]
\item $T_{i} \cup T_{j,n}$, $T_{i,m} \cup T_{j,n}$, $T_{k,m} \cup T_{\ell,n}$ triangulate the $3$-sphere
$S^{\hspace{.2mm}3}$ for \, $m, n\geq 0$, \, $(i, j) \in \{(1,7), (1,8), (2, 8), (2, 9), (3, 7), (3,9)\}$, $(k,
l) \in \{(4,7), (5,8), (6, 9)\}$.

\item $T_{i} \cup T_{j,n}$, $T_{i,m} \cup T_{j,n}$ triangulate the projective space
$\mathbb{RP}^{\hspace{.2mm}3}$ for \, $m, n\geq 0$, \, $(i, j) \in \{(1,9)$, $(2, 7), (3,8)\}$.

\item $T_{4,m} \cup T_{8,n}$, $T_{5,m} \cup T_{9,n}$, $T_{6,m} \cup T_{7,n}$ triangulate the lens space $L(4,1)$
for $m, n\geq 0$.

\item $T_{4,m} \cup T_{9,n}$, $T_{5,m} \cup T_{7,n}$, $T_{6,m} \cup T_{8,n}$ triangulate the lens space $L(5,2)$
for $m, n\geq 0$.

\item[{\rm (v)}] $T_{7,m} \cup T_{8,n}$, $T_{8,m} \cup T_{9,n}$, $T_{9,m} \cup T_{7,n}$ triangulate the lens
space $L(3,1)$ for $m, n\geq 0$.
\item $T_{j,m} \cup T_{j,n}$ triangulates $S^{\hspace{.2mm}2} \times S^1$ for $7 \leq j \leq 9$ and $m \neq n
\geq 0$.
\end{enumerate}
\end{cor}

\begin{proof}
By similar arguments as in the proof of Corollary \ref{CT4n}, all the simplicial complexes $T_{i} \cup T_{j,n}$,
$T_{k,m} \cup T_{\ell,n}$ in the statement are triangulated 3-manifolds without boundary.

(i), (ii) and (vi) follow by similar arguments as in Corollary \ref{CT4n}.

The elements $\alpha_4, \alpha_7, \alpha_8, \alpha_9 \in \pi_1(\mathcal{T}, 0)$ in Lemmas \ref{LT4n}, \ref{LT7n},
\ref{LT89n} satisfy the following\,:
\begin{eqnarray}
3\alpha_1 + \alpha_2 = \alpha_4 = 0 &&
\mbox{in } \, \pi_1(T_{4,m}, 0)  \, \mbox{ for all } \, m\geq 0, \nonumber \\
2\alpha_1 + \alpha_2 = \alpha_7 = 0 &&
\mbox{in } \, \pi_1(T_{7,m}, 0)  \, \mbox{ for all } \, m\geq 0, \nonumber \\
\alpha_1 - \alpha_2 = \alpha_8 = 0 &&
\mbox{in } \, \pi_1(T_{8,m}, 0)  \, \mbox{ for all } \, m\geq 0, \nonumber \\
\alpha_1 + 2\alpha_2 =  \alpha_9 = 0 &&
\mbox{in } \, \pi_1(T_{9,m}, 0)  \, \mbox{ for all } \, m\geq 0.
\end{eqnarray}
Here $\alpha_1, \alpha_2$ are as in Remark \ref{R1}. Then, by Seifert-Van Kampen theorem, we have
\begin{eqnarray}
\pi_1(T_{4,m} \cup T_{8,n}, 0)  &=&
\langle\{\alpha_1, \alpha_2\}; \, \{3\alpha_1+\alpha_2,
\alpha_1-\alpha_2\}\rangle  \nonumber \\
&=& \langle\{\alpha_1, \alpha_8\}; \{4\alpha_1-\alpha_8,
\alpha_8\}\rangle \cong\ZZ_4, \label{T48mn} \\
\pi_1(T_{4,m} \cup T_{9,n}, 0)  &=&
\langle\{\alpha_1, \alpha_2\}; \, \{3\alpha_1+\alpha_2,
\alpha_1+2\alpha_2\}\rangle  \nonumber \\
&=& \langle\{\alpha_1, \alpha_4\}; \{4\alpha_1, 5\alpha_1
-2\alpha_4\}\rangle \cong\ZZ_5, \label{T49mn} \\
\pi_1(T_{7,m} \cup T_{8,n}, 0)
&=& \langle\{\alpha_1, \alpha_2\}; \, \{2\alpha_1+\alpha_2,
\alpha_1-\alpha_2\}\rangle \nonumber \\
&=& \langle\{\alpha_1, \alpha_7\}; \, \{\alpha_7,
3\alpha_1-\alpha_7\} \rangle \cong \ZZ_3. \label{T78mn}
\end{eqnarray}

By \ref{T48mn}, $T_{4,m} \cup T_{8,n}$ triangulates the lens space $L(4,1)$ for $m, n\geq 0$. (iii) follows from
this since $g^2\colon T_{5,m} \cup T_{9,n} \to T_{4, 4m} \cup T_{8,4n}$, $g\colon T_{6,m} \cup T_{7,n} \to T_{4,
2m} \cup T_{8,2n}$ are isomorphisms.

By \ref{T49mn}, $T_{4,m} \cup T_{9,n}$ triangulates the lens space $L(5,2)$ for $m, n\geq 0$. (iv) follows from
this since $g^2\colon T_{5,m} \cup T_{7,n} \to T_{4, 4m} \cup T_{9,4n}$, $g\colon T_{6,m} \cup T_{8,n} \to T_{4,
2m} \cup T_{9,2n}$ are isomorphisms.

By \ref{T78mn}, $T_{7,m} \cup T_{8,n}$ triangulates the lens space $L(3,1)$ for $m, n\geq 0$. (v) follows from
this since $g^2\colon T_{8,m} \cup T_{9,n} \to T_{7, 4m} \cup T_{8,4n}$, $g\colon T_{7,m} \cup T_{9,n} \to T_{8,
2m} \cup T_{7,2n}$ are isomorphisms.
\end{proof}
In the rest of this paper, we denote the vertex set of $ \mathcal{T}$, $T_i$, $T_{j, m}$ by
$ V(\mathcal{T})$, $V(T_i)$, $V(T_{j, m})$ respectively.

\section{Construction of equilibrium triangulations}\label{qmtri}

In \cite{BR}, Buchstaber and Ray showed that quasitoric manifolds have smooth structures. So, by results of Cairns
in \cite{Ca}, any quasitoric manifold has a triangulation. But no explicit triangulations of a $4$-dimensional
quasitoric manifold except $\mathbb{\CP}^2$ and $S^2 \times S^2$ are known until now. Let $M^4$ be a quasitoric 4-manifold
over an $m$-gonal $2$-polytope $P$. By Proposition \ref{cla4}, $M^4$ is homeomorphic to $k_1 \mathbb{CP}^2 \# k_2
\bar{\mathbb{CP}^2} \# k_3 (S^2 \times S^2)$ for some $k_1, k_2, k_3 \geq 0$ where the connected sum is
non-equivariant. From Orlik and Raymond \cite{OR} we have
\begin{align*}
k_1+k_2+2k_3+2=m.
\end{align*}
From Proposition \ref{sigma123}, we know that $k_1 \mathbb{CP}^2 \# k_2 \bar{\mathbb{CP}^2} \# k_3 (S^2 \times
S^2)$ has a triangulation whose number of vertices is $4k_1+4k_2+6k_3 +5\leq 4(k_1+k_2+2k_3) +5 \leq
7(k_1+k_2+2k_3) + 2 = 7m-12$. Here, we are interested to construct triangulations of $M^4$ as quasitoric
manifold. 

Generalizing the definitions in \cite{BK}, we introduce equilibrium set, zones of influence and
equilibrium triangulations for all quasitoric manifolds. 
Let 
$$ 
D^{2n}= \{(z_1, \ldots, z_{n}) \in \CC^{n} : \Sigma_{i=1}^{n} |z_i|^2 \leqslant 1\}
$$ 
be the
closed unit $2n$-ball. The boundary of $D^{2n}$ is the unit sphere $S^{2n-1}$. There is a natural
action of $ \mathbb{T}^n$ on $D^{2n}$ (resp., on $S^{2n-1}$) given by 
$$
(t_1, \dots, t_n)\cdot(w_1, \dots, w_n) = (t_1w_1, \dots, t_nw_n),
$$
where $(t_1, \dots, t_n)\in \TT^n$ and $(w_1, \dots, w_n)\in D^{2n}$ (resp., $\in S^{2n-1}$).

\begin{defn}
{\rm Let $\pi : M \to P$ be a quasitoric $2n$-manifold over a simple $n$-polytope $P$.
If $a$ is the center of mass of $P$, then $\pi^{-1}(a)$ is said to be 
the {\em equilibrium set} of $M$. So, the equilibrium set is the 
$\TT^n$-orbit of a point $x \in \pi^{-1}(a)$ and homeomorphic to the torus $\TT^n$.}
\end{defn}

\begin{defn}
{\rm Let $\pi : M \to P$ be a quasitoric $2n$-manifold. A collection $\{Z_1, \dots, Z_m\}$
of submanifolds of $M$ is said to be  a set of {\em zones of influence} if (i) $M = Z_1 \cup \dots \cup Z_m$,
(ii) $Z_i\cap Z_j = (\partial Z_i)\cap (\partial Z_j)$ for $i\neq j$, (iii) each 
$Z_i$ contains the equilibrium set, (iv) each $\pi(Z_i)$ contains exactly one vertex of $P$, and
(v) each $Z_i$ is $\TT^n$-invariant and is $\TT^n$-equivariantly diffeomorphic to the closed unit 
ball $D^{2n}$. We also say that $Z_1\cup\cdots\cup Z_m$ is a {\em decomposition of $M$ into $m$ zones of influence}. }
\end{defn}

\begin{defn}
{\rm Let $X$ be a triangulation of a quasitoric $2n$-manifold $M$. Assume that $f : |X| \to M$ is a
homeomorphism. The triangulation $X$ is said to be an {\em equilibrium triangulation} if
$X= X_1 \cup\cdots\cup X_m$ for some subcomplexes   $X_1, \dots, X_m$ of $X$ such that
$\{f(|X_1|), \dots, f(|X_m|)\}$ is a set of zones of influence. So, an equilibrium triangulation
of $M$ is a triangulation which triangulates $M$ together with a set of zones of influence of $M$. }
\end{defn}


Following the Lecture series \cite{BP} of Buchstaber and Panov, we discuss the rectangular subdivision of
$2$-polytopes. Let $P \subset \RR^2$ be a $2$-polytope with $m$ vertices $V_1, V_2, \ldots, V_m$ and $m$ edges
$E_1=V_1V_2, \ldots, E_m =V_mV_1$. We will denote such $P$ by $V_1V_2 \cdots V_mV_1$. Let $O$ be an interior point of $P$. 
(For our purpose, we take $O$ is the center of mass of $P$.) 
For $1\leq i\leq m$, let $C_{i}$ be an interior point of the edge $E_i$.
 Let $I_{i}$ be the rectangle with
vertices $V_i, C_{i}, O, C_{{i-1}}$. Observe that $P = \cup_{i=1}^m I_{i}$, $I_{i} \cap I_{i+1} = C_{i}O$ and
$I_{i} \cap I_{j} = O$ if $|i-j| \geq 2$, $1 \leq i, j \leq m$. So, we have constructed a rectangular subdivision
of $P$ with $m$ rectangles. We denote  this cell complex by $C(P)$.
\begin{figure}[ht]
\centerline{ \scalebox{0.60}{\input{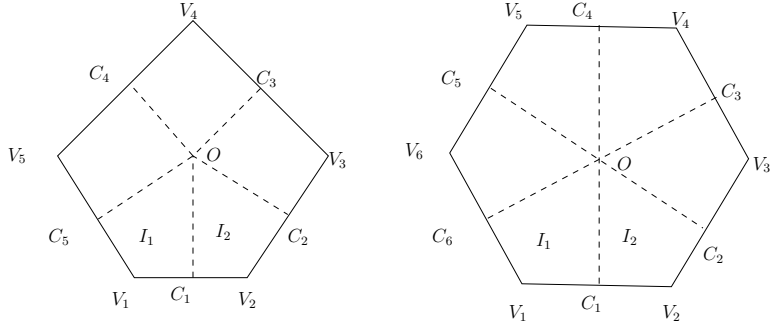} } }
\caption {Examples of rectangular subdivision.}
\label{egch005c}
\end{figure}

 Let $\pi : M(P, \xi) \to P$ be a quasitoric 4-manifold over a 2-polytope $P$ with
the characteristic function 
$$
\xi : \{E_1, \ldots, E_m\} \to \ZZ^2.
$$ 
Denote $\xi(E_i)$ by $(\xi_{i_1}, \xi_{i_2})
\in \ZZ^2$ for $i = 1, \ldots, m$. Let $C(P)$ be a cubical subdivision of $P$ as describe above.
Then, $\pi^{-1}(I_{i})$ is equivariantly homeomorphic to the $4$-ball $D^4$. The boundary of
$\pi^{-1}(I_{i})$ is $\pi^{-1}(C_{i}OC_{{i-1}})$ which is equivariantly homeomorphic to $S^3$. The set
$$
\pi^{-1}(I_{i}) \cap \pi^{-1}(I_{i+1})= \pi^{-1}(I_{i} \cap I_{i+1}) = \pi^{-1}(C_{i}O)
$$ 
is equivariantly
homeomorphic to the solid torus $S^1 \times D^2$. If $|i-j| \geqslant 2$, then the set $$\pi^{-1}(I_{i}) \cap
\pi^{-1}(I_{j})= \pi^{-1}(I_{i} \cap I_{j}) = \pi^{-1}(O)$$ is the torus $\TT^2$ and the set $\pi^{-1}(C_{i}
OC_{j})$ is equivariantly homeomorphic to the lens space $L(p_{ij},q_{ij})$, where
\begin{align*}
p_{ij}= \det ((\xi_{j_1},\xi_{j_2}), (\xi_{i_1}, \xi_{i_2}))~ \mbox{ and }~ q_{ij}= \det ((r, s), (\xi_{j_1},
\xi_{j_2}))
\end{align*}
for some $(r,s) \in \ZZ^2$ with $\det ((r, s), (\xi_{i_1}, \xi_{i_2})) = 1$. Here $q_{ij}$ is independent of the
choice of $(r,s)$ (cf. \cite{OR}). 

From the construction of quasitoric
manifolds it is clear that if $p_{ij}=0$, then $\pi^{-1}(C_{i}OC_{j})= S^1 \times S^2$ and if $|p_{ij}|=1$ then
$\pi^{-1}(C_{i}OC_{j}) = S^3$. Note that $\pi^{-1}(C_{i}OC_{j})$ is the manifold obtained by gluing the solid
tori $S^1 \times D^2$ and $S^1 \times D^2$ via the map $\varphi_{(p_{ij}, q_{ij})}$ on the boundary. So, we
construct a handle body decomposition of the manifold $M(P, \xi)$ using the characteristic function and the
cubical subdivision of $P$, namely, $M(P, \xi) = \cup_1^m \pi^{-1}(I_{i})$. We call this decomposition a cubical
subdivision of $M(P, \xi)$.


A 10-vertex equilibrium triangulation of $\mathbb{CP}^2$ is discussed in \cite{BK}. With the help of  
rectangular subdivisions of a 2-polytope $P$, we construct some equilibrium triangulations of $M(P, \xi)$ 
in the remaining sections. 


\section{Equilibrium triangulations of Hirzebruch surfaces:}\label{m4}

Let $\pi : M \to P$ be a 4-dimensional quasitoric manifold over a rectangle $P=V_1V_2V_3V_4V_1$
and $\xi: \{V_1V_2, V_2V_3, V_3V_4, V_1V_4\} \to \ZZ^2$ be a characteristic function on $P$.
By the definition of characteristic function, we may assume $\xi(V_1V_2)=(-1, 0)$ and $\xi(V_1V_4)=(0, -1)$. So, $\xi(V_2V_3)=(l, 1)$
and $\xi(V_3V_4)=(1, k)$ (up to sign) for some $\ell, k \in \ZZ$.
Then, $k\ell-1=\pm 1$. If $k\ell-1=-1$, then $k\ell=0$. So, either $k=0$ or $\ell=0$. On the other hand, if $k\ell-1=1$, then
$k\ell=2$. So, either $k=\pm 1$ and $\ell=\pm 2$ or $k=\pm 2$ and $\ell=\pm 1$. In this section, we consider the rectangular
subdivision of $P$ as in Figure \ref{egch002b}. When $\ell=0$, we denote $M$ by $M_k$. 
For the characteristic function of
Figure \ref{egch002b}, we have 
\begin{align*}
\pi^{-1}(C_1OC_3) &=
L(-k, 1)= \bar{L(k, 1)} \cong L(k, 1), \\
\pi^{-1}(C_1OC_2) &\cong \pi^{-1}(C_2OC_3) \cong \pi^{-1}(C_3OC_4)
 \cong \pi^{-1}(C_1OC_4) \cong S^3, \mbox{ and} \\
\pi^{-1}(C_2OC_4) &= L(0,1) \cong S^2 \times S^1.
\end{align*}

\begin{figure}[ht]
\centerline{\scalebox{0.90}{\input{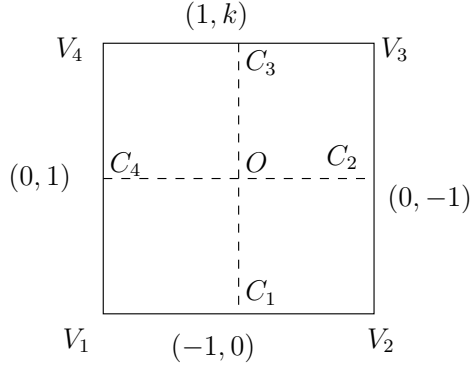} } }
\caption {Characteristic functions of rectangle.}
\label{egch002b}
\end{figure}


 We recall the number of vertices of solid tori of Section \ref{tortri}.
 By Example \ref{tst1}, $f_0(T_i) = f_0(\mathcal{T})=7$ if $1 \leq i \leq 3$.
By Example \ref{ETjn}, $f_0(T_{j,n})= f_0(\mathcal{T})+2$ if $1 \leq j \leq 3, 0 \leq n \leq 6$ and
$f_0(T_{j,n})= f_0(\mathcal{T})+3$ if $1 \leq j \leq 3, n \geq 7$. By Example \ref{ET56n}, $f_0(T_{j,n})=f_0(\mathcal{T}) +10$ 
if $4 \leq j \leq 6, n \geq 0$. By Example \ref{ET89n}, $f_0(T_{j,n}) = f_0(\mathcal{T}) +6$ if $7 \leq j \leq 9, n \geq 0$.

By our assumption for any $k \in \ZZ$, $M_k$ is a smooth projective surface, called {\em Hirzebruch surface},
see Example \ref{square}.
In this subsection, we construct equilibrium triangulation of Hirzebruch surface $M_k$ for $|k| \leq 4$. 

\begin{eg}\label{egr1}
{\rm For $k= 1$, $L(-k, 1) \cong S^3$. Consider triangulations $T_1$ for $\pi^{-1}(C_1O)$, $T_{2}$ for
$\pi^{-1}(C_2 O)$, $T_{3}$ for $\pi^{-1}(C_3O)$ and $T_{2, 7}$ for $\pi^{-1}(C_4O)$. Then, by Lemma \ref{LT123n},
$T_1 \cup T_2$, $T_2 \cup T_3$, $T_3 \cup T_{2, 7}$, $T_{2, 7} \cup T_1$, $T_1 \cup T_3$ and $T_2 \cup T_{2, 7}$
triangulate $\pi^{-1}(C_1OC_2)$, $\pi^{-1}(C_2OC_3)$, $\pi^{-1}(C_3OC_4)$, $\pi^{-1}(C_1OC_4)$, $\pi^{-1}(C_1OC_3)$ and
$\pi^{-1}(C_2OC_4)$ respectively. Consider the cone $\pi^{-1}(I_{i})$ over $\pi^{-1}(C_{i-1}OC_{i})$ with apex $V_i$ for
$1 \leq i \leq 4$ (addition is modulo $4$). This implies
$$ 
X_{\ref{egr1}}= (V_2 \ast (T_1 \cup T_2)) \cup (V_3 \ast (T_2 \cup T_3)) \cup (V_4 \ast (T_3 \cup T_{2,7})) \cup (V_1 \ast
(T_{2, 7} \cup T_1)) 
$$
is a triangulation of $M_1$
with $f_0(X_{\ref{egr1}}) = 7+3+4 =14$. Since the minimum
number of vertices require for a triangulation of $S^1 \times S^2$ is $10$ (cf. \cite{BD}), any triangulation of
$\pi^{-1}(C_2OC_4)$ needs $10$ vertices. This implies that $X_{\ref{egr1}}$ is a vertex minimal equilibrium triangulation of
$M_1$. The case $k=-1$ is identical.}
\end{eg}


\begin{eg}\label{egr2}
{\rm For $k= 2$, $L(-k, 1) \cong \mathbb{RP}^3$. Consider triangulations $T_{2}$ for $\pi^{-1}(C_1O)$,
$T_{1}$ for $\pi^{-1}(C_2O)$, $T_{7, 0}$ for $\pi^{-1}(C_3O)$ and $T_{1, 7}$ for $\pi^{-1}(C_4O)$. Then, by Lemma
\ref{LT123n} and Corollary \ref{CT7n}, $T_1 \cup T_2$, $T_1 \cup T_{7, 0}$, $T_{7, 0} \cup T_{1, 7}$,
$T_{2} \cup T_{1, 7}$, $T_2 \cup T_{7, 0}$ and $T_1 \cup T_{1, 7}$ triangulate $\pi^{-1}(C_1OC_2)$,
$\pi^{-1}(C_2OC_3)$, $\pi^{-1}(C_3OC_4)$, $\pi^{-1}(C_1OC_4)$, $\pi^{-1}(C_1OC_3)$ and $\pi^{-1}(C_2OC_4)$ respectively.
Consider the cone $\pi^{-1}(I_i)$ over $\pi^{-1}(C_{i-1}OC_{i})$ with apex $V_i$ for $i= 1, 2, 3, 4$
(addition is modulo 4). This gives
$$ 
X_{\ref{egr2}}= (V_2 \ast (T_1 \cup T_2)) \cup (V_3 \ast (T_1 \cup T_{7, 0})) \cup (V_4 \ast (T_{7, 0} \cup T_{1,7}))
\cup (V_1 \ast (T_{1, 7} \cup T_2)) 
$$
a triangulation of $M_2$ with $f_0(X_{\ref{egr2}})=7+6+3+4 = 20$. The case $k=-2$ is identical.}
\end{eg}


\begin{eg}\label{egr3}
{\rm For $k= 3$, $L(-k, 1) \cong L(3,1)$. Consider triangulations $T_2$ for $\pi^{-1}(C_1O)$, $T_{1}$ for
$\pi^{-1}(C_2O)$, $T_{4, 0}$ for $\pi^{-1}(C_3O)$ and $T_{1,7}$ for $\pi^{-1}(C_4O)$. Then, by Remark \ref{R1},
Lemma \ref{LT123n} and Corollary \ref{CT4n}, we get that $T_1 \cup T_2$, $T_1 \cup T_{4, 0}$, $T_{4, 0} \cup T_{1, 7}$,
$T_2 \cup T_{1, 7}$, $T_{2} \cup T_{4, 0}$ and $T_1 \cup T_{1, 7}$ triangulate $\pi^{-1}(C_1OC_2)$,
$\pi^{-1}(C_2OC_3)$, $\pi^{-1}(C_3OC_4)$, $\pi^{-1}(C_1OC_4)$, $\pi^{-1}(C_1OC_3)$ and $\pi^{-1}(C_2OC_4)$
respectively. Now, consider the cone $\pi^{-1}(I_i)$ over $\pi^{-1}(C_{i-1}OC_{i})$ with apex $V_i$ for $i= 1, 2, 3, 4$
(addition is modulo 4). This gives
$$ 
X_{\ref{egr3}}= (V_2 \ast (T_1 \cup T_2)) \cup (V_3 \ast (T_1 \cup T_{4, 0})) \cup (V_4 \ast (T_{4, 0} \cup T_{1,7}))
\cup (V_1 \ast (T_{1, 7} \cup T_2))
$$
is triangulation of $M_3$ with $f_0(X_{\ref{egr3}})=7+3+10+4 = 24$. The case $k=-3$ is identical.}
\end{eg}


\begin{eg}\label{egr4}
{\rm For $k= 4$, $L(-k, 1) \cong L(4,1)$. Consider triangulations $T_{4,0}$ for $\pi^{-1}(C_1O)$, $T_{1}$
for $\pi^{-1}(C_2O)$, $T_{8,0}$ for $\pi^{-1}(C_3O)$ and $T_{1,7}$ for $\pi^{-1}(C_4O)$. Then, by Lemma
\ref{LT123n} and Corollaries \ref{CT4n}, \ref{CT7n}, we get that $T_1 \cup T_{4, 0}$,
$T_{1} \cup T_{8, 0}$, $T_{8, 0} \cup T_{1, 7}$, $T_{4, 0} \cup T_{8, 0}$ and $T_1 \cup T_{1, 7}$
triangulate $\pi^{-1}(C_1OC_2)$, $\pi^{-1}(C_2OC_3)$, $\pi^{-1}(C_3OC_4)$, $\pi^{-1}(C_1OC_4)$,
$\pi^{-1}(C_1OC_3)$ and $\pi^{-1}(C_2OC_4)$ respectively. Now, consider the cone $\pi^{-1}(I_i)$ over
$\pi^{-1}(C_{i-1}OC_{i})$ with apex $V_i$ for $i= 1, 2, 3, 4$ (addition is modulo 4). This gives
$$ 
X_{\ref{egr4}}= (V_2 \ast (T_1 \cup T_{4,0})) \cup (V_3 \ast (T_1 \cup T_{8, 0})) \cup (V_4 \ast (T_{8, 0} \cup T_{1,7}))
\cup (V_1 \ast (T_{1, 7} \cup T_{4,0}))
$$
is an equilibrium triangulation of $M_4$
with $f_0(X_{\ref{egr4}}) = 7+10+6+3+4 = 30$. The case $k=-4$ is identical.}
\end{eg}

\begin{eg}\label{egr5}
{\rm If $k=0$, then $L(-k, 1) \cong S^1 \times S^2$ and $M_0$ is equivariantly homeomorphic to $S^2 \times S^2$.
Consider triangulations
$T_1$ for $\pi^{-1}(C_1O)$, $T_2$ for $\pi^{-1}(C_2O)$, $T_{1, 7}$ for $\pi^{-1}(C_3O)$ and $T_{2, 7}$
for $\pi^{-1}(C_4O)$. Then, by Remark \ref{R1} and Lemma \ref{LT123n}, we get that $T_1 \cup T_2$,
$T_2 \cup T_{1,7}$, $T_{1, 7} \cup T_{2, 7}$, $T_{1} \cup T_{2, 7}$, $T_1 \cup T_{1, 7}$ and $T_2 \cup T_{2, 7}$
triangulate $\pi^{-1}(C_1OC_2)$, $\pi^{-1}(C_2OC_3)$, $\pi^{-1}(C_3OC_4)$, $\pi^{-1}(C_1OC_4)$,
$\pi^{-1}(C_1OC_3)$ and $\pi^{-1}(C_2OC_4)$ respectively. Now consider the cone $\pi^{-1}(I_i)$ over
$\pi^{-1}(C_{i-1}OC_{i})$ with apex $V_i$ for $i= 1, 2, 3, 4$ (addition is modulo 4). This gives that
$$ 
X_{\ref{egr5}}= (V_2 \ast (T_1 \cup T_{2})) \cup (V_3 \ast (T_2 \cup T_{1, 7})) \cup (V_4 \ast (T_{1, 7} \cup T_{2, 7}))
\cup (V_1 \ast (T_{1} \cup T_{2, 7}))
$$
is an equilibrium triangulation of $M_0$ with $f_0(X_{\ref{egr5}})=7+3+3+4 = 17$.}
\end{eg}

\begin{qn}
{\rm Find a vertex minimal equilibrium triangulation of the Hirzebruch surface $M_k$ when $|k| = 2, 3, 4$.}
\end{qn}

\begin{qn}
{\rm Find an equilibrium triangulation of the Hirzebruch surface $M_k$ when $|k| \geq 5$.}
\end{qn}

\section{Equilibrium triangulations of quasitoric 4-manifolds over pentagons:}\label{m5}

Let $\pi : M \to P$ be a 4-dimensional quasitoric manifold over a pentagon $P=V_1 \cdots V_5V_1$
and $\xi: \{V_1V_2, \ldots, V_4V_5, V_1V_5\} \to \ZZ^2$ be a characteristic function on $P$.
By the definition of characteristic function, we may assume $\xi(V_1V_2)=(-1, 0)$ and
$\xi(V_2V_3)=(0, -1)$. So, $\xi(V_3V_4)=(1, k)$ and $\xi(V_1V_5)=(\ell, 1)$ (up to sign) for some $\ell, k \in \ZZ$.
Let $\xi(V_4V_5) =(a, b)$. Then, $a-\ell b=\pm 1$ and $ak-b=\pm 1$. For computational purpose,
we assume that
\begin{equation}\label{eqp}
 a-b\ell= 1 \mbox{ and } b-ak= 1.
\end{equation}

 Thus, $a$ and $b$ are rational functions of $k, \ell$ if
$k\ell \neq 1$. In this case we denote $M$ by $M_{k,\ell}$.

If $k\ell=1$, then $(k,\ell)=(1,1)$ or $(k, \ell)=(-1, -1)$. Since $a-b \ell =1= b- ak$, $(k, \ell) \neq (1, 1)$.
Thus $(k, \ell) = (-1, -1)$. In this case, $M$ is denoted by $M_{11b}$.

\begin{lemma}
There is an $(a, b) \in \ZZ^2$ such that $a-\ell b= 1$ and $b-ak= 1$ if and only if $k=-1, 0$ or $\ell=-1, 0$ or $(k,\ell) \in
\{(3,1), (2,1), (2,2), (1,2), (1,3)\}$.
\end{lemma}
\begin{proof}
 Let $A=\{(k,\ell) : k=-1, 0 \mbox{ or } \ell=-1, 0 \mbox{ or } (k,\ell) \in \{(3,1), (2,1), (2,2), (1,2), (1,3)\}\}$.
 It is clear that if $(k,\ell) \in A$, then $a-\ell b= 1$ and $b-ak= 1$ has a solution in $\ZZ^2$.
 
 For the converse, let $B$ be the compliment of $A$ in $\ZZ^2$. Then
 
 $B=\{(1,1)\} \cup \{(k,\ell) : k \leq -2, \ell \geq 1\} \cup 
 \{(k, \ell) : k \leq -2, \ell \leq -2\} \cup \{(k, \ell) : k \geq 1, \ell \leq -2\} \cup \{(k, \ell) : k \geq 1, \ell \geq 4\}
 \cup \{(k, \ell) : k \geq 2, \ell \geq 3\} \cup \{(k, \ell) : k \geq 3, \ell \geq 2\} \cup \{(k, \ell) : k \leq 4, \ell \geq 1\}$.
 
 If $(k,\ell)=(1,1)$, $a-\ell b= 1$ and $b-ak= 1$ has no solution in $\ZZ^2$.
Let $k \leq -2, \ell \geq 1$. Then, $k\ell \leq -2\ell$. So, $1-k\ell \geq 2\ell+1 > \ell+1 > 0$. This implies $1 > a=\frac{\ell+1}{1-k\ell} >0$.
Similarly, we can show that for any $(k,\ell) \in B$, $a-\ell b= 1$ and $b-ak= 1$ has no solution in $\ZZ^2$.
 \end{proof}

In this section, we consider the rectangular subdivision of the pentagon as in Figure \ref{egch006}. For the characteristic
function of Figure \ref{egch006}, we have
\begin{align*}
&\pi^{-1}(C_iOC_{i+1})  \cong  S^3 \cong  \pi^{-1}(C_1OC_5) \mbox{ for } 1\leq  i  \leq 4, \\
&\pi^{-1}(C_1O_PC_3)  = L(k, 1), ~~ \pi^{-1}(C_1OC_4)= L(b,-a)= \bar{L(b, a)} \cong L(b, a), \\
&\pi^{-1}(C_2OC_4) =L(a, b), ~~ \pi^{-1}(C_2OC_5)= L(\ell, 1) \mbox{ and }  \pi^{-1}(C_3OC_5)= L(\ell k-1,-\ell).
\end{align*}

\begin{figure}[ht]
\centerline{\scalebox{0.90}{\input{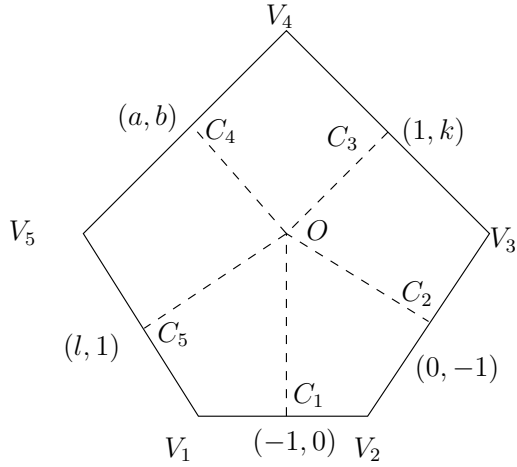} } }
\caption {Characteristic functions of pentagon.}
\label{egch006}
\end{figure}

\begin{lemma}
 The manifold $M_{k,\ell}$ is $\delta$-equivariantly homeomorphic to $M_{\ell,k}$ where $\delta$ is the automorphism
 of $\TT^2$ obtained by flipping the coordinates. Moreover, this homeomorphism induces a bijective correspondence
 between the equilibrium triangulations of the Hirzebruch surfaces $M_{k, \ell}$ and $M_{\ell, k}$.
 \end{lemma}

\begin{proof}
 Consider the reflection of the pentagon which fixes $V_2$. Then, the first claim of the lemma follows
 from Proposition \ref{clema3} and Equation (\ref{eqp}). Since the automorphism $\delta$ of $\TT^2$ interchange
 the coordinates of $\TT^2$, the $\delta$-equivariant homeomorphism preserves the equilibrium set, zones of
 influence and the boundary of zones of influence. Thus, the second claim follows.
\end{proof}

 We recall the number of vertices of solid tori of Section \ref{tortri}.
 By Example \ref{tst1}, $f_0(T_i) = f_0(\mathcal{T})=7$ if $1 \leq i \leq 3$.
By Example \ref{ETjn}, $f_0(T_{j,n})= f_0(\mathcal{T})+2$ if $1 \leq j \leq 3, 0 \leq n \leq 6$ and
$f_0(T_{j,n})= f_0(\mathcal{T})+3$ if $1 \leq j \leq 3, n \geq 7$. By Example \ref{ET56n}, $f_0(T_{j,n})=f_0(\mathcal{T}) +10$ 
if $4 \leq j \leq 6, n \geq 0$. By Example \ref{ET89n}, $f_0(T_{j,n}) = f_0(\mathcal{T}) +6$ if $7 \leq j \leq 9, n \geq 0$.
Under the assumption in the beginning of this section, the following examples are possible complex projective
 surfaces over the pentagon with  $-3 \leq k, \ell \leq 3$ and $-4 \leq a \leq 4$. Note that the corresponding
 characteristic function in the following examples is complete.

\begin{eg}\label{egp4}
 {\rm Let $(k, \ell) =(-1, 0)$. Then, $(a, b)=(1, 0)$.
 Now $L(k, 1)=L(-1, 1) \cong S^3$,
 $L(b, a)=L(0, 1) \cong S^1 \times S^2$, $L(a, b)=L(1, 0) \cong S^3$, $L(\ell, 1)=L(0, 1) \cong S^1 \times S^2$,
 $L(\ell k-1, -\ell)=L(-1, 0) \cong S^3$.
 
 Consider triangulations $T_1$ for $\pi^{-1}(C_1O)$, $T_2$ for $\pi^{-1}(C_2O)$, $T_3$ for $\pi^{-1}(C_3O)$,
 $T_{1,7}$ for $\pi^{-1}(C_4O)$ and $T_{2, 7}$ for $\pi^{-1}(C_5O)$.
 Then, by Remark \ref{R1} and Lemma \ref{LT123n}, we get that
$T_{1} \cup T_2$, $T_2 \cup T_3$, $T_3 \cup T_{1, 7}$, $T_{1,7} \cup T_{2,7}$, $T_{2,7} \cup T_{1}$
triangulate $\pi^{-1}(C_1OC_2)$, $\pi^{-1}(C_2OC_3)$, $\pi^{-1}(C_3OC_4)$, $\pi^{-1}(C_4OC_5)$,
$\pi^{-1}(C_1OC_5)$ respectively and $T_1 \cup T_{3}$, $T_1 \cup T_{1, 7}$, $T_2 \cup T_{1,7}$,
$T_2 \cup T_{2,7}$, $T_3 \cup T_{2,7}$ triangulate $\pi^{-1}(C_1OC_3)$, $\pi^{-1}(C_1OC_4)$,
$\pi^{-1}(C_2OC_4)$, $\pi^{-1}(C_2OC_5)$, $\pi^{-1}(C_3OC_5)$ respectively. This implies that
 $$
X_{\ref{egp4}}=(V_2 \ast (T_1 \cup T_2)) \cup (V_3 \ast (T_2 \cup T_3)) \cup (V_4 \ast (T_3 \cup T_{1, 7}))
\cup (V_5 \ast (T_{1, 7} \cup T_{2, 7})) \cup (V_1 \ast (T_1 \cup T_{2, 7}))
$$
is an equilibrium triangulation of $M_{-1,0}$ with $f_0(X_{\ref{egp4}})=7+3+3+5=18$.

Since the minimum number of vertices require for a triangulation of $S^1 \times S^2$ is $10$ (cf. \cite{BD}),
any triangulation of $\pi^{-1}(C_1OC_4)$ and $\pi^{-1}(C_2OC_5)$ needs $10$ vertices. This implies that $X_{\ref{egp4}}$
is a vertex minimal equilibrium triangulation of $M_{-1,0}$.}
\end{eg}

\begin{eg}\label{egp7}
{\rm Let $(k,\ell)= (1, -1)$. Then, $(a, b)=(0, 1)$.
Now, $L(k, 1)=L(1, 1) \cong S^3$,
$L(b, a) = L(1, 0) \cong S^3$, $L(a, b)= L(0, 1) \cong S^2 \times S^1$, $L(\ell, 1)=L(-1, 1) \cong S^3,
$ $L(k\ell-1, -\ell) = L(-2, 1) \cong \mathbb{RP}^3$.

Consider triangulations $T_{1}$ for $\pi^{-1}(C_1O)$, $T_{3}$ for $\pi^{-1}(C_2O)$,
$T_{2}$ for $\pi^{-1}(C_3O)$, $T_{3, 7}$ for $\pi^{-1}(C_4O)$ and $T_{7, 0}$ for $\pi^{-1}(C_5O)$.
Then, by Remark \ref{R1}, Lemma \ref{LT123n} and Corollary \ref{CT7n}, we get that
$T_{1} \cup T_3$, $T_3 \cup T_2$, $T_2 \cup T_{3,7}$, $T_{3,7} \cup T_{7, 0}$, $T_{7, 0} \cup T_{1}$
triangulate $\pi^{-1}(C_1OC_2)$, $\pi^{-1}(C_2OC_3)$, $\pi^{-1}(C_3OC_4)$, $\pi^{-1}(C_4OC_5)$,
$\pi^{-1}(C_1OC_5)$ respectively and $T_1 \cup T_{2}$, $T_1 \cup T_{3, 7}$, $T_3 \cup T_{3,7}$,
$T_3 \cup T_{7,0}$, $T_2 \cup T_{7,0}$ triangulate $\pi^{-1}(C_1OC_3)$, $\pi^{-1}(C_1OC_4)$,
$\pi^{-1}(C_2OC_4)$, $\pi^{-1}(C_2OC_5)$, $\pi^{-1}(C_3OC_5)$ respectively. This implies that
 $$
X_{\ref{egp7}} = (V_2 \ast (T_1 \cup T_{3})) \cup (V_3 \ast (T_3 \cup T_2)) \cup (V_4 \ast (T_2 \cup T_{3,7}))
\cup (V_5 \ast (T_{3,7} \cup T_{7,0})) \cup (V_1 \ast (T_{7,0} \cup T_{1}))
$$
is an equilibrium triangulation of $M_{1,-1}$ with $f_0(X_{\ref{egp7}})=7+3+6+5=21$.}
\end{eg}

\begin{eg}\label{egp8}
{\rm Let $(k,\ell)= (-2, -1)$. Then, $(a, b)=(0, 1)$. 
Now $L(k, 1)=L(-2, 1) \cong \mathbb{RP}^3$,
$L(b, a) = L(1, 0) \cong S^3$, $L(a, b)= L(0, 1) \cong S^1 \times S^2$, $L(\ell, 1)=L(-1, 1) \cong S^3,$  $L(k\ell-1, -\ell) = L(1, -1) \cong S^3$. 

Consider triangulations $T_2$ for $\pi^{-1}(C_1O)$, $T_{1}$ for $\pi^{-1}(C_2O)$,
$T_{7, 0}$ for $\pi^{-1}(C_3O)$, $T_{1,7}$ for $\pi^{-1}(C_4O)$ and $T_{3}$ for $\pi^{-1}(C_5O)$.
 Similarly, as in Example \ref{egp7}, we can show that
 $$
X_{\ref{egp8}}=(V_2 \ast (T_1 \cup T_2)) \cup (V_3 \ast (T_1 \cup T_{7, 0})) \cup (V_4 \ast (T_{7, 0} \cup T_{1, 7}))
\cup (V_5 \ast (T_{1, 7} \cup T_{3})) \cup (V_1 \ast (T_2 \cup T_{3}))
$$
is an equilibrium triangulation of $M_{-2,-1}$ with $f_0(X_{\ref{egp8}})=7+6+3+5=21$.}
\end{eg}

\begin{eg}\label{egp11}
{\rm Let $(k,\ell)= (2, -1)$. Then, $(a, b)=(0, 1)$.
Now $L(k, 1)=L(2, 1) \cong \mathbb{RP}^3$,
$L(b, a) = L(1, 0) \cong S^3$, $L(a, b)= L(0, 1) \cong S^1 \times S^2$, $L(\ell, 1)=L(-1, 1) \cong S^3,
$ $L(k\ell-1, - \ell) = L(-3, 1) \cong L(3, 1)$. 

Consider triangulations $T_2$ for $\pi^{-1}(C_1O)$, $T_{3, 7}$ for $\pi^{-1}(C_2O)$, $T_{7, 0}$
for $\pi^{-1}(C_3O)$, $T_{3}$ for $\pi^{-1}(C_4O)$ and $T_{9,0}$ for $\pi^{-1}(C_5O)$. Then, by Remark
\ref{R1}, Lemma \ref{LT123n} and Corollary \ref{CT7n}, we get that $T_2 \cup T_{3, 7}$, ~ $T_{3, 7} \cup 
T_{7, 0}$, ~ $T_3 \cup T_{7,0}$, ~ $T_3 \cup T_{9,0}$, ~ $T_2 \cup T_{9, 0}$ triangulate $\pi^{-1}(C_1OC_2), ~$
$\pi^{-1}(C_2OC_3)$, $~\pi^{-1}(C_3OC_4)$, $~\pi^{-1}(C_4OC_5)$ respectively and $T_2 \cup T_{7,0}$,
$~T_2 \cup T_{3}$, $~T_{3} \cup T_{3, 7}$, $~T_{3,7} \cup T_{9,0}$, $~T_{7,0} \cup T_{9,0}$ triangulate
$\pi^{-1}(C_1OC_3)$, $~\pi^{-1}(C_1OC_4)$, $~\pi^{-1}(C_2OC_4)$, $~\pi^{-1}(C_2OC_5)$, $~\pi^{-1}(C_3OC_5)$
respectively. This implies that
 $$
X_{\ref{egp11}}=(V_2 \ast (T_2 \cup T_{3, 7})) \cup (V_3 \ast (T_{3, 7} \cup T_{7,0})) \cup (V_4 \ast
 (T_{ 7,0} \cup T_{3})) \cup (V_5 \ast (T_{3} \cup T_{9, 0})) \cup (V_1 \ast (T_2 \cup T_{ 9, 0}))
$$
is an equilibrium triangulation of $M_{2,-1}$
with $f_0(X_{\ref{egp11}})=7+3+6+6+5=27$.}
\end{eg}

\begin{eg}\label{egp19}
{\rm Let $(k,\ell)= (-3, -1)$. Then, $(a, b)=(0, 1)$.
Now $L(k, 1)=L(-3, 1) \cong L(3,1)$,
$L(b, a) = L(1, 0) \cong S^3$, $L(a, b)= L(0, 1) \cong S^1 \times S^2$, $L(\ell, 1)=L(-1, 1) \cong S^3,
$ $L(k\ell-1, -\ell) = L(2, 1) \cong \mathbb{RP}^3$. 

Consider triangulations $T_{9,0}$ for $\pi^{-1}(C_1O)$, $T_{3}$ for $\pi^{-1}(C_2O)$, $T_{7,0}$ for
$\pi^{-1}(C_3O)$, $T_{3,7}$ for $\pi^{-1}(C_4O)$ and $T_{2}$ for $\pi^{-1}(C_5O)$.  Similarly
as in Example \ref{egp11} we can show that
 $$
X_{\ref{egp19}}=(V_2 \ast (T_{3} \cup T_{9,0})) \cup (V_3 \ast (T_3 \cup T_{7,0})) \cup (V_4 \ast
 (T_{3,7} \cup T_{7,0})) \cup (V_5 \ast (T_{2} \cup T_{3,7})) \cup (V_1 \ast (T_{2} \cup T_{9,0}))
$$
is an equilibrium triangulation of $M_{-3,-1}$ with $f_0(X_{\ref{egp19}})=7+6+6+3+5=27$.}
\end{eg}

\begin{eg}\label{egp22}
{\rm Let $(k,\ell)= (3, -1)$. Then, $(a, b)=(0, 1)$. Thus, $L(k, 1)=L(3, 1)$,
$L(b, a) = L(1, 0) \cong S^3$, $L(a, b)= L(0, 1) \cong S^1 \times S^2$,
$L(\ell, 1)=L(-1, 1) \cong S^3 $, $L(k\ell-1, -\ell) = L(-4, 1) \cong L(4, 1)$. 

Consider triangulations $T_{7,0}$ for $\pi^{-1}(C_1O)$, $T_{1, 7}$ for $\pi^{-1}(C_2O)$, $T_{8, 0}$
for $\pi^{-1}(C_3O)$, $T_{1}$ for $\pi^{-1}(C_4O)$ and $T_{4,0}$ for $\pi^{-1}(C_5O)$. Then, by Remark
\ref{R1}, Lemma \ref{LT123n}, Corollary \ref{CT4n} and Corollary \ref{CT7n}, we get that
$T_{1,7} \cup T_{7,0}$, ~ $T_{1, 7} \cup T_{8, 0}$, ~ $T_1 \cup T_{8,0}$, ~ $T_1 \cup T_{4,0}$,
~ $T_{4,0} \cup T_{7, 0}$ triangulate $\pi^{-1}(C_1OC_2)$ $~\pi^{-1}(C_2OC_3)$, $~\pi^{-1}(C_3OC_4)$,
$~\pi^{-1}(C_4OC_5)$ respectively and $T_{7,0} \cup T_{8,0}$, $~T_{1} \cup T_{7,0}$, $~T_{1} \cup T_{1, 7}$,
$~T_{1,7} \cup T_{4,0}$, $~T_{4,0} \cup T_{8,0}$ triangulate $\pi^{-1}(C_1OC_3)$, $~\pi^{-1}(C_1OC_4)$,
$~\pi^{-1}(C_2OC_4)$, $~\pi^{-1}(C_2OC_5)$, $~\pi^{-1}(C_3OC_5)$ respectively. This implies that
 $$
X_{\ref{egp22}}=(V_2 \ast (T_{1,7} \cup T_{ 7,0})) \cup (V_3 \ast (T_{1, 7} \cup T_{8,0})) \cup (V_4 \ast
 (T_{ 1} \cup T_{8,0})) \cup (V_5 \ast (T_{1} \cup T_{4, 0})) \cup (V_1 \ast (T_{4,0} \cup T_{ 7, 0}))
$$
is an equilibrium triangulation of $M_{3,-1}$
with $f_0(X_{\ref{egp22}})=7+3+10+6+6+5=37$.}
\end{eg}
 
\begin{eg}\label{egp1}
{\rm Let $(k, \ell)= (0,0)$. Then, $(a, b)=(1, 1)$.
Now $L(k, 1)= L(0, 1) \cong S^2 \times S^1$,
$L(b, a) =L(1, 1) \cong S^3$, $L(a, b) = L(1, 1) \cong S^3$, $L(\ell, 1)=L(0, 1) \cong
S^1 \times S^2$ and $L(\ell k-1, -\ell)= L(-1,0) \cong S^3$.

Consider triangulations $T_1$ for $\pi^{-1}(C_1O)$, $T_{2}$ for $\pi^{-1}(C_2O)$, $T_{1, 7}$
for $\pi^{-1}(C_3O)$, $T_{3}$ for $\pi^{-1}(C_4O)$ and $T_{2,7}$ for $\pi^{-1}(C_5O)$. Then, by
Remark \ref{R1} and Lemma \ref{LT123n}, we get that $T_1 \cup T_2$, $T_2 \cup T_{1, 7}$,
$T_{1, 7} \cup T_{3}$, $T_3 \cup T_{2, 7}$, $T_{1} \cup T_{2, 7}$ triangulate $\pi^{-1}(C_1OC_2)$,
$\pi^{-1}(C_2OC_3)$, $\pi^{-1}(C_3OC_4)$, $\pi^{-1}(C_4OC_5)$, $\pi^{-1}(C_1OC_5)$ respectively and
$T_1 \cup T_{1, 7}$, $T_1 \cup T_{3}$, $T_2 \cup T_{3}$, $T_2 \cup T_{2, 7}$, $T_{1, 7} \cup T_{2, 7}$
triangulate $\pi^{-1}(C_1OC_3)$, $\pi^{-1}(C_1OC_4)$, $\pi^{-1}(C_2OC_4)$, $\pi^{-1}(C_2OC_5)$,
$\pi^{-1}(C_3OC_5)$ respectively. Now consider the cone $\pi^{-1}(I_i)$ over $\pi^{-1}(C_{i-1})C_i)$ 
for $1  \leq  i \leq 5$ (addition is modulo 5). This implies that
$$
X_{\ref{egp1}}=(V_2 \ast (T_1 \cup T_2)) \cup (V_3 \ast (T_2 \cup T_{1, 7})) \cup (V_4 \ast (T_3 \cup T_{1, 7}))
\cup (V_5 \ast (T_3 \cup T_{2, 7})) \cup (V_1 \ast (T_1 \cup T_{1, 7}))
$$
is an equilibrium triangulation of $M_{0,0}$
with $f_0(X_{\ref{egp1}}) = 7+3+3+5=18$. 

Since the minimum number of vertices require for a triangulation of $S^1 \times S^2$ is $10$ (cf. \cite{BD}),
any triangulation of $\pi^{-1}(C_1OC_3)$ and $\pi^{-1}(C_2OC_5)$ needs $10$ vertices. This implies that $X_{\ref{egp1}}$
is a vertex minimal equilibrium triangulation of $M_{0,0}$.}
\end{eg}

\begin{eg}\label{egp5}
 {\rm Let $(k,\ell)= (1, 0)$. Then, $(a, b)=(1, 2)$. 
 Now,  $L(k, 1)=L(1, 1) \cong S^3$, $L(b, a)=L(2, 1) \cong \mathbb{RP}^3$,
 $L(a, b)=L(1, 2)=L(1, 0) \cong S^3$,  $L(\ell, 1)=L(0, 1) \cong S^1 \times S^2$, $L(\ell k-1, -\ell)=L(-1, 0) \cong S^3$.
 
 Consider triangulations $T_2$ for $\pi^{-1}(C_1O)$, $T_1$ for $\pi^{-1}(C_2O)$, $T_3$ for $\pi^{-1}(C_3O)$,
 $T_{7, 0}$ for $\pi^{-1}(C_4O)$ and $T_{1, 7}$ for $\pi^{-1}(C_5O)$. Similarly, as in Example \ref{egp1},
 we can show that
 $$
X_{\ref{egp5}}=(V_2 \ast (T_1 \cup T_2)) \cup (V_3 \ast (T_1 \cup T_3)) \cup (V_4 \ast (T_3 \cup T_{7, 0}))
\cup (V_5 \ast (T_{7, 0} \cup T_{1, 7})) \cup (V_1 \ast (T_2 \cup T_{1, 7}))
$$
is an equilibrium triangulation of $M_{1,0}$ with $f_0(X_{\ref{egp5}})=7+6+3+5=21$.}
\end{eg}

\begin{eg}\label{egp12}
{\rm  Let $(k,\ell)= (-2, 0)$. Then, $(a, b)=(1, -1)$.
Now, $L(k, 1)=L(-2, 1) \cong \mathbb{RP}^3$,
$L(b, a) = L(3, 1)$, $L(a, b)= L(1, 3) \cong S^3 $, $L(\ell, 1)=L(0, 1) \cong S^1 \times S^2,
$ $L(k\ell-1, -\ell) = L(-1, 0) \cong S^3$. 

Consider triangulations $T_{7,0}$ for $\pi^{-1}(C_1O)$, $T_{3, 7}$ for $\pi^{-1}(C_2O)$, $T_{2}$
for $\pi^{-1}(C_3O)$, $T_{9,0}$ for $\pi^{-1}(C_4O)$ and $T_{3}$ for $\pi^{-1}(C_5O)$.  Similarly,
as in Example \ref{egp11}, we can show that
 $$
X_{\ref{egp12}}=(V_2 \ast (T_{7,0} \cup T_{3, 7})) \cup (V_3 \ast (T_2 \cup T_{3,7})) \cup (V_4 \ast
 (T_{2} \cup T_{9,0})) \cup (V_5 \ast (T_{3} \cup T_{9, 0})) \cup (V_1 \ast (T_3 \cup T_{ 7, 0}))
$$
is an equilibrium triangulation of $M_{-2,0}$ with $f_0(X_{\ref{egp12}})=7+6+3+6+5=27$.}
\end{eg}

\begin{eg}\label{egp120}
{\rm Let $(k,\ell)= (2, 0)$. Then, $(a, b)=(1, 3)$.
Now, $L(k, 1)=L(2, 1) \cong \mathbb{RP}^3$,
$L(b, a) = L(-1, 1) \cong S^3$, $L(a, b)= L(1, -1) \cong S^3$, $L(\ell, 1)=L(0, 1) \cong S^1 \times S^2,
$ $L(k\ell-1, -\ell) = L(-1, 0) \cong S^3$. 

Consider triangulations $T_{2}$ for $\pi^{-1}(C_1O)$, $T_{3, 7}$ for $\pi^{-1}(C_2O)$, $T_{7,0}$
for $\pi^{-1}(C_3O)$, $T_{1}$ for $\pi^{-1}(C_4O)$ and $T_{3}$ for $\pi^{-1}(C_5O)$.  Similarly,
as in Example \ref{egp11}, we can show that
 $$
X_{\ref{egp120}}=(V_2 \ast (T_{2} \cup T_{3, 7})) \cup (V_3 \ast (T_{3,7} \cup T_{7,0})) \cup (V_4 \ast
 (T_{7,0} \cup T_{1})) \cup (V_5 \ast (T_{1} \cup T_{3})) \cup (V_1 \ast (T_3 \cup T_{2}))
$$
is an equilibrium triangulation of $M_{2,0}$ with $f_0(X_{\ref{egp120}})=7+3+6+5=21$.}
\end{eg}

\begin{eg}\label{egp18}
{\rm Let $(k,\ell)= (-3, 0)$. Then, $(a, b)=(1, -2)$.
Now $L(k, 1)=L(-3, 1) \cong L(3, 1)$,
$L(b, a) = L(-2, 1) \cong \mathbb{RP}^3$, $L(a, b)= L(1, -2) \cong S^3$, $L(\ell, 1)=L(0, 1) \cong S^1 \times S^2,
$ $L(k\ell-1, -\ell) = L(-1, 0) \cong S^3$. 

Consider triangulations $T_{7,0}$ for $\pi^{-1}(C_1O)$, $T_{3}$ for $\pi^{-1}(C_2O)$, $T_{9,0}$ for
$\pi^{-1}(C_3O)$, $T_{2}$ for $\pi^{-1}(C_4O)$ and $T_{3,7}$ for $\pi^{-1}(C_5O)$.  Similarly
as in Example \ref{egp11} we can show that
 $$
X_{\ref{egp18}}=(V_2 \ast (T_{7,0} \cup T_{3})) \cup (V_3 \ast (T_3 \cup T_{9,0})) \cup (V_4 \ast
 (T_{2} \cup T_{9,0})) \cup (V_5 \ast (T_{2} \cup T_{3,7})) \cup (V_1 \ast (T_{3,7} \cup T_{7,0}))
$$
is an equilibrium triangulation of $M_{-3,0}$ with $f_0(X_{\ref{egp18}})=7+6+6+3+5=27$.}
\end{eg}

\begin{eg}\label{egp23}
{\rm Let $(k,\ell)= (3, 0)$. Then, $(a, b)=(1, 4)$.
Now $L(k, 1)=L(3, 1)$, $L(b, a) = L(4, 1)$, $L(a, b)= L(1, 4) \cong S^3$,
 $L(\ell, 1)=L(0, 1) \cong S^1 \times S^2,$ $L(k\ell-1, -\ell) = L(-1, 0) \cong S^3$. 

Consider triangulations $T_{8,0}$ for $\pi^{-1}(C_1O)$, $T_{1}$ for $\pi^{-1}(C_2O)$, $T_{7,0}$ for
$\pi^{-1}(C_3O)$, $T_{4,0}$ for $\pi^{-1}(C_4O)$ and $T_{1,7}$ for $\pi^{-1}(C_5O)$.  Similarly
as in Example \ref{egp22} we can show that
 $$
X_{\ref{egp23}}=(V_2 \ast (T_{1} \cup T_{8,0})) \cup (V_3 \ast (T_{1} \cup T_{7,0})) \cup (V_4 \ast
 (T_{4,0} \cup T_{7,0})) \cup (V_5 \ast (T_{1,7} \cup T_{4, 0})) \cup (V_1 \ast (T_{1,7} \cup T_{8,0}))
$$
is an equilibrium triangulation of $M_{3,0}$ with $f_0(\ref{egp23})=7+6+6+10+3+5=37$.}
\end{eg}

Now assume $(k, \ell)=(-1,-1)$. Then, $a+b=1$. By the same argument as before, we can assume that $a \leq b$.
So, assume that $b >0$. Here we are considering the cases $(a, b)=(0,1), (-1, 2), (-2, 3)$ and $(-3, 4)$.

\begin{eg}\label{egp28}
 {\rm Let $(k,\ell)= (-1, -1)$. Then, $a+b=1$. Suppose $(a, b)=(0, 1)$.
 Now $L(k, 1)=L(-1, 1) \cong S^3$, $L(b, a)= L(1, 0) \cong S^3$,
 $L(a, b)=L(0, 1) \cong S^1 \times S^2$, $L(\ell, 1) =L(-1, 1) \cong S^3$, $L(k\ell-1, -\ell)=  L(0, 1) \cong S^1 \times S^2$.
 
 Consider triangulations $T_1$ for $\pi^{-1}(C_1O)$, $T_{3,7}$ for $\pi^{-1}(C_2O)$, $T_{2}$ for
 $\pi^{-1}(C_3O)$, $T_{3}$ for $\pi^{-1}(C_4O)$ and $T_{2,7}$ for $\pi^{-1}(C_5O)$. Similarly, as in 
 Example \ref{egp1}, we can show that
 $$
X_{\ref{egp28}}=(V_2 \ast (T_1 \cup T_{3,7})) \cup (V_3 \ast (T_{3,7} \cup T_{2})) \cup (V_4 \ast (T_{2} \cup T_{3}))
\cup (V_5 \ast (T_3 \cup T_{2, 7})) \cup (V_1 \ast (T_1 \cup T_{2,7}))
$$
is a vertex minimal equilibrium triangulation of $M_{1,1,1}$ with $f_0(X_{\ref{egp28}})=7+3+3+5=18$.}
\end{eg}

\begin{eg}\label{egp29}
{\rm Let $(k,\ell)= (-1, -1)$. Then, $a+b=1$. Suppose $(a, b)=(-1, 2)$.
Now $L(k, 1)=L(-1, 1) \cong S^3$, $L(b, a) = L(2, -1) \cong \mathbb{RP}^3$,
 $L(a, b)= L(-1, 2) \cong S^3$, $L(\ell, 1)=L(-1, 1) \cong S^3,$ $L(k\ell-1, -\ell) = L(0, 1) \cong S^1 \times S^2$.

Consider triangulations $T_{1}$ for $\pi^{-1}(C_1O)$, $T_{3}$ for $\pi^{-1}(C_2O)$, $T_{2}$ for
$\pi^{-1}(C_3O)$, $T_{9,0}$ for $\pi^{-1}(C_4O)$ and $T_{2,7}$ for $\pi^{-1}(C_5O)$.  Similarly
as in Example \ref{egp22} we can show that
 $$
X_{\ref{egp29}}=(V_2 \ast (T_{1} \cup T_{3})) \cup (V_3 \ast (T_{3} \cup T_{2})) \cup (V_4 \ast
 (T_{2} \cup T_{9,0})) \cup (V_5 \ast (T_{9,0} \cup T_{2, 7})) \cup (V_1 \ast (T_{2,7} \cup T_{1}))
$$
is an equilibrium triangulation of the $M_{1,1,2}$ with $f_0(X_{\ref{egp29}})=7+6+3+5=21$.}
\end{eg}

\begin{eg}\label{egp30}
{\rm Let $(k,\ell)= (-1, -1)$. Then, $a+b=1$. Suppose $(a, b)=(-2, 3)$.
Now $L(k, 1)=L(-1, 1) \cong S^3$, $L(b, a) = L(3, -2) \cong L(3,1)$, $L(a, b)= L(-2, 3) \cong \mathbb{RP}^3$,
 $L(\ell, 1)=L(-1, 1) \cong S^3,$ $L(k\ell-1, -\ell) = L( 0, 1) \cong S^1 \times S^2$. 

Consider triangulations $T_{7,0}$ for $\pi^{-1}(C_1O)$, $T_{3}$ for $\pi^{-1}(C_2O)$, $T_{1}$ for
$\pi^{-1}(C_3O)$, $T_{8,0}$ for $\pi^{-1}(C_4O)$ and $T_{1,7}$ for $\pi^{-1}(C_5O)$.  Similarly
as in Example \ref{egp22} we can show that
 $$
X_{\ref{egp30}}=(V_2 \ast (T_{7,0} \cup T_{3})) \cup (V_3 \ast (T_{3} \cup T_{1})) \cup (V_4 \ast
 (T_{1} \cup T_{8,0})) \cup (V_5 \ast (T_{8,0} \cup T_{1,7})) \cup (V_1 \ast (T_{1,7} \cup T_{7,0}))
$$
is an equilibrium triangulation of $M_{1,1,3}$ with $f_0(X_{\ref{egp30}})=7+6+6+3+5=27$.}
\end{eg}

\begin{eg}\label{egp31}
{\rm Let $(k,\ell)= (-1, -1)$. Then, $a+b=1$. Suppose $(a, b)=(-3, 4)$.
Now $L(k, 1)=L(-1, 1) \cong S^3$, $L(b, a) = L(4, -3) \cong L(4,1)$,
 $L(a, b)= L(-3, 4) \cong L(3, 1)$, $L(\ell, 1)=L(-1, 1) \cong S^3,$ $L(k\ell-1, -\ell) = L(0, 1) \cong S^1 \times S^2$. 

Consider triangulations $T_{4,0}$ for $\pi^{-1}(C_1O)$, $T_{7,0}$ for $\pi^{-1}(C_2O)$, $T_{1}$ for
$\pi^{-1}(C_3O)$, $T_{8,0}$ for $\pi^{-1}(C_4O)$ and $T_{1,7}$ for $\pi^{-1}(C_5O)$.  Similarly
as in Example \ref{egp22} we can show that
 $$
X_{\ref{egp31}}=(V_2 \ast (T_{4,0} \cup T_{7,0})) \cup (V_3 \ast (T_{7,0} \cup T_{1})) \cup (V_4 \ast
 (T_{1} \cup T_{8,0})) \cup (V_5 \ast (T_{8,0} \cup T_{1,7})) \cup (V_1 \ast (T_{1,7} \cup T_{4,0}))
$$
is an equilibrium triangulation of $M_{1,1,4}$ with $f_0(X_{\ref{egp31}})=7+10+6+6+3+5=37$.}
\end{eg}

\begin{qn}
{\rm Find an equilibrium triangulation of $M(P, \xi)$ when $|k|, |\ell| \geq 4$.}
\end{qn}


\section{Equilibrium triangulations of quasitoric 4-manifolds over hexagons:}\label{m6}

Let $\pi : M \to P$ be a 4-dimensional quasitoric manifold over the hexagon $P=V_1 \cdots V_6V_1$
and $\xi: \{V_1V_2, \ldots, V_5V_6, V_1V_6\} \to \ZZ^2$ be a characteristic function on $P$.
By the definition of characteristic function, we may assume $\xi(V_1V_2)=(-1, 0)$ and $\xi(V_2V_3)=(0, -1)$.
So, $\xi(V_3V_4)=(1, k)$ and $\xi(V_1V_6)=(\ell, 1)$ up to sign of characteristic vectors for some $\ell, k \in \ZZ$.
Let $\xi(V_4V_5) =(a, b)$ and $\xi(V_5V_6)=(c, d)$. Then, $c-\ell d=\pm 1$, $ ad-bc = \pm 1$ and $ak - b=\pm 1$. 
We choose the characteristic vectors such that 
$$
\{(1,k),(a,b)\}, ~ \{(a,b), (c,d)\},  \mbox{ and } \{(c,d), (\ell,1)\}
$$
form a positively oriented basis. Then
\begin{equation}\label{eqm6}
 b-ak=1, ~ ad-bc = 1, \mbox{ and } c - d\ell= 1.
\end{equation}


In this section, we consider the rectangular subdivision of hexagon as in Figure \ref{egch007}.
When $P$ is a heptagon there are many characteristic functions of it. We study equilibrium triangulation of
$M = M(P, \xi)$ whose characteristic functions are given in the Figure \ref{egch007}.
For the characteristic function of Figure \ref{egch007}, by Orlik and Reymond \cite{OR}, we have
\begin{align*}
&\pi^{-1}(C_iOC_{i+1})  \cong  S^3  \cong  \pi^{-1}(C_1OC_6) ~\mbox{ for } 1\leq  i \leq 5, \\
&\pi^{-1}(C_1OC_3) \cong L(k, 1), ~~ \pi^{-1}(C_1OC_4) \cong L(b, a), ~~ \pi^{-1}(C_1OC_5) \cong L(d, c), \\
& \pi^{-1}(C_2OC_4)=L(a, b), ~~ \pi^{-1}(C_2OC_5)= L(c, d), ~~ \pi^{-1}(C_2OC_6) \cong L(\ell, 1), \\
&\pi^{-1}(C_3OC_5)= L(ck-d, 1), ~~ \pi^{-1}(C_3OC_6) \cong L(k\ell-1, \ell), ~~ \pi^{-1}(C_4OC_6) \cong L(a-b\ell, 1).
\end{align*}

\begin{figure}[ht]
\centerline{\scalebox{0.90}{\input{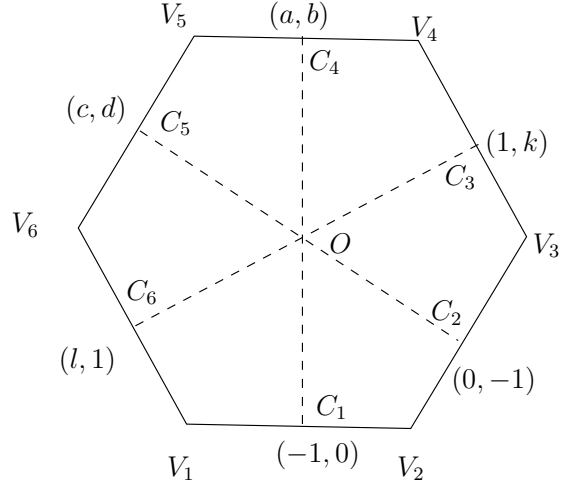} } }
\caption {Characteristic function of hexagon.}
\label{egch007}
\end{figure}

\subsection{Characteristic functions of hexagon.}\label{charhex}

In this subsection, we compute all characteristic functions (see Figure \ref{egch007}) of hexagon when
$-1 \leq \ell \leq 1$, $-3 \leq k \leq 3$ and $-3 \leq a, c \leq 3$.

\begin{itemize}
 \item If $(k,\ell)=(-3,0)$, then $b=-3a+1$, $c=1$ and $ad-b=1$. So, $a(d+3)=2$. Thus, $(a, b, d)=(1,-2, -1)$,
 $(-1, 4, -5)$, $(2, -5, -2)$, $(-2, 7, -4)$. In this case, complete characteristic functions are given by
 $(a,b,c,d) =(1, -2, 1,-1), (2, -5, 1, -2)$.
 
 \item If $(k,\ell)=(-2,0)$, then $b=-2a+1$, $c=1$ and $ad-b=1$. So, $a(d+2)=2$. Thus, $(a, b, d)=(1,-1, 0)$,
 $(-1, 3, -4)$, $(2, -3, -1)$, $(-2, 5, -3)$. In this case, complete characteristic functions are given by
 $(a,b,c,d) =(1, -1, 1,0), (2, -3, 1, -1)$.
 
 \item If $(k,\ell)=(-1,0)$, then $b=-a+1$, $c=1$ and $ad-b=1$. So, $a(d+1)=2$. Thus, $(a, b, d)=(1,0, 1)$,
 $(-1, 2, -3)$, $(2, -1, 0)$, $(-2, 3, -2)$. In this case, complete characteristic functions are given by
 $(a,b,c,d) =(1, 0, 1, 1), (2, -1, 1, 0)$.
 
  \item If $(k,\ell)=(0,0)$, then $b=1$, $c=1$ and $ad=2$. So, $(a, d)=(1,2)$,
 $(-1, 2)$, $(2, 1)$, $(-2, -1)$. In this case, complete characteristic functions are given by
 $(a,b,c,d) =(1, 1, 1, 2), (2, 1, 1, 1)$.
 
  \item If $(k,\ell)=(1,0)$, then $b=a+1$, $c=1$ and $ad-b=1$. So, $a(d-1)=2$. Thus, $(a, b, d)=(1,2, 3)$,
 $(-1, 0, -1)$, $(2, 3, 2)$, $(-2, 1, 0)$. In this case, complete characteristic functions are given by
 $(a,b,c,d) =(1, 2, 1, 3), (2, 3, 1, 2)$.
 
  \item If $(k,\ell)=(2,0)$, then $b=-2a+1$, $c=1$ and $ad-b=1$. So, $a(d-2)=2$. Thus, $(a, b, d)=(1, 3, 4)$,
 $(-1, -1, 0)$, $(2, 5, 3)$, $(-2, -3, 1)$. In this case, complete characteristic functions are given by
 $(a,b,c,d) =(1, 3, 1, 4), (2, 5, 1, 3)$.
 
  \item If $(k,\ell)=(3,0)$, then $b=3a+1$, $c=1$ and $ad-b=1$. So, $a(d-3)=2$. Thus, $(a, b, d)=(1, 4, 5)$,
 $(-1, -2, 1)$, $(2, 7, 4)$, $(-2, -5, 2)$. In this case, complete characteristic functions are given by
 $(a,b,c,d) =(1, 4, 1, 5), (2, 7, 1, 4)$.
 
  \item If $(k,\ell)=(-3,-1)$, then $b=-3a+1$, $c=-d+1$ and $ad-bc=1$. So, $ad - (-3a+1)(-d+1)=1$. This implies,
  $d= 1 + \frac{a-1}{2a-1}$. One can show that $d \in \ZZ$ if and only if $a=0, 1$. Therefore,  $(a, b, c, d)=
  (0, 1, -1, 2)$, $(1, -2, 0, 1)$. Both of them give complete characteristic function on $P$.
 
  \item If $(k,\ell)=(-2,-1)$, then $b=-2a+1$, $c=-d+1$ and $ad-bc=1$. This implies, $ad - (2a-1)(d-1)=1$. So,
  $(1-a)(d-2)=0$. Thus, $a=1$ or $d=2$. In this case, the complete characteristic functions are given by
 $(a,b,c,d) =(1, -1, -3, 4),$ $(1, -1, -2, 3),$ $(1, -1, -1, 2),$ $~(1, -1, 0, 1)$, $~(1, -1, 1, 0)$, $~(1, -1, 2, -1)$,
 $~(1, -1, 3, -2)$, $(-3, 7, -1, 2)$, $(-2, 5, -1, 2)$, $(-1, 3, -1, 2)$, $(0, 1, -1, 2)$, $(2, -3, -1, 2)$,
 $(3, -5, -1, 2)$. We construct an equilibrium triangulation of the complex projective surfaces when
 $(a, b,c,d) = (1, -1, -3, 4), (3, -5, -1, 2)$ in Example \ref{egh18} and \ref{egh23}. Other cases are similar.
 
 \item If $(k,\ell)=(-1,-1)$, then $b=-a+1$, $c=-d+1$ and $ad-bc=1$. So, $ad - (a-1)(d-1)=1$. This implies,
  $a+d=2$. Thus, $d=2-a$ and $c=a-1$. In this case, complete characteristic functions are given by
 $(a,b,c,d) =(a, -a+1, a-1, 2-a) = (-3, 4, -4, 5)$, $(-2, 3, -3, 4)$, $(-1, 2, -2, 3)$, $(0, 1, -1, 2)$,
 $(1, 0, 0, 1)$, $(2, -1, 1, 0)$, $3, -2, 2, -1$. We construct an equilibrium triangulation of the complex
 projective surface when $(a, b,c,d) = (-2, 3, -3, 4)$ in Example \ref{egh31}. Other cases are similar.
 
 
  \item If $(k,\ell)=(1,-1)$, then $b=a+1$, $c=-d+1$ and $ad-bc=1$. So, $ad + (a+1)(d-1)=1$. This implies,
  $d= \frac{a+2}{2a+1}$. One can show that $d \in \ZZ$ if and only if $a=-2, -1, 0 , 1$.
  Therefore, $(a, b, c, d)= (-2, -1, 1, 0)$, $(-1, 0, 2, -1)$, $(0, 1, -1, 2)$, $(1, 2, 0, 1)$. 
  In this case, complete characteristic functions are given by $(a,b,c,d) =(0, 1, -1, 2), (1, 2, 0, 1)$.
  
  \item If $(k,\ell)=(2,-1)$, then $b=2a+1$, $c=-d+1$ and $ad-bc=1$. So, $ad + (2a+1)(d-1)=1$. This implies,
  $d= \frac{2a+2}{3a+1}$. One can show that $d \in \ZZ$ if and only if $a=-1, 0, 1$. Therefore,  $(a, b, c, d)=
  (-1, -1, 1, 0)$, $(0, 1, -1, 2)$, $(1, 3, 0, 1)$. In this case, complete characteristic functions are given by
 $(a,b,c,d) =(0, 1, -1, 2), (1, 3, 0, 1)$.

 \item If $(k,\ell)=(3,-1)$, then $b=3a+1$, $c=-d+1$ and $ad-bc=1$. So, $ad + (3a+1)(d-1)=1$. This implies,
  $d= \frac{3a+2}{4a+1}$. Therefore, $d \in \ZZ$ if and only if $a=0 , 1$. Thus,  $(a, b, c, d)=
  (0, 1, -1, 2)$, $(1, 4, 0, 1)$. Both of the give complete characteristic function on $P$.

 \item If $(k,\ell)=(-3,1)$, then $b=-3a+1$, $c=d+1$ and $ad-bc=1$. So, $ad - (-3a+1)(d+1)=1$. This implies,
  $a= \frac{d+2}{4d+3}$. Therefore, $a \in \ZZ$ if and only if $d=-2$. Thus, $(a, b, c, d)= (0, 1, -1, -2)$.
  
 \item If $(k,\ell)=(-2,1)$, then $b=-2a+1$, $c=d+1$ and $ad-bc=1$. So, $ad - (-2a+1)(d+1)=1$. This implies,
  $a= \frac{d+2}{3d+2}$. One can show that $a \in \ZZ$ if and only if $d=0 , -2$. Therefore,  $(a, b, c, d)=
  (1, -1, 1, 0)$, $(0, 1, -1, -2)$.

 

 \item If $(k,\ell)=(1,1)$, then $b=a+1$, $c=d+1$ and $ad-bc=1$. So, $ad - (a+1)(d+1)=1$. This implies,
  $d= -a-2$. Therefore,  $(a, b, c, d)= (a, a+1, -a-1, -a-2)$.

 \item If $(k,\ell)=(2,1)$, then $b=2a+1$, $c=d+1$ and $ad-bc=1$. So, $ad - (2a+1)(d+1)=1$. This implies,
 $(d+2)(a+1)=0$. Therefore, if $a=-1$ then $b=-1$ and $d=c-1$ and if $d=-2$ then $c=-1$ and $b=2a+ 1$.

 \item If $(k,\ell)=(3,1)$, then $b=3a+1$, $c=d+1$ and $ad-bc=1$. So, $ad - (3a+1)(d+1)=1$. This implies,
 $-d= 1 + \frac{a+1}{2a+1}$. Thus, $d \in \ZZ$ if and only if $a=0 , -1$. Therefore,  $(a, b, c, d)=
 (0, 1, -1, -2)$, $(-1, -2, 0, -1)$.
 \end{itemize}
 
 \begin{remark}
  For the last five cases there are no complete characteristic function.
 \end{remark}

 \begin{lemma}
 Let $M_{1}$ and $M_2$ be quasitoric manifolds over the hexagon such that corresponding $(k, \ell)$ are $(-1,0)$
 (or $(1, 0)$ or $(1, -1)$) and $(0, -1)$ (or $(0, 1)$ or $(-1, 1)$) respectively. Then, $M_1$ and $M_2$ are 
 $\delta$-equivariantly homeomorphic where $\delta$ is the automorphism  of $\mathbb{T}^2$ obtained by flipping the coordinates.
Moreover, this homeomorphism induces a bijective correspondence between the equilibrium triangulations of $M_1$ and $M_2$.
 \end{lemma}

\begin{proof}
Consider the reflection of the hexagon which fixes $V_2$ and $V_5$. Then, the first claim of the lemma follows
from Proposition \ref{clema3} and Equation (\ref{eqm6}).
Since the automorphism $\delta$ of $\TT^2$ interchange the coordinates of $\TT^2$, the $\delta$-equivariant 
homeomorphism preserves the equilibrium set, zones of influence and the boundary of zones of influence. So, the
second claim follows.
 \end{proof}

\subsection{Nonsingular projective surfaces over hexagon}\label{eqhex}

Under the assumption of subsection \ref{charhex}, the following examples are possible non-equivariant complex projective
 surfaces over the hexagon (except when 
$$
(a,c,k, \ell)\in \{(1,1,3,0),  (-3,-1,-2,-1),  (-3,-4,-1,-1), (0, -1, 3, -1), (1,0,3,-1)\}).
$$ 
The corresponding characteristic functions in examples below are complete.
 Therefore, by Proposition \ref{corf}, the associate quasitoric manifolds are projective surfaces. The examples respect
 the ordering of Subsection \ref{charhex}. We recall the number of vertices of solid tori of Section \ref{tortri}.
 By Example \ref{tst1}, $f_0(T_i) = f_0(\mathcal{T})=7$ if $1 \leq i \leq 3$.
By Example \ref{ETjn}, $f_0(T_{j,n})= f_0(\mathcal{T})+2$ if $1 \leq j \leq 3, 0 \leq n \leq 6$ and
$f_0(T_{j,n})= f_0(\mathcal{T})+3$ if $1 \leq j \leq 3, n \geq 7$. By Example \ref{ET56n}, $f_0(T_{j,n})=f_0(\mathcal{T}) +10$ 
if $4 \leq j \leq 6, n \geq 0$. By Example \ref{ET89n}, $f_0(T_{j,n}) = f_0(\mathcal{T}) +6$ if $7 \leq j \leq 9$ and $ n \geq 0$.

\begin{eg}\label{egh11}
{\rm Let $(k,\ell)= (-3, 0)$, $(a,b)=(1, -2)$ and $(c,d)=(1,-1)$. Then, $L(k, 1)= L(-3, 1) \cong L(3,1)$,
$~L(b, a) =L(-2, 1) \cong \mathbb{RP}^3$, $~L(d, c)=L(-1,1) \cong S^3$, $~L(a, b) = L(1, -2) \cong S^3$,
$~L(c,d) = L(1, -1) \cong S^3$, $~L(\ell, 1) =L(0,1) \cong S^1 \times S^2$, $~L(ck-d, 1)=L(-2, 1) \cong \mathbb{RP}^3$,
$~L(\ell k-1, \ell)= L(-1,0) \cong S^3$ and $L(a-b \ell, 1)=L(1,1) \cong S^3$.

Consider triangulations $T_{7,0}$ for $\pi^{-1}(C_1O)$, $T_{1}$ for $\pi^{-1}(C_2O)$, $T_{8,0}$ for
$\pi^{-1}(C_3O)$, $T_{2}$ for $\pi^{-1}(C_4O)$, $T_{3}$ for $\pi^{-1}(C_5O)$ and $T_{1,7}$ for $\pi^{-1}(C_6O)$.
Then, by Remark \ref{R1}, Lemma \ref{LT123n} and Corollary \ref{CT7n}, we get that $T_{7,0} \cup T_{1}$,
$~T_{1} \cup T_{8, 0}$, $~T_{8,0} \cup T_{2}$, $~T_2 \cup T_{3}$, $~T_{3} \cup T_{1,7}$,
$~T_{1,7} \cup T_{8,0}$ triangulate $\pi^{-1}(C_1OC_2)$, $~\pi^{-1}(C_2OC_3)$, $~\pi^{-1}(C_3OC_4)$,
$~\pi^{-1}(C_4OC_5)$, $~\pi^{-1}(C_5OC_6)$, $~\pi^{-1}(C_1OC_6)$ respectively and
$T_{7,0} \cup T_{8,0}$, $~T_{7,0} \cup T_{2}$, $~T_{7,0} \cup T_{3}$, $~T_{1} \cup T_{2}$, $~T_{1} \cup T_{3}$,
$~T_{1} \cup T_{1,7}$, $~T_{8,0} \cup T_{3}$, $~T_{8,0} \cup T_{1,7}$, $T_2 \cup T_{1,7}$
triangulate $\pi^{-1}(C_1OC_3)$, $~\pi^{-1}(C_1OC_4)$, $~\pi^{-1}(C_1OC_5)$, $~\pi^{-1}(C_2OC_4)$,
$~\pi^{-1}(C_2OC_5)$, $~\pi^{-1}(C_2OC_6)$, $~\pi^{-1}(C_3OC_5)$, $~\pi^{-1}(C_3OC_6)$, $~\pi^{-1}(C_4OC_6)$
respectively. Now consider the cone $\pi^{-1}(I_i)$ over $\pi^{-1}(C_{i-1}OC_i)$ 
for $1 \leq i \leq 6$ (addition is modulo 6). This implies that
\begin{align*}
X_{\ref{egh11}}=&(V_2 \ast (T_{7,0} \cup T_{1})) \cup (V_3 \ast (T_{1} \cup T_{8,0})) \cup (V_4 \ast (T_{8,0} \cup T_{2})\\
  &\cup (V_5 \ast (T_{2} \cup T_{3})) \cup (V_6 \ast (T_{3} \cup T_{1,7})) \cup (V_1 \ast (T_{1,7} \cup T_{8,0}))
\end{align*}
is an equilibrium triangulation of the corresponding projective surface $M$ with $f_0(X_{\ref{egh11}}) = 7+3+6+6+6=28$.}
\end{eg}

\begin{eg}\label{egh12}
{\rm Let $(k,\ell)= (-3, 0)$, $(a,b)=(2,-5)$ and $(c,d)=(1,-2)$. Then, $L(k, 1)= L(-3, 1) \cong L(3,1)$,
$~L(b, a) =L(-5, 2) \cong L(5,2)$, $~L(d, c)=L(-2,1) \cong \mathbb{RP}^3$, $~L(a, b) = L(2, -5) \cong \mathbb{RP}^3$,
$~L(c,d) = L(1, -2) \cong S^3$, $~L(\ell, 1) =L(0,1) \cong S^1 \times S^2$, $~L(ck-d, 1)=L(-1, 1) \cong S^3$,
$~L(\ell k-1, \ell)= L(-1,0) \cong S^3$ and $L(a-b \ell, 1)=L(2,1) \cong \mathbb{RP}^3$.

Consider triangulations $T_{9,0}$ for $\pi^{-1}(C_1O)$, $T_{3}$ for $\pi^{-1}(C_2O)$, $T_{7,0}$ for
$\pi^{-1}(C_3O)$, $T_{4,0}$ for $\pi^{-1}(C_4O)$, $T_{1}$ for $\pi^{-1}(C_5O)$ and $T_{3,7}$ for $\pi^{-1}(C_6O)$.
Similarly, as in Example \ref{egh11}, we can show that
\begin{align*}
X_{\ref{egh12}}= &(V_2 \ast (T_{9,0} \cup T_{3})) \cup (V_3 \ast (T_{3} \cup T_{7,0})) \cup (V_4 \ast (T_{7,0} \cup T_{4,0})\\
  &\cup (V_5 \ast (T_{4,0} \cup T_{1})) \cup (V_6 \ast (T_{1} \cup T_{3,7})) \cup (V_1 \ast (T_{3,7} \cup T_{9,0}))
\end{align*}
is an equilibrium triangulation of the corresponding complex surface $M$ with $f_0(X_{\ref{egh12}}) = 7+3+10+6+6+6=38$.}
\end{eg} 

\begin{eg}\label{egh9}
{\rm Let $(k,\ell)= (-2, 0)$, $(a,b)=(1,-1)$ and $(c,d)=(1,0)$. Then, $L(k, 1)= L(-2, 1) \cong \mathbb{RP}^3$,
$~L(b, a) =L(-1, 1) \cong S^3$, $~L(d, c)=L(0,1) \cong S^1 \times S^2$, $~L(a, b) = L(1, -1) \cong S^3$,
$~L(c,d) = L(1, 0) \cong S^3$, $~L(\ell, 1) =L(0,1) \cong S^1 \times S^2$, $~L(ck-d, 1)=L(-2, 1) \cong \mathbb{RP}^3$,
$~L(\ell k-1, \ell)= L(-1,0) \cong S^3$ and $L(a-b \ell, 1)=L(1,1) \cong S^3$.

Consider triangulations $T_{2}$ for $\pi^{-1}(C_1O)$, $T_{3}$ for $\pi^{-1}(C_2O)$, $T_{7,0}$ for
$\pi^{-1}(C_3O)$, $T_{1}$ for $\pi^{-1}(C_4O)$, $T_{2,7}$ for $\pi^{-1}(C_5O)$ and $T_{3,7}$ for $\pi^{-1}(C_6O)$.
Similarly, as in Example \ref{egh11}, we can show that
\begin{align*}
X_{\ref{egh9}}= &(V_2 \ast (T_{2} \cup T_{3})) \cup (V_3 \ast (T_{3} \cup T_{7,0})) \cup (V_4 \ast (T_{7,0} \cup T_{1})\\
 &\cup (V_5 \ast (T_{1} \cup T_{2,7})) \cup (V_6 \ast (T_{2,7} \cup T_{3,7})) \cup (V_1 \ast (T_{3,7} \cup T_{2}))
\end{align*}
is an equilibrium triangulation of the corresponding nonsingular projective surface $M$ with
$f_0(X_{\ref{egh9}})=7+6+3+3+6=25$.}
\end{eg}

\begin{eg}\label{egh10}
{\rm Let $(k,\ell)= (-2, 0)$, $(a,b)=(2,-3)$ and $(c,d)=(1,-1)$. Then, $L(k, 1)= L(-2, 1) \cong \mathbb{RP}^3$,
$~L(b, a) =L(-3, 2) \cong L(3,1)$, $~L(d, c)=L(-1,1) \cong S^3$, $~L(a, b) = L(2, -3) \cong \mathbb{RP}^3$,
$~L(c,d) = L(1, -1) \cong S^3$, $~L(\ell, 1) =L(0,1) \cong S^1 \times S^2$, $~L(ck-d, 1)=L(-1, 1) \cong S^3$,
$~L(\ell k-1, \ell)= L(-1,0) \cong S^3$ and $L(a-b\ell, 1)=L(2,1) \cong \mathbb{RP}^3$.

Consider triangulations $T_{7,0}$ for $\pi^{-1}(C_1O)$, $T_{3}$ for $\pi^{-1}(C_2O)$, $T_{2}$ for
$\pi^{-1}(C_3O)$, $T_{8,0}$ for $\pi^{-1}(C_4O)$, $T_{1}$ for $\pi^{-1}(C_5O)$ and $T_{3,7}$ for $\pi^{-1}(C_6O)$.
Similarly, as in Example \ref{egh11}, we can show that
\begin{align*}
X_{\ref{egh10}}= & (V_2 \ast (T_{7,0} \cup T_{3})) \cup (V_3 \ast (T_{3} \cup T_{2})) \cup (V_4 \ast (T_{2} \cup T_{8,0})\\
   &\cup (V_5 \ast (T_{8,0} \cup T_{1})) \cup (V_6 \ast (T_{1} \cup T_{3,7})) \cup (V_1 \ast (T_{3,7} \cup T_{7,0}))
\end{align*}
is an equilibrium triangulation of the corresponding nonsingular projective surface $M$ with $f_0(X_{\ref{egh10}})=7+6+6+3+6=28$.}
\end{eg}

\begin{eg}\label{egh7}
{\rm Let $(k,\ell)= (-1, 0)$, $(a,b)=(1,0)$ and $(c,d)=(1,1)$. Then, $L(k, 1)= L(-1, 1) \cong S^3$,
$~L(b, a) =L(0, 1) \cong S^1 \times S^2$, $~L(d, c)=L(1,1) \cong S^3$, $~L(a, b) = L(1, 0) \cong S^3$,
$~L(c,d) = L(1, 1) \cong S^3$, $~L(\ell, 1) =L(0,1) \cong S^1 \times S^2$, $~L(ck-d, 1)=L(-2, 1) \cong \mathbb{RP}^3$,
$~L(\ell k-1, \ell)= L(-1,0) \cong S^3$ and $L(a-b \ell, 1)=L(1,1) \cong S^3$.

Consider triangulations $T_{1}$ for $\pi^{-1}(C_1O)$, $T_{3}$ for $\pi^{-1}(C_2O)$, $T_{2}$ for
$\pi^{-1}(C_3O)$, $T_{1,7}$ for $\pi^{-1}(C_4O)$, $T_{7,0}$ for $\pi^{-1}(C_5O)$ and $T_{3,7}$ for $\pi^{-1}(C_6O)$.
Similarly, as in Example \ref{egh11}, we can show that
\begin{align*}
X_{\ref{egh7}}= & (V_2 \ast (T_1 \cup T_{3})) \cup (V_3 \ast (T_{3} \cup T_{2})) \cup (V_4 \ast (T_{2} \cup T_{1,7})\\
 &\cup (V_5 \ast (T_{1,7} \cup T_{7,0})) \cup (V_6 \ast (T_{7,0} \cup T_{3,7})) \cup (V_1 \ast (T_{3,7} \cup T_1))
\end{align*}
is an equilibrium triangulation of the corresponding nonsingular projective surface $M$ with
$f_0(X_{\ref{egh7}})=7+3+6+3+6= 25$.}
\end{eg}

\begin{eg}\label{egh8}
{\rm Let $(k,\ell)= (-1, 0)$, $(a,b)=(2,-1)$ and $(c,d)=(1,0)$. Then, $L(k, 1)= L(-1, 1) \cong S^3$,
$~L(b, a) =L(-1, 2) \cong S^3$, $~L(d, c)=L(0,1) \cong S^1 \times S^2$, $~L(a, b) = L(2, -1) \cong \mathbb{RP}^3$,
$~L(c,d) = L(1, 0) \cong S^3$, $~L(\ell, 1) =L(0,1) \cong S^1 \times S^2$, $~L(ck-d, 1)=L(-1, 1) \cong S^3$,
$~L(\ell k-1, \ell)= L(-1,0) \cong S^3$ and $L(a-b \ell, 1)=L(2,1) \cong \mathbb{RP}^3$.

Consider triangulations $T_{3}$ for $\pi^{-1}(C_1O)$, $T_{2}$ for $\pi^{-1}(C_2O)$, $T_{1}$ for
$\pi^{-1}(C_3O)$, $T_{7,0}$ for $\pi^{-1}(C_4O)$, $T_{3,7}$ for $\pi^{-1}(C_5O)$ and $T_{2,7}$ for $\pi^{-1}(C_6O)$.
Similarly, as in Example \ref{egh11}, we can show that
\begin{align*}
X_{\ref{egh8}}= & (V_2 \ast (T_3 \cup T_{2})) \cup (V_3 \ast (T_{2} \cup T_{1})) \cup (V_4 \ast (T_{1} \cup T_{7,0})\\
  &\cup (V_5 \ast (T_{7,0} \cup T_{3,7})) \cup (V_6 \ast (T_{3,7} \cup T_{2,7})) \cup (V_1 \ast (T_{2,7} \cup T_3))
\end{align*}
is an equilibrium triangulation of the associated nonsingular projective surface  $M$ with
$f_0(X_{\ref{egh8}})=7+6+3+3+6=25$.}
\end{eg}

\begin{eg}\label{egh1}
{\rm Let $(k,\ell)= (0, 0)$, $(a,b)=(1, 1)$ and $(c,d)=(1,2)$. Then,  
$L(k, 1)= L(0, 1) \cong S^2 \times S^1$, $~L(b, a) =L(1, 1) \cong S^3$,
$~L(d, c)=L(2,1) \cong \mathbb{RP}^2$,~ $L(a, b) = L(1, 1) \cong S^3$, ~$L(c,d) = L(1, 2) \cong S^3$,
~ $L(\ell, 1) =L(0,1) \cong S^1 \times S^2$,~ $L(ck-d, 1)=L(-2, 1) \cong \mathbb{RP}^3$,
$~L(\ell k-1, - \ell)= L(-1,0) \cong S^3$ and $L(a-b \ell, 1)=L(1,1) \cong S^3$.

Consider triangulations $T_1$ for $\pi^{-1}(C_1O)$, $T_{3}$ for $\pi^{-1}(C_2O)$, $T_{1, 7}$
for $\pi^{-1}(C_3O)$, $T_{2}$ for $\pi^{-1}(C_4O)$, $T_{9,0}$ for $\pi^{-1}(C_5O)$ and $T_{3,7}$ for
$\pi^{-1}(C_6O)$. Then, by Remark \ref{R1} and Lemma \ref{LT123n} and Corollary \ref{CT7n}, we get that $T_1 \cup T_{3}$,
$~T_{3} \cup T_{1, 7}$, $~T_{1, 7} \cup T_{2}$, $~T_2 \cup T_{9,0}$, $~T_{9,0} \cup T_{3,7}$,
$~T_{3,7} \cup T_{1}$ triangulate $\pi^{-1}(C_1OC_2)$, $~\pi^{-1}(C_2OC_3)$, $~\pi^{-1}(C_3OC_4)$,
$~\pi^{-1}(C_4OC_5)$, $~\pi^{-1}(C_5OC_6)$, $~\pi^{-1}(C_1OC_6)$ respectively and
$T_1 \cup T_{1, 7}$, $~T_1 \cup T_{2}$, $~T_1 \cup T_{9,0}$, $~T_{3} \cup T_{2}$, $~T_{3} \cup T_{9,0}$,
$~T_{3} \cup T_{3,7}$, $~T_{1,7} \cup T_{9,0}$, $~T_{1,7} \cup T_{3,7}$, $T_2 \cup T_{3,7}$
triangulate $\pi^{-1}(C_1OC_3)$, $~\pi^{-1}(C_1OC_4)$, $~\pi^{-1}(C_1OC_5)$, $~\pi^{-1}(C_2OC_4)$,
$~\pi^{-1}(C_2OC_5)$, $~\pi^{-1}(C_2OC_6)$, $~\pi^{-1}(C_3OC_5)$, $~\pi^{-1}(C_3OC_6)$, $~\pi^{-1}(C_4OC_6)$
respectively. Now consider the cone $\pi^{-1}(I_i)$ over $\pi^{-1}(C_{i-1}OC_i)$ 
for $1 \leq i \leq 6$ (addition is modulo 6). This implies that
\begin{align*}
X_{\ref{egh1}}= & (V_2 \ast (T_1 \cup T_{3})) \cup (V_3 \ast (T_{3} \cup T_{1, 7})) \cup (V_4 \ast (T_{1,7} \cup T_{2})\\
           &\cup (V_5 \ast (T_2 \cup T_{9,0})) \cup (V_6 \ast (T_{9,0} \cup T_{3,7})) \cup (V_1 \ast (T_{3,7} \cup T_1))
\end{align*}
is an equilibrium triangulation of the corresponding nonsingular projective surface $M$ with $f_0(X_{\ref{egh1}}) = 7+3+6+3+6=25$. 

Since the minimum number of vertices require for a triangulation of $S^1 \times S^2$ is $10$ (cf. \cite{BD}),
any triangulation of $\pi^{-1}(C_1OC_3)$, $\pi^{-1}(C_1OC_5)$, $\pi^{-1}(C_2OC_5)$, $\pi^{-1}(C_2OC_6)$,
$\pi^{-1}(C_3OC_5)$ and $\pi^{-1}(C_4OC_6)$ needs $10$ vertices. This implies that $X_{\ref{egh1}}$
is a vertex minimal equilibrium triangulation of $M$.}
\end{eg}

\begin{eg}\label{egh2}
{\rm Let $(k,\ell)= (0, 0)$, $(a,b)=(2, 1)$ and $(c,d)=(1,1)$. So, $L(k, 1)= L(0, 1) \cong S^2 \times S^1$,
$~L(b, a) =L(1, 2) \cong S^3$, $~L(d, c)=L(1,1) \cong S^3$,~ $L(a, b) = L(2, 1) \cong \mathbb{RP}^3$,
~$L(c,d) = L(1, 1) \cong S^3$,~ $L(\ell, 1) =L(0,1) \cong S^1 \times S^2$,~ $L(ck-d, 1)=L(-1, 1) \cong S^3$,
$~L(\ell k-1, - \ell)= L(-1,0) \cong S^3$ and $L(a-b \ell, 1)=L(2,1) \cong \mathbb{RP}^3$.

Consider triangulations $T_1$ for $\pi^{-1}(C_1O)$, $T_{2}$ for $\pi^{-1}(C_2O)$, $T_{1, 7}$
for $\pi^{-1}(C_3O)$, $T_{7,0}$ for $\pi^{-1}(C_4O)$, $T_{3}$ for $\pi^{-1}(C_5O)$ and $T_{2,7}$ for
$\pi^{-1}(C_6O)$. 
Similarly, as in Example \ref{egh1}, we can show that
\begin{align*}
X_{\ref{egh2}}= & (V_2 \ast (T_1 \cup T_{2})) \cup (V_3 \ast (T_{2} \cup T_{1, 7})) \cup (V_4 \ast (T_{1,7} \cup T_{7,0})\\
           &\cup (V_5 \ast (T_{7,0} \cup T_{3})) \cup (V_6 \ast (T_{3} \cup T_{2,7})) \cup (V_1 \ast (T_{2,7} \cup T_1))
\end{align*}
is an equilibrium triangulation of the corresponding nonsingular projective space $M$ with $f_0(X_{\ref{egh2}}) = 7+3+3+6+6=25$. 

}
\end{eg}

\begin{eg}\label{egh3}
{\rm Let $(k,\ell)= (1, 0)$, $(a,b)=(1,2)$ and $(c,d)=(1,3)$. Then, $L(k, 1)= L(1, 1) \cong S^3$,
$~L(b, a) =L(2, 1) \cong \mathbb{RP}^3$, $~L(d, c)=L(3,1)$, $~L(a, b) = L(1, 2) \cong S^3$,
$~L(c,d) = L(1, 3) \cong S^3$, $~L(\ell, 1) =L(0,1) \cong S^1 \times S^2$, $L(ck-d, 1)=L(-2, 1) \cong \mathbb{RP}^3$,
$~L(\ell k-1, - \ell)= L(-1,0) \cong S^3$ and $L(a-b\ell, 1)=L(1,1) \cong S^3$.

Consider triangulations $T_{7,0}$ for $\pi^{-1}(C_1O)$, $T_{1}$ for $\pi^{-1}(C_2O)$, $T_{3}$ for
$\pi^{-1}(C_3O)$, $T_{2}$ for $\pi^{-1}(C_4O)$, $T_{8,0}$ for $\pi^{-1}(C_5O)$ and $T_{1,7}$ for $\pi^{-1}(C_6O)$.
Similarly, as in Example \ref{egh1}, we can show that
\begin{align*}
X_{\ref{egh3}}= & (V_2 \ast (T_{7,0} \cup T_{1})) \cup (V_3 \ast (T_{1} \cup T_{3})) \cup (V_4 \ast (T_{3} \cup T_{2})\\
           &\cup (V_5 \ast (T_2 \cup T_{8,0})) \cup (V_6 \ast (T_{8,0} \cup T_{1,7})) \cup (V_1 \ast (T_{1,7} \cup T_{7,0}))
\end{align*}
is an equilibrium triangulation of the corresponding nonsingular projective surface $M$ with
$f_0(X_{\ref{egh3}})=7+6+6+3+6=28$.}
\end{eg}

\begin{eg}\label{egh300}
{\rm Let $(k,\ell)= (1, 0)$, $(a,b)=(2,3)$ and $(c,d)=(1,2)$. Then, $L(k, 1)= L(1, 1) \cong S^3$,
$~L(b, a) =L(3, 2) \cong L(3,1)$, $~L(d, c)=L(2,1) \cong \mathbb{RP}^3$, $~L(a, b) = L(2, 3) \cong \mathbb{RP}^3$,
$~L(c,d) = L(1, 2) \cong S^3$, $~L(\ell, 1) =L(0,1) \cong S^1 \times S^2$, $L(ck-d, 1)=L(-1, 1) \cong S^3$,
$~L(\ell k-1, - \ell)= L(-1,0) \cong S^3$ and $L(a-b \ell, 1)=L(2,1) \cong \mathbb{RP}^3$.

Consider triangulations $T_{2}$ for $\pi^{-1}(C_1O)$, $T_{3}$ for $\pi^{-1}(C_2O)$, $T_{1}$ for
$\pi^{-1}(C_3O)$, $T_{4,0}$ for $\pi^{-1}(C_4O)$, $T_{7,0}$ for $\pi^{-1}(C_5O)$ and $T_{3,7}$ for $\pi^{-1}(C_6O)$.
Similarly, as in Example \ref{egh1}, we can show that
\begin{align*}
X_{\ref{egh300}}= &(V_2 \ast (T_{2} \cup T_{3})) \cup (V_3 \ast (T_{3} \cup T_{1})) \cup (V_4 \ast (T_{1} \cup T_{4,0})\\
           &\cup (V_5 \ast (T_{4,0} \cup T_{7,0})) \cup (V_6 \ast (T_{7,0} \cup T_{3,7})) \cup (V_1 \ast (T_{3,7} \cup T_{2}))
\end{align*}
is an equilibrium triangulation of the corresponding nonsingular projective surface $M$ with
$f_0(X_{\ref{egh300}})=7+10+6+3+6=32$.}
\end{eg}

\begin{eg}\label{egh4}
{\rm Let $(k,\ell)= (2, 0)$, $(a,b)=(1,3)$ and $(c,d)=(1,4)$. Then, $L(k, 1)= L(2, 1) \cong \mathbb{RP}^3$,
$~L(b, a) =L(3, 1)$, $~L(d, c)=L(4,1)$, $~L(a, b) = L(1, 3) \cong S^3$,
$~L(c,d) = L(1, 4) \cong S^3$, $~L(\ell, 1) =L(0,1) \cong S^1 \times S^2$, $L(ck-d, 1)=L(-2, 1) \cong \mathbb{RP}^3$,
$~L(\ell k-1, \ell)= L(-1,1) \cong S^3$ and $L(a-b \ell, 1)=L(1,1) \cong S^3$.

Consider triangulations $T_{7,0}$ for $\pi^{-1}(C_1O)$, $T_{3}$ for $\pi^{-1}(C_2O)$, $T_{2}$ for
$\pi^{-1}(C_3O)$, $T_{9,0}$ for $\pi^{-1}(C_4O)$, $T_{6,0}$ for $\pi^{-1}(C_5O)$ and $T_{3,7}$ for $\pi^{-1}(C_6O)$.
Similarly, as in Example \ref{egh1}, we can show that
\begin{align*}
X_{\ref{egh4}} =& (V_2 \ast (T_{7,0} \cup T_{3})) \cup (V_3 \ast (T_{3} \cup T_{2})) \cup (V_4 \ast (T_{2} \cup T_{9,0})\\
           &\cup (V_5 \ast (T_{9,0} \cup T_{6,0})) \cup (V_6 \ast (T_{6,0} \cup T_{3,7})) \cup (V_1 \ast (T_{3,7} \cup T_{7,0}))
\end{align*}
is an equilibrium triangulation of the corresponding nonsingular projective surface $M$ with
$f_0(X_{\ref{egh4}})=7+6+6+10+3+6=38$.}
\end{eg}

\begin{eg}\label{egh5}
{\rm Let $(k,\ell)= (2, 0)$, $(a,b)=(2,5)$ and $(c,d)=(1,3)$. Then, $L(k, 1)= L(2, 1) \cong \mathbb{RP}^3$,
$~L(b, a) =L(5, 2)$, $~L(d, c)=L(3,1)$, $~L(a, b) = L(2, 5) \cong \mathbb{RP}^3$,
$~L(c,d) = L(1, 3) \cong S^3$, $~L(\ell, 1) =L(0,1) \cong S^1 \times S^2$, $L(ck-d, 1)=L(-1, 1) \cong S^3$,
$~L(\ell k-1, \ell)= L(-1,0) \cong S^3$ and $L(a-b \ell, 1)=L(2,1) \cong \mathbb{RP}^3$.

Consider triangulations $T_{4,0}$ for $\pi^{-1}(C_1O)$, $T_{1}$ for $\pi^{-1}(C_2O)$, $T_{3}$ for
$\pi^{-1}(C_3O)$, $T_{9, 0}$ for $\pi^{-1}(C_4O)$, $T_{2}$ for $\pi^{-1}(C_5O)$ and $T_{1,7}$ for $\pi^{-1}(C_6O)$.
Similarly, as in Example \ref{egh1}, we can show that
\begin{align*}
X_{\ref{egh5}}= &(V_2 \ast (T_{4,0} \cup T_{1})) \cup (V_3 \ast (T_{1} \cup T_{3})) \cup (V_4 \ast (T_{3} \cup T_{9,0})\\
      &\cup (V_5 \ast (T_{9,0} \cup T_{2})) \cup (V_6 \ast (T_{2} \cup T_{1,7})) \cup (V_1 \ast (T_{1,7} \cup T_{4,0}))
\end{align*}
is an equilibrium triangulation of the corresponding nonsingular projective surface $M$ with
$f_0(X_{\ref{egh5}})=7+10+6+3+6=32$.}
\end{eg}

\begin{eg}\label{egh6}
{\rm Let $(k,\ell)= (3, 0)$, $(a,b)=(2,7)$ and $(c,d)=(1,4)$. Then, $L(k, 1)= L(3, 1)$,
$~L(b, a) =L(7, 2)$, $~L(d, c)=L(4,1)$, $~L(a, b) = L(2, 7) \cong \mathbb{RP}^3$,
$~L(c,d) = L(1, 4) \cong S^3$, $~L(\ell, 1) =L(0,1) \cong S^1 \times S^2$, $~L(ck-d, 1)=L(-1, 1) \cong S^3$,
$~L(k\ell-1, \ell)= L(-1,0) \cong S^3$ and $L(a-b\ell, 1)=L(2,1) \cong \mathbb{RP}^3$.

Consider triangulations $T_{5,0}$ for $\pi^{-1}(C_1O)$, $T_{2}$ for $\pi^{-1}(C_2O)$, $T_{3}$ for
$\pi^{-1}(C_3O)$, $T_{6,0}$ for $\pi^{-1}(C_4O)$, $T_{9,0}$ for $\pi^{-1}(C_5O)$ and $T_{2,7}$ for $\pi^{-1}(C_6O)$.
Similarly, as in Example \ref{egh1}, we can show that
\begin{align*}
X_{\ref{egh6}}= &(V_2 \ast (T_{5,0} \cup T_{2})) \cup (V_3 \ast (T_{2} \cup T_{3})) \cup (V_4 \ast (T_{3} \cup T_{6,0})\\
           &\cup (V_5 \ast (T_{6,0} \cup T_{9,0})) \cup (V_6 \ast (T_{9,0} \cup T_{2,7})) \cup (V_1 \ast (T_{2,7} \cup T_{5,0}))
\end{align*}
is an equilibrium triangulation of the corresponding nonsingular projective surface $M$ with
$f_0(X_{\ref{egh6}})=7+10+10+6+3+6=42$.}
\end{eg}

\begin{eg}\label{egh13}
{\rm Let $(k,\ell)= (-3, -1)$, $(a,b)=(0,1)$ and $(c,d)=(-1,2)$. So, $L(k, 1)= L(-3, 1) \cong L(3,1)$,
$~L(b, a) =L(1, 0) \cong S^3$, $~L(d, c)=L(2,-1) \cong \mathbb{RP}^3$, $~L(a, b) = L(0, 1) \cong S^1 \times S^2$,
$~L(c,d) = L( -1, 2) \cong S^3$, $~L(\ell, 1) =L(-1,1) \cong S^3$, $~L(ck-d, 1)=L(1, 1) \cong S^3$,
$~L(\ell k-1, \ell)= L(2,-1) \cong \mathbb{RP}^3$ and $L(a-b \ell, 1)=L(1,1) \cong S^3$.

Consider triangulations $T_{7,0}$ for $\pi^{-1}(C_1O)$, $T_{1,7}$ for $\pi^{-1}(C_2O)$, $T_{8,0}$ for
$\pi^{-1}(C_3O)$, $T_{1}$ for $\pi^{-1}(C_4O)$, $T_{2}$ for $\pi^{-1}(C_5O)$ and $T_{3}$ for $\pi^{-1}(C_6O)$.
Similarly, as in Example \ref{egh1}, we can show that
\begin{align*}
X_{\ref{egh13}}= & (V_2 \ast (T_{7,0} \cup T_{1,7})) \cup (V_3 \ast (T_{1,7} \cup T_{8,0})) \cup (V_4 \ast (T_{8,0} \cup T_{1})\\
  &\cup (V_5 \ast (T_{1} \cup T_{2})) \cup (V_6 \ast (T_{2} \cup T_{3})) \cup (V_1 \ast (T_{3} \cup T_{7,0}))
\end{align*}
is an equilibrium triangulation of the corresponding surface with $f_0(X_{\ref{egh13}})=7+6+3+6+6=28$.}
\end{eg}

\begin{eg}\label{egh14}
{\rm Let $(k,\ell)= (-3, -1)$, $(a,b)=(1,-2)$ and $(c,d)=(0,1)$. Then, $L(k, 1)= L(-3, 1) \cong L(3,1)$,
$~L(b, a) =L(-2, 1) \cong \mathbb{RP}^3$, $~L(d, c)=L(1,0) \cong S^3$, $~L(a, b) = L(1, -2) \cong S^3$,
$~L(c,d) = L(0, 1) \cong S^1 \times S^2$, $~L(\ell, 1) =L(-1,1) \cong S^3$, $~L(ck-d, 1)=L(-1, 1) \cong S^3$,
$~L(\ell k-1, \ell)= L(2, -1) \cong \mathbb{RP}^3$ and $L(a-b \ell, 1)=L(-1,1) \cong S^3$.

Consider triangulations $T_{7,0}$ for $\pi^{-1}(C_1O)$, $T_{1,7}$ for $\pi^{-1}(C_2O)$, $T_{8,0}$ for
$\pi^{-1}(C_3O)$, $T_{2}$ for $\pi^{-1}(C_4O)$, $T_{1}$ for $\pi^{-1}(C_5O)$ and $T_{3}$ for $\pi^{-1}(C_6O)$.
Similarly, as in Example \ref{egh1}, we can show that
\begin{align*}
X_{\ref{egh14}} = &(V_2 \ast (T_{7,0} \cup T_{1,7})) \cup (V_3 \ast (T_{1,7} \cup T_{8,0})) \cup (V_4 \ast (T_{8,0} \cup T_{2})\\
  &\cup (V_5 \ast (T_{2} \cup T_{1})) \cup (V_6 \ast (T_{1} \cup T_{3})) \cup (V_1 \ast (T_{3} \cup T_{7,0}))
\end{align*}
is an equilibrium triangulation of the corresponding projective surface $M$ with $f_0(X_{\ref{egh14}})=7+6+3+6+6=28$.}
\end{eg}

\begin{eg}\label{egh18}
{\rm Let $(k,\ell)= (-2, -1)$, $(a,b)=(1,-1)$ and $(c,d)=(-3, 4)$. Then, $L(k, 1)= L(-2, 1) \cong \mathbb{RP}^3$,
$~L(b, a) =L(-1, 1) \cong S^3$, $~L(d, c)=L(4,-3) \cong L(4,1)$, $~L(a, b) = L(1, -1) \cong S^3$,
$~L(c,d) = L(-3,4) \cong L(3,1)$, $~L(\ell, 1) =L(-1,1) \cong S^3$, $~L(ck-d, 1)=L(2, 1) \cong \mathbb{RP}^3$,
$~L(\ell k-1, \ell)= L(1, -1) \cong S^3$ and $L(a-b \ell, 1)=L(0,1) \cong S^1 \times S^2$.

Consider triangulations $T_{4,0}$ for $\pi^{-1}(C_1O)$, $T_{7,0}$ for $\pi^{-1}(C_2O)$, $T_{3}$ for
$\pi^{-1}(C_3O)$, $T_{1}$ for $\pi^{-1}(C_4O)$, $T_{8,0}$ for $\pi^{-1}(C_5O)$ and $T_{1,7}$ for $\pi^{-1}(C_6O)$.
Similarly, as in Example \ref{egh1}, we can show that
\begin{align*}
X_{\ref{egh18}} = & (V_2 \ast (T_{4,0} \cup T_{7,0})) \cup (V_3 \ast (T_{7,0} \cup T_{3})) \cup (V_4 \ast (T_{3} \cup T_{1})\\
  &\cup (V_5 \ast (T_{1} \cup T_{8,0})) \cup (V_6 \ast (T_{8,0} \cup T_{1,7})) \cup (V_1 \ast (T_{1,7} \cup T_{4,0}))
\end{align*}
is an equilibrium triangulation of the corresponding projective surface with $f_0(X_{\ref{egh18}})=7+10+6+6+3+6=38$.}
\end{eg}

\begin{eg}\label{egh23}
{\rm Let $(k,\ell)= (-2, -1)$, $(a,b)=(3,-5)$ and $(c,d)=(-1, 2)$. Then, $L(k, 1)= L(-2, 1) \cong \mathbb{RP}^3$,
$~L(b, a) =L(-5, 3) \cong L(5,2)$, $~L(d, c)=L(2,-1) \cong \mathbb{RP}^3$, $~L(a, b) = L(3, -5) \cong L(3,1)$,
$~L(c,d) = L(-1,2) \cong S^3$, $~L(\ell, 1) =L(-1,1) \cong S^3$, $~L(ck-d, 1)=L(0, 1) \cong S^1 \times S^2$,
$~L(k\ell-1,\ell)= L(1, -1) \cong S^3$ and $L(a-b\ell, 1)=L(-2,1) \cong \mathbb{RP}^3$.

Consider triangulations $T_{9,0}$ for $\pi^{-1}(C_1O)$, $T_{2}$ for $\pi^{-1}(C_2O)$, $T_{1}$ for
$\pi^{-1}(C_3O)$, $T_{4,0}$ for $\pi^{-1}(C_4O)$, $T_{1,7}$ for $\pi^{-1}(C_5O)$ and $T_{3}$ for $\pi^{-1}(C_6O)$.
Similarly, as in Example \ref{egh1}, we can show that
\begin{align*}
X_{\ref{egh23}}= & (V_2 \ast (T_{9,0} \cup T_{2})) \cup (V_3 \ast (T_{2} \cup T_{1})) \cup (V_4 \ast (T_{1} \cup T_{4,0})\\
  &\cup (V_5 \ast (T_{4,0} \cup T_{1,7})) \cup (V_6 \ast (T_{1,7} \cup T_{3})) \cup (V_1 \ast (T_{3} \cup T_{9,0}))
\end{align*}
is an equilibrium triangulation of the corresponding projective surface with $f_0(X_{\ref{egh23}})=7+6+10+3+6=32$.}
\end{eg}

\begin{eg}\label{egh31}
{\rm Let $(k,\ell)= (-1, -1)$, $(a,b)=(-2,3)$ and $(c,d)=(-3, 4)$. Then, $L(k, 1)= L(-1, 1) \cong S^3$,
$~L(b, a) =L(3, -2) \cong L(3,1)$, $~L(d, c)=L(4,-3) \cong L(4,1)$, $~L(a, b) = L(-2, 3) \cong \mathbb{RP}^3$,
$~L(c,d) = L(-3,4) \cong L(3,1)$, $~L(\ell, 1) =L(-1,1) \cong S^3$, $~L(ck-d, 1)=L(-1, 1) \cong S^3$,
$~L(k\ell-1, \ell)= L(0, -1) \cong S^1 \times S^2$ and $L(a-b\ell, 1)=L(1,1) \cong S^3$.

Consider triangulations $T_{8,0}$ for $\pi^{-1}(C_1O)$, $T_{2}$ for $\pi^{-1}(C_2O)$, $T_{1}$ for
$\pi^{-1}(C_3O)$, $T_{7,0}$ for $\pi^{-1}(C_4O)$, $T_{4,0}$ for $\pi^{-1}(C_5O)$ and $T_{1,7}$ for $\pi^{-1}(C_6O)$.
Similarly, as in Example \ref{egh1}, we can show that
\begin{align*}
X_{\ref{egh31}}= &(V_2 \ast (T_{8,0} \cup T_{2})) \cup (V_3 \ast (T_{2} \cup T_{1})) \cup (V_4 \ast (T_{1} \cup T_{7,0})\\
  &\cup (V_5 \ast (T_{7,0} \cup T_{4,0})) \cup (V_6 \ast (T_{4,0} \cup T_{1,7})) \cup (V_1 \ast (T_{1,7} \cup T_{8,0}))
\end{align*}
is an equilibrium triangulation of the corresponding projective surface with $f_0(X_{\ref{egh31}})=7+6+6+10+3+6=38$.}
\end{eg}

\begin{eg}\label{egh34}
{\rm Let $(k,\ell)= (1, -1)$, $(a,b)=(0,1)$ and $(c,d)=(-1, 2)$. Then, $L(k, 1)= L(1, 1) \cong S^3$,
$~L(b, a) =L(1, 0) \cong S^3$, $~L(d, c)=L(2,-1) \cong \mathbb{RP}^3$, $~L(a, b) = L(0, 1) \cong S^1 \times S^2$,
$~L(c,d) = L(-1,2) \cong S^3$, $~L(\ell, 1) =L(-1,1) \cong S^3$, $~L(ck-d, 1)=L(-3, 1) \cong L(3,1)$,
$~L(k\ell-1, \ell)= L(-2, -1) \cong \mathbb{RP}^3$ and $L(a-b\ell, 1)=L(1,1) \cong S^3$.

Consider triangulations $T_{3}$ for $\pi^{-1}(C_1O)$, $T_{1,7}$ for $\pi^{-1}(C_2O)$, $T_{7,0}$ for
$\pi^{-1}(C_3O)$, $T_{1}$ for $\pi^{-1}(C_4O)$, $T_{8,0}$ for $\pi^{-1}(C_5O)$ and $T_{2}$ for $\pi^{-1}(C_6O)$.
Similarly, as in Example \ref{egh1}, we can show that
\begin{align*}
X_{\ref{egh34}}= & (V_2 \ast (T_{3} \cup T_{1,7})) \cup (V_3 \ast (T_{1,7} \cup T_{7,0})) \cup (V_4 \ast (T_{7,0} \cup T_{1})\\
  & \cup (V_5 \ast (T_{1} \cup T_{8,0})) \cup (V_6 \ast (T_{8,0} \cup T_{2})) \cup (V_1 \ast (T_{2} \cup T_{3}))
\end{align*}
is an equilibrium triangulation of the corresponding projective surface with $f_0(X_{\ref{egh34}})=7+3+6+6+6=28$.}
\end{eg}

\begin{eg}\label{egh35}
{\rm Let $(k,\ell)= (1, -1)$, $(a,b)=(1,2)$ and $(c,d)=(0, 1)$. Then, $L(k, 1)= L(1, 1) \cong S^3$,
$~L(b, a) =L(2, 1) \cong \mathbb{RP}^3$, $~L(d, c)=L(1,0) \cong S^3$, $~L(a, b) = L(1, 2) \cong S^3$,
$~L(c,d) = L(0,1) \cong S^1 \times S^2$, $~L(\ell, 1) =L(-1,1) \cong S^3$, $~L(ck-d, 1)=L(-1, 1) \cong S^3$,
$~L(k\ell-1, \ell)= L(-2, -1) \cong \mathbb{RP}^3$ and $L(a-b\ell, 1)=L(3,1)$.

Consider triangulations $T_{3}$ for $\pi^{-1}(C_1O)$, $T_{1,7}$ for $\pi^{-1}(C_2O)$, $T_{2}$ for
$\pi^{-1}(C_3O)$, $T_{8,0}$ for $\pi^{-1}(C_4O)$, $T_{1}$ for $\pi^{-1}(C_5O)$ and $T_{7,0}$ for $\pi^{-1}(C_6O)$.
Similarly, as in Example \ref{egh1}, we can show that
\begin{align*}
X_{\ref{egh35}}= & (V_2 \ast (T_{3} \cup T_{1,7})) \cup (V_3 \ast (T_{1,7} \cup T_{2})) \cup (V_4 \ast (T_{2} \cup T_{8,0})\\
  &\cup (V_5 \ast (T_{8,0} \cup T_{1})) \cup (V_6 \ast (T_{1} \cup T_{7,0})) \cup (V_1 \ast (T_{7,0} \cup T_{3}))
\end{align*}
is an equilibrium triangulation of the corresponding projective surface with $f_0(X_{\ref{egh35}})=7+3+6+6+6=28$.}
\end{eg}

\begin{eg}\label{egh36}
{\rm Let $(k,\ell)= (2, -1)$, $(a,b)=(0,1)$ and $(c,d)=(-1, 2)$. Then, $L(k, 1)= L(2, 1) \cong \mathbb{RP}^3$,
$~L(b, a) =L(1, 0) \cong S^3$, $~L(d, c)=L(2,-1) \cong \mathbb{RP}^3$, $~L(a, b) = L(0, 1) \cong S^1 \times S^2$,
$~L(c,d) = L(-1,2) \cong S^3$, $~L(\ell, 1) =L(-1,1) \cong S^3$, $~L(ck-d, 1)=L(-4, 1) \cong L(4,1)$,
$~L(k\ell-1, \ell)= L(-3, -1) \cong L(3,1)$ and $L(a-b\ell, 1)=L(1,1) \cong S^3$.

Consider triangulations $T_{3}$ for $\pi^{-1}(C_1O)$, $T_{1,7}$ for $\pi^{-1}(C_2O)$, $T_{8,0}$ for
$\pi^{-1}(C_3O)$, $T_{1}$ for $\pi^{-1}(C_4O)$, $T_{4,0}$ for $\pi^{-1}(C_5O)$ and $T_{7,0}$ for $\pi^{-1}(C_6O)$.
Similarly, as in Example \ref{egh1}, we can show that
\begin{align*}
X_{\ref{egh36}}= &(V_2 \ast (T_{3} \cup T_{1,7})) \cup (V_3 \ast (T_{1,7} \cup T_{8,0})) \cup (V_4 \ast (T_{8,0} \cup T_{1})\\
  &\cup (V_5 \ast (T_{1} \cup T_{4,0})) \cup (V_6 \ast (T_{4,0} \cup T_{7,0})) \cup (V_1 \ast (T_{7,0} \cup T_{3}))
\end{align*}
is an equilibrium triangulation of the corresponding projective surface with $f_0(X_{\ref{egh36}})=7+3+6+10+6+6=38$.}
\end{eg}

\begin{eg}\label{egh37}
{\rm Let $(k,\ell)= (-2, 1)$, $(\ell,1)=(1,1$, $(a,b)=(1,-1)$ and $(c,d)=(1, 0)$. Then, $L(k, 1)= L(-2, 1) \cong \mathbb{RP}^3$,
$~L(b, a) =L(-1, 1) \cong S^3$, $~L(d, c)=L(0,1) \cong S^1 \times S^2$, $~L(a, b) = L(1, -1) \cong S^3$,
$~L(c,d) = L(1,0) \cong S^3$, $~L(\ell, 1) =L(1,1) \cong S^3$, $~L(ck-d, 1)=L(-2, 1) \cong \mathbb{RP}^3$,
$~L(k\ell-1, \ell)= L(-3, 1) \cong L(3,1)$ and $L(a-b\ell, 1)=L(2,1) \cong \mathbb{RP}^3$.

Consider triangulations $T_{2}$ for $\pi^{-1}(C_1O)$, $T_{1}$ for $\pi^{-1}(C_2O)$, $T_{7,0}$ for
$\pi^{-1}(C_3O)$, $T_{3}$ for $\pi^{-1}(C_4O)$, $T_{2,7}$ for $\pi^{-1}(C_5O)$ and $T_{8,0}$ for $\pi^{-1}(C_6O)$.
Similarly, as in Example \ref{egh1}, we can show that
\begin{align*}
X_{\ref{egh37}}= &(V_2 \ast (T_{2} \cup T_{1})) \cup (V_3 \ast (T_{1} \cup T_{7,0})) \cup (V_4 \ast (T_{7,0} \cup T_{3})\\
  &\cup (V_5 \ast (T_{3} \cup T_{2,7})) \cup (V_6 \ast (T_{2,7} \cup T_{8,0})) \cup (V_1 \ast (T_{8,0} \cup T_{2}))
\end{align*}
is an equilibrium triangulation of the corresponding projective surface with $f_0(X_{\ref{egh37}})=7+6+3+6+6=28$.}
\end{eg}

\begin{remark}
Since face preserving automorphism of $n$-gon is given by reflection and rotation, by Proposition
\ref{clema3}, we can see that no two nonsingular projective surfaces in the examples of Sections \ref{m4}, 
\ref{m5} and \ref{m6} are equivariantly homeomorphic. But there are equilibrium triangulations in these 
examples which are isomorphic. So, isomorphism of equilibrium triangulation does not guarantee equivariant
homeomorphism of quasitoric manifolds.
\end{remark}

{\bf Acknowledgements:} The first author is supported by DIICCSRTE, Australia and DST, India, under the
Australia-India Strategic Research Fund (project AISRF06660) and by the UGC Centre for Advanced Studies. 
This work was started when the second author was a CPDF of IISc. He 
thanks IISc, PIMS, Fields Institute and University of Regina for financial support.

\bibliographystyle{alpha}
\bibliography{bibliography.bib}

\end{document}